\documentclass[a4paper,leqno,10t]{article}

\usepackage[utf8]{inputenc} 
\usepackage[T1]{fontenc}

\usepackage{units}
\usepackage{amsmath}
\usepackage{amsfonts}
\usepackage{amssymb}
\usepackage{amsthm}
\usepackage{graphicx}
\usepackage{fancyhdr}
\usepackage{geometry}
\usepackage{hyperref}
\usepackage{mathrsfs}
\usepackage[nottoc,numbib]{tocbibind}
\usepackage{mathtools}
\usepackage{color}
\usepackage[normalem]{ulem}

\numberwithin{equation}{section}

\theoremstyle{definition} 

\newtheorem{ex}{\bfseries \upshape Example}[section]
\newtheorem{dfn}[ex]{\bfseries \upshape Definition}
\newtheorem{rem}[ex]{\bfseries \upshape Remark}

\theoremstyle{plain}
\newtheorem{prop}[ex]{\bfseries \upshape Proposition}
\newtheorem{lem}[ex]{\bfseries \upshape Lemma}
\newtheorem*{theo}{\bfseries \upshape Main Theorem}
\newtheorem{thm}[ex]{\bfseries \upshape Theorem}
\newtheorem*{que*}{\bfseries \upshape Question}
\newtheorem*{claim*}{\textit{Claim}}

\newenvironment{pr}{\begin{proof}[{\bf Proof}]}{\end{proof}}

\def\a{\alpha}
\def\b{\beta}

\def\R{\mathbb R}
\def\N{{\mathbb N}}

\def\Z{\mathbb Z}
\def\T{\mathbb T}

\def\({\biggl(}
\def\){\biggr)}
\def\<{\mathbf{\langle}}
\def\>{\mathbf{\rangle}}

\newcommand{\Meng}[2]{\left\{#1\mathrel{}\middle|\mathrel{}#2\right\}}
\newcommand{\abs}[1]{\left\lvert#1\right\rvert}

\begin{document}
\title{{ \bf A smooth zero-entropy diffeomorphism whose product with itself is loosely Bernoulli}}
\author{{\sc Marlies Gerber\thanks{Indiana University, Department of Mathematics, Bloomington, IN 47405, USA} and Philipp Kunde\thanks{University of Hamburg, Department of Mathematics, 20146 Hamburg, Germany}}}

\maketitle

\bigskip

\begin{abstract} 
Let $M$ be a smooth compact connected manifold of dimension $d\geq 2$, possibly with boundary, that admits a smooth effective $\T^2$-action $\mathcal{S}=\left\{S_{\a,\b}\right\}_{(\a,\b) \in \mathbb{T}^2}$ preserving a smooth volume $\nu$, and let $\mathcal{B}$ be the $C^{\infty}$ closure of $\left\{h \circ S_{\a,\b} \circ h^{-1} \;:\;h \in \text{Diff}^{\infty}\left(M,\nu\right), (\a,\b) \in \T^2\right\}$. 
We construct a $C^{\infty}$ diffeomorphism $T \in \mathcal{B}$ with topological entropy $0$  such that $T \times T$ is loosely Bernoulli. Moreover, we show that the set of such $T \in \mathcal{B}$ contains a dense $G_{\delta}$ subset of $\mathcal{B}$. The proofs are based on a two-dimensional version of the approximation-by-conjugation method. 
\end{abstract}

\insert\footins{\footnotesize - 
}

\section{Introduction}
An important question in ergodic theory (see, for example,  \cite[p.89]{OW}) dating back to the foundational paper \cite{Neu} of von Neumann is: 
\begin{que*} 
Are there smooth versions of the objects and concepts of abstract ergodic theory? 
\end{que*}
The objects that we consider are ergodic automorphisms of finite measure
spaces, more precisely, ergodic measure-preserving bijections from
an atom-free standard measure space of finite measure to itself. By a smooth version of an ergodic automorphism we mean a $C^{\infty}$ diffeomorphism of a compact manifold
preserving a $C^{\infty}$ measure equivalent to the volume element
that is (measure) isomorphic to that ergodic automorphism. The
only known restriction for obtaining a smooth version of such an automorphism
is due to A. G. Kushnirenko, who proved that the measure-theoretic  entropy must be finite.
However, there is a scarcity of general results on the smooth
realization problem. 

The pioneering work in the area of smooth realization is the paper
of D. Anosov and A. Katok \cite{AK}, who gave the first example of
a smooth version of a weakly mixing automorphism of the disk, using
what is now called the \emph{approximation-by-conjugation method} (sometimes
abbreviated as the\emph{ AbC method}), or the \emph{Anosov-Katok method}.
The related theory of periodic approximation was developed by Katok
and A. Stepin \cite{KS1}. These results have been applied in many
subsequent papers, e.g., \cite{FS}, \cite{FSW}, \cite{GK}, and are also important for the present paper. 

Our main result is the smooth realization of a zero-entropy automorphism
$T$ whose Cartesian product $T\times T$ is loosely Bernoulli, on
any compact manifold, possibly with boundary, that admits an effective
smooth $\mathbb{T}^{2}$-action. In the positive entropy case, the
smooth realization of such an automorphism $T$ already follows from the work
of M. Brin, J. Feldman, and Katok \cite{BFK}, who showed that any
compact manifold of dimension greater than one admits a smooth Bernoulli
diffeomorphism. In the present discussion, and throughout our paper, the entropy of an automorphism means the measure-theoretic entropy with respect to the given invariant measure, unless we specify otherwise. However, it follows from the work of S. Glasner and D. Maon \cite{GM} and H. Furstenberg and B. Weiss \cite{FuW} that our example $T$ also has topological entropy zero. 

The loosely Bernoulli property was introduced by Katok \cite{K2}
in the case of zero entropy, and, independently, by Feldman \cite{Fe} in the general case. This property is used to study orbit equivalence of
flows. Two ergodic measure-preserving automorphisms $T$ and $S$
are said to be \emph{Kakutani equivalent} if they are isomorphic to
cross-sections of the same ergodic measure-preserving flow (see Theorem 1.14 in \cite{ORW}). From Abramov's entropy formula,
it follows that two Kakutani equivalent automorphisms must both have
entropy zero, both have finite positive entropy, or both have infinite
entropy. It was a long-standing open problem whether these three possibilities
for entropy completely characterized Kakutani equivalence classes,
but Feldman \cite{Fe} showed that there are at least two non-Kakutani
equivalent ergodic transformations $T$ and $S$ of entropy zero,
and likewise for finite positive entropy and infinite entropy. D.
Ornstein, D. Rudolph, and Weiss \cite{ORW} generalized Feldman's
construction to obtain an uncountable family of pairwise non-Kakutani
equivalent ergodic automorphisms
in each of the three entropy classes. 

An ergodic automorphism $T$ that is Kakutani equivalent to an irrational
circle rotation (in the case of entropy zero) or a Bernoulli shift
(in case of non-zero entropy) is said to be \emph{loosely Bernoulli.
}Zero-entropy loosely Bernoulli automorphisms are also called \emph{standard.}
There is a metric on strings of symbols called the $\overline{f}$
metric (defined in Section \ref{section:crit}) that was used by Feldman to create a
Kakutani equivalence theory for loosely Bernoulli automorphisms that
parallels Ornstein's $\overline{d}$ metric and his isomorphism theory
for Bernoulli shifts. In case of zero entropy, the $\overline{f}$
metric gives a simple characterization of loosely Bernoulli automorphisms
\cite{KSa}. (See Theorem \ref{thm:Katok-Sataev} below.)

If $T$ is an ergodic automorphism such that $T\times T$ is loosely
Bernoulli, then $T$ must itself be loosely Bernoulli. However, the
converse is not true. M. Ratner showed that for any (non-identity)
transformation $T$ in the horocycle flow, $T$ is loosely Bernoulli
(\cite{R1}), but $T\times T$ is not loosely Bernoulli (\cite{R2}). Katok
gave the first example of a zero-entropy automorphism $T$ whose product
with itself is loosely Bernoulli (described in Section 2 of \cite{G}).
Moreover, Katok \cite[Proposition 3.5]{K} found a sufficient condition
in terms of periodic approximations of $T$ for $T\times T$
to be loosely Bernoulli (see Section \ref{subsection:theory}). We prove that $T\times T$ is loosely Bernoulli under a slightly weaker hypothesis on the periodic approximations of $T$ (see  Proposition \ref{prop:crit}). Automorphisms obtained from
periodic approximations in this way are weakly mixing but not mixing. However, there is also a zero-entropy mixing automorphism whose product with
itself is loosely Bernoulli (\cite{G}). All of the previously known
zero-entropy examples with $T\times T$ loosely Bernoulli are not
smooth. The zero-entropy example $T$ in the present paper with $T\times T$
loosely Bernoulli is a smooth weakly mixing diffeomorphism, but it
is not mixing. We also show that such examples are generic in the
following sense: Suppose  $\{S_{\alpha,\beta}\}_{(\alpha,\beta)\in\mathbb{T}^2}$ is a
smooth effective toral action preserving a smooth volume $\nu$ on a compact connected manifold $M$, and Diff$^{\infty}(M,\nu)$ is the set of $C^{\infty}$ diffeomorphisms of $M$ that preserve $\nu$. Then the set of those transformations $T$ in the $C^{\infty}$ closure
of $\{h\circ S_{\alpha,\beta}\circ h^{-1}:h\in\text{Diff}^{\infty}(M,\nu),(\alpha,\beta)\in\mathbb{T}^{2}\}$
such that $T\times T$ is loosely Bernoulli contains a dense $G_{\delta}$ subset in the $C^{\infty}$ topology.

In contrast to our results, A. Kanigowski and D. Wei (\cite{KW}) recently
proved that for a full measure set of parameters, the Cartesian product
of two Kochergin flows (\cite{Ko}) with different exponents is \emph{not}
loosely Bernoulli, even though the Kochergin flows themselves are
loosely Bernoulli. The case of a Cartesian product of a Kochergin
flow with itself is not covered by their proof, but Kanigowski and
Wei suggest that analogous results may also hold in that case. The
Kochergin flows are smooth, but have one degenerate fixed point. 

Another interesting recent result related to the smooth realization
problem was obtained by M. Foreman and Weiss in a series of papers \cite{FW1, FW2, FW3}, based partly
on earlier joint work with Rudolph \cite{FRW}. For $M=\mathbb{T}^{2},\mathbb{D}^{2}$
or the annulus, they found a way to code smooth diffeomorphisms on
$M$ constructed by the (untwisted) AbC method into some symbolic
spaces called \emph{uniform circular systems.} Thereby, they
were able to show that the measure isomorphism relation among pairs
of measure-preserving smooth diffeomorphisms on $M$ is not a Borel
set with respect to the $C^{\infty}$ topology. As discussed in \cite{FW1}, this can be interpreted as an "anti-classification" result for measure-preserving smooth diffeomorphisms on $M$.

\section{Preliminaries}
We begin by reviewing the concept of periodic approximation and several important related notions and results. We continue with a presentation of the general scheme of the AbC (approximation-by-conjugation) method. 
Finally we give definitions of distances on the space of $C^{\infty}$ functions that will be useful later in this paper, especially in Section \ref{section:conv}.

\subsection{Periodic approximation in ergodic theory} \label{subsection:theory}
The brief introduction to periodic approximation in this section is based on the more comprehensive presentation in \cite{K}. 

Let $\left(X, \mu\right)$ be  an atom-free standard probability space. A \emph{tower} $t$ of height $h(t)=h$ is an ordered sequence of disjoint measurable sets $t=\left\{c_1,...,c_h\right\}$ of $X$ having equal measure, which is denoted by $m\left(t\right)$. The sets $c_i$ are called the \emph{levels} of the tower; in particular, $c_1$ is the \emph{base}. Associated with a tower there is a \emph{cyclic permutation} $\sigma$ sending $c_1$ to $c_2$, $c_2$ to $c_3$,..., and $c_h$ to $c_1$. Using the notion of a tower we can give the next definition:
\begin{dfn}
A \emph{periodic process} is a collection of disjoint towers in the space $X$ and the associated cyclic permutations together with an equivalence relation among these towers identifying their bases. 
\end{dfn}
\begin{rem}\label{towers}
The towers of a periodic process, as defined in \cite{KS1, KS, K}, are actually required to cover the space. However, this is not necessary for Theorems \ref{thm:erg}, \ref{thm:wm}, and \ref{thm:square} stated at the end of this subsection. In fact, along a sequence of exhaustive periodic processes (to be defined below), the limit of the measure of the union of the towers is equal to one, and this slightly weaker condition can be substituted for having the towers cover. This follows from the results of \cite[Section 1.1]{{K}} for equivalent sequences of periodic processes.
\end{rem}
Furthermore, we introduce the notion of a \emph{partial partition} of a measure space $\left(X,\mu\right)$, which is a pairwise disjoint countable collection of measurable subsets of $X$.
\begin{dfn}
\begin{itemize}
	\item A sequence of partial partitions $\nu_n$ \emph{converges to the decomposition into points} if and only if for every measurable set $A$ and for every $n \in \mathbb{N}:=\{1,2,\dots \}$ there exists a measurable set $A_n$, which is a union of elements of $\nu_n$, such that $\lim_{n \rightarrow \infty} \mu \left( A \triangle A_n \right) = 0$. We often denote this by $\nu_n \rightarrow \varepsilon$.
	\item A partial partition $\nu$ is a \emph{refinement} of a partial partition $\eta$ if and only if for every $C \in \nu$ there exists a set $D \in \eta$ such that $C\subseteq D$. We write this as $\eta \leq \nu$.
\end{itemize}
\end{dfn}
Using the notion of a partition we can introduce the weak topology in the space of measure-preserving transformations on a Lebesgue space:
\begin{dfn}
\begin{enumerate}
	\item For two measure-preserving transformations $T,S$ and for a finite partition $\xi$ the \emph{weak distance} with respect to $\xi$ is defined by $d\left(\xi, T, S\right) \coloneqq \sum_{c \in \xi} \mu\left(T\left(c\right) \triangle S\left(c\right)\right)$.
	\item The base of neighbourhoods of $T$ in the \emph{weak topology} consists of the sets 
	\begin{equation*}
	W\left(T, \xi, \delta\right) = \left\{S\;:\;d\left(\xi,T,S\right)< \delta \right\},
	\end{equation*}
	where $\xi$ is a finite partition and $\delta$ is a positive number.
\end{enumerate}
\end{dfn}
There are two partial partitions associated with a periodic process: The partition $\xi$ into all levels of all towers and the partition $\eta$ consisting of the union of bases of towers in each equivalence class and their images under the iterates of $\sigma$, where we disregard a tower once we go beyond the height of that tower and we continue until the highest tower in the equivalence class has been exhausted. Obviously, we have $\eta \leq \xi$. We will identify a periodic process with the corresponding triple $\left( \xi, \eta, \sigma \right)$. 

A sequence $\left(\xi_n, \eta_n, \sigma_n\right)$ of periodic processes is called \emph{exhaustive} if $\eta_n \rightarrow \varepsilon$. Moreover, we will call a sequence of towers $t^{(n)}$ from the periodic process $\left(\xi_n, \eta_n, \sigma_n\right)$ \emph{substantial} if there exists $r>0$ such that $h\left(t^{(n)}\right) \cdot m\left(t^{(n)}\right) >r$ for every $n \in \mathbb{N}$.
\begin{dfn}
Let $T: \left(X, \mu\right)\rightarrow \left(X, \mu\right)$ be a measure-preserving transformation. An exhaustive sequence of periodic processes $\left(\xi_n,\eta_n, \sigma_n\right)$ forms a \emph{periodic approximation} of $T$ if 
\begin{equation*}
d\left(\xi_n,T,\sigma_n\right)=\sum_{c \in \xi_n} \mu\left(T\left(c\right) \triangle \sigma_n\left(c\right)\right) \rightarrow 0 \ \ \ \ \text{as } n\rightarrow \infty.
\end{equation*}
Given a sequence $g\left(n\right)$ of positive numbers we will say that the sequence of periodic processes  $(\xi_n,\eta_n,\sigma_n)$ is a periodic approximation of the transformation $T$ with \emph{speed} $g(n)$ if the sequence is exhaustive and  $d\left(\xi_n,T,\sigma_n \right)\le g\left(n\right)$ for all $n$. 
\end{dfn}

There are various types of periodic approximations. We introduce the most important ones:
\begin{dfn}
\begin{enumerate}
	\item A \emph{cyclic process} is a periodic process which consists of a single tower of height $h$. An approximation by an exhaustive sequence of cyclic processes with towers of height $h_n$  is called a \emph{cyclic approximation}. A \emph{good cyclic approximation} will refer to a cyclic approximation with speed $o\left(1/h\right)$, that is, speed $g(n)$, where $g(n)h_n\to 0$ as $n\to\infty.$ 
	\item  A \emph{type} $(h,h+1)$ \emph{process} is a periodic process which consists of two towers of height $h$ and $h+1$. An approximation by an exhaustive sequence of type $(h_n,h_{n+1})$ processes, where the towers of height $h_n$ and the towers of height $h_{n+1}$ are substantial, is called a \emph{type} $(h,h+1)$ \emph{approximation}. Equivalently for some $r>0$ we have $\mu\left(B^{(n)}_1\right) > \frac{r}{h_n}$ as well as $\mu\left(B^{(n)}_2\right)>\frac{r}{h_n+1}$ where the heights of the two towers $t^{(n)}_1$ and $t^{(n)}_2$ with base $B^{(n)}_1$ and $B^{(n)}_2$, respectively, are equal to $h_n$ and $h_n+1$. We will call the approximation of type $\left(h,h+1\right)$ with speed $o\left(1/h\right)$ \emph{good} and with speed $o\left(1/h^2\right)$ \emph{excellent}.
	\item An approximation of type $\left(h,h+1\right)$ will be called a \emph{linked approximation of type $\left(h,h+1\right)$} if the two towers involved in the approximation are equivalent. This insures that the sequence of partitions $\eta_n$ generated by the union of the bases of the two towers and the iterates of this set converges to the decomposition into points.
\end{enumerate}
\end{dfn}

From the different types of approximations various ergodic properties can be derived. For example in \cite[Corollary 2.1.]{KS1} the following theorem is proven.
\begin{thm} \label{thm:erg}
Let $T: \left(X, \mu\right)\rightarrow \left(X, \mu\right)$ be a measure-preserving transformation. If $T$ admits a good cyclic approximation, then $T$ is ergodic.
\end{thm}
Katok and Stepin deduced the following criterion for the weak mixing property:
\begin{thm}[\cite{KS}, Theorem 5.1.] \label{thm:wm}
Let $T: \left(X, \mu\right)\rightarrow \left(X, \mu\right)$ be a measure-preserving transformation. If $T$ admits a good linked approximation of type $(h,h+1)$ or if $T$ is ergodic and admits a good approximation of type $\left(h,h+1\right)$, then $T$ is weakly mixing.
\end{thm}

The concept of periodic approximation can also be used to show that $T \times T$ is loosely Bernoulli:
\begin{thm}[\cite{K}, Proposition 3.5.] \label{thm:square}
Let $T: \left(X, \mu\right)\rightarrow \left(X, \mu\right)$ be a measure-preserving transformation. If $T$ admits an excellent linked approximation of type $(h,h+1)$ or if $T$ is ergodic and admits an excellent approximation of type $(h,h+1)$, then $T \times T$ is standard (i.e., zero entropy loosely Bernoulli).
\end{thm}
Both forms of the hypothesis in Theorem \ref{thm:square} require an approximation of speed $o(1/h^2)$. The speed of the type $(h,h+1)$ approximation that we obtain for our example is only $O(1/h^2)$. However, in Proposition \ref{prop:crit} we weaken the hypothesis in Theorem \ref{thm:square}, replacing speed $o(1/h^2)$ by speed $O(1/h^2)$, while obtaining the same conclusion. Thus we obtain a diffeomorphism whose product with itself is zero entropy loosely Bernoulli.

\subsection{Approximation-by-conjugation method} \label{subsection:AbC}
As mentioned in the introduction, one of the most powerful tools of constructing smooth zero-entropy diffeomorphisms with prescribed ergodic or topological properties is the AbC method developed by Anosov and  Katok in \cite{AK}. In fact, on every smooth compact connected manifold $M$ of dimension $d\geq 2$ admitting a non-trivial circle action $\mathcal{S} = \left\{S_t\right\}_{t \in \mathbb{S}^1}$ preserving a smooth volume $\nu$ this method enables the construction of smooth diffeomorphisms with specific ergodic properties (e.\,g., weakly mixing ones in \cite[section 5]{AK}) or non-standard smooth realizations of measure-preserving systems (e.\,g., \cite[section 6]{AK} and \cite{FSW}). These diffeomorphisms are constructed as limits of conjugates $T_n = H^{-1}_n \circ S_{\alpha_{n+1}} \circ H_n$, where $\alpha_{n+1} = \frac{p_{n+1}}{q_{n+1}} = \alpha_n + \frac{1}{k_n \cdot l_n \cdot q^2_n} \in \mathbb{Q}$, $H_n =h_n \circ H_{n-1}$ and $h_n$ is a measure-preserving diffeomorphism satisfying $S_{\frac{1}{q_n}} \circ h_n = h_n \circ S_{\frac{1}{q_n}}$. In each step the conjugation map $h_n$ and the parameter $l_n \in \mathbb{N}$ are chosen such that the diffeomorphism $T_n$ imitates the desired property with a certain precision. Then the parameter $k_n \in \mathbb{N}$ is chosen large enough to guarantee closeness of $T_{n}$ to $T_{n-1}$ in the $C^{\infty}$-topology and so the convergence of the sequence $\left(T_n\right)_{n \in \mathbb{N}}$ to a limit diffeomorphism is guaranteed. See the very interesting survey article \cite{FK} for more details and other results of this method.

Recently, the second author used the AbC method to construct $C^{\infty}$-diffeomorphisms admitting a good linked approximation of type $(h,h+1)$ and a good cyclic approximation. Hereby, he proved that on any smooth compact connected manifold $M$ of dimension $d \geq 2$ admitting a non-trivial circle action $\mathcal{S} = \left\{S_t\right\}_{t \in \mathbb{S}^1}$ preserving a smooth volume $\nu$ the set of $C^{\infty}$-diffeomorphisms $T$, that have a homogeneous spectrum of multiplicity $2$ for $T\times T$ and a maximal spectral type disjoint with its convolutions, is residual (i.\,e., it contains a dense $G_{\delta}$-set) in $\mathcal{A}_{\alpha}\left(M\right)= \overline{\left\{h \circ S_{\alpha} \circ h^{-1} \ : h \in \text{Diff}^{\infty}\left(M, \nu \right)\right\}}^{C^{\infty}}$ for every Liouvillean number $\alpha$ (\cite{K-DC}). In \cite{BK} S. Banerjee and the second author obtained a real-analytic version of that result on any torus $\mathbb{T}^d$, $d \geq 2$. 

As mentioned above our methods use a linked approximation of type $(h,h+1)$ with speed of approximation $O\left(\nicefrac{1}{h^2}\right)$ in order to prove the loosely Bernoulli property of the Cartesian square. Although we have not been able to construct a diffeomorphism admitting such an approximation on manifolds admitting just an $\mathbb{S}^1$-action, we are able to do this on all smooth, compact and connected manifolds $M$, possibly with boundary, admitting a smooth effective torus action $\mathcal{S} = \left\{S_{\a,\b}\right\}_{(\a,\b) \in \mathbb{T}^2}$ (i.\,e., for each $(\a,\b) \in \T^2 \setminus \{(0,0)\}$ there is an $x \in M$ such that $S_{\a,\b}(x) \neq x$). Our construction can be thought of as a two-dimensional version of the AbC method. The assumption of the existence of a $\mathbb{T}^2$-action was already suggested by B. Fayad and Katok \cite[Section 7]{FK}, who stated that this assumption is likely to make it easier to produce approximations of type $(h,h+1)$ than if we just assume the existence of an $\mathbb{S}^1$-action.

Transformations obtained by the AbC method are usually uniformly rigid.
\begin{dfn}
Let $\left(X, \mathcal{B},\mu\right)$ be a probability space,  where $X$ is a compact metric space with metric $d$ and $\mathcal{B}$ is the Borel sigma-field. A measure-preserving homeomorphism $T: X\rightarrow X$ is called \emph{uniformly rigid} if there exists an increasing sequence  $\left(m_n\right)_{n \in \mathbb{N}}$ of natural numbers such that $\sup_{x\in  X}d\left(T^{m_n}(x),x\right)\to 0$ as $n\to\infty.$ 
\end{dfn}

In our setting we have $T^{q_nq^{\prime}_n}_n = \text{id}$ and obtain uniform rigidity of the limit diffeomorphism $T$ provided a sufficient closeness between $T_n$ and $T$ (see Remark \ref{rem:rigid} for details). The following result of Glasner and Maon, which depends on the earlier work of Furstenberg and Weiss \cite{FuW}, implies that our constructed diffeomorphism $T$ has topological entropy zero.
\begin{thm}[\cite{GM}, Proposition 6.3.] \label{thm:Glasner}
Let $X$ be a compact metric space and $T:X \rightarrow X$ be a uniformly rigid homeomorphism. Then the topological entropy of $T$ is zero.
\end{thm}

\subsection{Metrics on smooth diffeomorphisms}
Throughout this paper, $M$ will denote a smooth compact connected manifold, possibly with boundary, $\nu$ will be a smooth measure on $M$, $M_0$ will denote the particular manifold with corners, $\mathbb{T}^2\times [0,1]^{d-2},$ and $\mu$ will be Lebesgue measure on $M_0.$ As we explain in Section
\ref{subsection:First steps}, it suffices to prove the existence part of the Main Theorem for $(M_0,\mu).$

For $k\in\mathbb{N}$, we let $\tilde{d}_k$ denote a complete metric on $C^k(M),$ the set of $C^k$ functions from $M$ to itself, such that a sequence $(\varphi_n)_{n\in\mathbb{N}}$ in $C^k(M)$ converges to $\varphi\in C^k(M)$ if and only if $\varphi_n$ and its first $k$ derivatives converge uniformly to $\varphi$ and its first $k$ derivatives. We define
\begin{equation*}
	\tilde{d}_{\infty}\left(f,g\right) = \sum^{\infty}_{k=0} \frac{\tilde{d}_k\left(f,g\right)}{2^k \cdot \left(1 + \tilde{d}_k\left(f,g\right)\right)},
	\end{equation*}
for $f,g\in C^{\infty}(M).$

In the case of $M_0,$ we use the same formula for $\tilde{d}_{\infty}$ in terms of  the $\tilde{d}_k$'s, but we give an explicit definition of $\tilde{d}_k$ that will be convenient for our estimates in Section \ref{section:conv}. 

We let $\text{Diff}^{\infty}(M,\nu)$ and $\text{Diff}^{\infty}(M_0,\mu)$ denote the set of $\nu$-preserving $C^{\infty}$ diffeomorphisms of $M$ and the set of $\mu$-preserving $C^{\infty}$ diffeomorphisms of $M_0$, respectively.  Note that these sets are closed in the $C^{\infty}$ topology, i.\,e.,  the topology determined by the metric $\tilde{d}_{\infty}$. Since $C^{\infty}(M)$ and $C^{\infty}(M_0)$ are complete metric spaces with respect to the metric $\tilde{d}_{\infty}$, 
$\text{Diff}^{\infty}(M,\nu)$ and 
$\text{Diff}^{\infty}(M_0,\mu)$ are also complete metric spaces with respect to $\tilde{d}_{\infty}.$

As announced we discuss explicit metrics on Diff$^{\infty}(M_0)$, the space of smooth diffeomorphisms of the manifold $M_0 = \mathbb{T}^2 \times \left[0,1\right]^{d-2}$ to itself. For a diffeomorphism $f = \left(f_1,...,f_d\right): \mathbb{T}^2 \times \left[0,1\right]^{d-2} \rightarrow \mathbb{T}^2 \times \left[0,1\right]^{d-2}$, where $f_1,\dots,f_d$ are the coordinate functions, let $\tilde{f}=(\tilde{f}_1,\dots,\tilde{f}_d):\mathbb{R}^2\times \left[0,1\right]^{d-2} \rightarrow \mathbb{R}^2 \times \left[0,1\right]^{d-2}$ be a lift of $f$ to the universal cover, with coordinate functions $\tilde{f}_1,\dots,\tilde{f}_d$.
Then for $n,m\in \mathbb{Z}$, $\tilde{f}_i\left(\theta_1 +m, \theta_2+n, r_1,...,r_{d-2}\right) - \tilde{f}_i \left(\theta_1, \theta_2, r_1,...,r_{d-2}\right) \in \mathbb{Z}$ for $i=1,2$, and $\tilde{f}_i\left(\theta_1 +m, \theta_2+n, r_1,...,r_{d-2}\right) = \tilde{f}_i \left(\theta_1, \theta_2, r_1,...,r_{d-2}\right)$, for $i=3,\dots,d-2$.

For defining explicit metrics on Diff$^k\left(\mathbb{T}^2 \times \left[0,1\right]^{d-2}\right)$ and throughout the paper the following notation will be useful:
\begin{dfn}
\begin{enumerate}
	\item For a sufficiently differentiable function $f: \mathbb{R}^d \rightarrow \mathbb{R}$ and a multiindex $\vec{a} = \left(a_1,...,a_d\right) \in \mathbb{N}^d_0,$
\begin{equation*}
D_{\vec{a}}f := \frac{\partial^{\left|\vec{a}\right|}}{\partial x_1^{a_1}...\partial x_d^{a_d}} f,
\end{equation*}
where $\left|\vec{a}\right| = \sum^{d}_{i=1} a_i$ is the order of $\vec{a},$ and $\mathbb{N}_0:=\{0,1,2,\dots\}$.

\item For a continuous function $F: \left[0,1\right]^d \rightarrow \mathbb{R},$
\begin{equation*}
\left\|F\right\|_0 := \sup_{z \in \left[0,1\right]^d} \left|F\left(z\right)\right|.
\end{equation*}
\end{enumerate}
\end{dfn}

A diffeomorphism $f\in$ Diff$^k\left(\mathbb{T}^2 \times \left[0,1\right]^{d-2}\right)$ can be regarded as a map from $\left[0,1\right]^d$ to $\mathbb{R}^d$ by taking a lift of $f$ to the universal cover and then restricting the domain to $\left[0,1\right]^d$. In this way the expressions $\left\|f_i\right\|_0,$ as well as  $\left\|D_{\vec{a}}f_i\right\|_0$ for any multiindex $\vec{a}$ with $\left| \vec{a}\right|\leq k,$ can be understood for $f=\left(f_1,...,f_d\right) \in$ Diff$^k\left(\mathbb{T}^2 \times \left[0,1\right]^{d-2}\right)$. (Here $\left\|f_i\right\|_0$ is taken to be the minimum value of $\left\|F\right\|_0$, over all choices of lifts of $f$, where $F$ is the $i$th coordinate function of the lift.) Thus such a diffeomorphism can be regarded as a continuous map on the compact set $\left[0,1\right]^d$,  and every partial derivative of order at most $k$ can be extended continuously to the boundary.  Therefore the maxima that occur in the definition below are finite.
\begin{dfn}
\begin{enumerate}
	\item For $f,g \in$ Diff$^k\left(\mathbb{T}^2 \times \left[0,1\right]^{d-2}\right)$ with coordinate functions $f_i$ and $g_i$ respectively we define
\begin{equation*}
\tilde{d}_0\left(f,g\right) = \max_{i=1,..,d} \left\{ \inf_{p \in \mathbb{Z}} \left\| \left(f - g\right)_i + p\right\|_0\right\},
\end{equation*}
as well as
\begin{equation*}
\tilde{d}_k\left(f,g\right) = \max \left\{ \tilde{d}_0\left(f,g\right), \left\|D_{\vec{a}}\left(f-g\right)_i\right\|_0 \ : \ i=1,...,d \ , \ 1\leq \left|\vec{a}\right| \leq k \right\}.
\end{equation*}
  \item Using the definitions from 1. we define for $f,g \in$ Diff$^k\left(\mathbb{T}^2 \times \left[0,1\right]^{d-2}\right)$:
  \begin{equation*}
  d_k\left(f,g\right) = \max \left\{ \tilde{d}_k\left(f,g\right) \ , \ \tilde{d}_k\left(f^{-1},g^{-1}\right)\right\}.
  \end{equation*}
\end{enumerate}
\end{dfn}

Obviously $d_k$ describes a metric on Diff$^k\left(\mathbb{T}^2 \times \left[0,1\right]^{d-2}\right)$ measuring the distance between the diffeomorphisms as well as their inverses. As in the case of a general compact manifold the following definition connects to it:

\begin{dfn}
\begin{enumerate}
	\item A sequence of Diff$^{\infty}\left(\mathbb{T}^2 \times \left[0,1\right]^{d-2}\right)$-diffeomorphisms is called convergent in Diff$^{\infty}\left(\mathbb{T}^2 \times \left[0,1\right]^{d-2}\right)$ if it converges in Diff$^k\left(\mathbb{T}^2 \times \left[0,1\right]^{d-2}\right)$ for every $k \in \mathbb{N}_0$.
	\item On Diff$^{\infty}\left(\mathbb{T}^2 \times \left[0,1\right]^{d-2}\right)$ we declare the following metric
	\begin{equation*}
	d_{\infty}\left(f,g\right) = \sum^{\infty}_{k=0} \frac{d_k\left(f,g\right)}{2^k \cdot \left(1 + d_k\left(f,g\right)\right)}.
	\end{equation*}
\end{enumerate}
\end{dfn}

For $k\in\mathbb{N}\cup \{\infty\}$, the metrics $\tilde{d_k}$ and $d_k$ give the same topology on   Diff$^k\left(\mathbb{T}^2 \times \left[0,1\right]^{d-2}\right)$, but $\tilde{d_k}$ is complete, while  $d_k$ is not complete. We will refer to this topology as the $C^k$-topology, or the Diff$^k$-topology, to emphasize the fact that the topology is being induced on the subset Diff$^k(M_0)$ of $C^k(M_0).$ If we restrict to Diff$^k\left(M_0,\mu\right),$ then the metrics $d_k$ and $\tilde{d}_k$ are both complete.

Again considering diffeomorphisms on $M_0=\mathbb{T}^2 \times \left[0,1\right]^{d-2}$ as maps from $\left[0,1\right]^d$ to $\mathbb{R}^d$ we add the following  notation:
\begin{dfn}
Let $f \in$ Diff$^k(M_0)$ with coordinate functions $f_i$ be given. Then
\begin{equation*}
\left\| Df \right\|_0 := \max_{i,j \in \left\{1,...,d\right\}} \left\| D_j f_i \right\|_0
\end{equation*}
and
\begin{equation*}
||| f ||| _k := \max \left\{ \left\|D_{\vec{a}} f_i \right\|_0 , \left\|D_{\vec{a}} \left(f^{-1}_{i}\right)\right\|_0 \ : \ i = 1,...,d, \ \vec{a} \text{ multi-index with } 0\leq \left| \vec{a}\right| \leq k \right\}.
\end{equation*}
\end{dfn}

\section{Main result}
As announced in the introduction, our main result is the smooth realization of a zero-entropy automorphism $T$ whose Cartesian product $T \times T$ is loosely Bernoulli, on any compact manifold admitting an effective smooth $\mathbb{T}^2$-action. Moreover, we obtain a genericity statement in the space $\mathcal{B}:= \overline{\left\{h \circ S_{\a,\b} \circ h^{-1} \;:\;h \in \text{Diff}^{\infty}\left(M,\nu\right), (\a,\b) \in \T^2\right\}}^{C^{\infty}}$ introduced in \cite[Section 7.2]{FK}.

\begin{theo}
Let $d \geq 2$ and let $M$ be a $d$-dimensional smooth, compact and connected manifold, possibly with boundary, admitting a smooth effective torus action $\mathcal{S} = \left\{S_{\a,\b}\right\}_{(\a,\b) \in \mathbb{T}^2}$ preserving a smooth volume $\nu$. \\
Then there is a   $C^{\infty}$-diffeomorphism $T \in \mathcal{B}= \overline{\left\{h \circ S_{\a,\b} \circ h^{-1} \;:\;h \in \text{Diff}^{\infty}\left(M,\nu\right), (\a,\b) \in \T^2\right\}}^{C^{\infty}}$ of topological entropy $0$ such that $T \times T$ is loosely Bernoulli. Furthermore, the set of such diffeomorphisms is residual in $\mathcal{B}$ in the $C^{\infty}(M)$-topology.
\end{theo}

\subsection{First steps of the proof}
\label{subsection:First steps}
First of all, we show how constructions on $M_0= \mathbb{T}^2 \times \left[0,1\right]^{d-2}$ with Lebesgue measure $\mu$ can be transferred to a general smooth compact connected manifold $M$ with an effective $\T^2$-action $\mathcal{S} = \left\{S_{\a,\b}\right\}_{(\a,\b) \in \mathbb{T}^2}$. Furthermore, let $\mathcal{R} = \left\{R_{\a,\b} \right\}_{(\a,\b) \in \mathbb{T}^2}$ be the standard action of $\mathbb{T}^2$ on $\mathbb{T}^2 \times \left[0,1\right]^{d-2}$, where the map $R_{\a,\b}$ is given by $R_{\a,\b}\left(\theta_1,\theta_2, r_1,...,r_{d-2}\right) = \left(\theta_1 + \a,\theta_2 + \b, r_1,...,r_{d-2}\right)$. Hereby, we can formulate the following result (see \cite[Proposition 6.4.]{FK}):
\begin{prop} \label{prop:G}
Let $M$, $d$, $\nu$, and $\mathcal{S}$ be as in the Main Theorem. Let $B$ be the union of the boundary of $M$ and the set of points with a nontrivial isotropy group. There exists a continuous surjective map $G: \mathbb{T}^2 \times \left[0,1\right]^{d-2} \rightarrow M$ with the following properties:
\begin{enumerate}
	\item The restriction of $G$ to $\mathbb{T}^{2} \times \left(0,1\right)^{d-2}$ is a $C^{\infty}$-diffeomorphic embedding.
	\item $\nu\left(G\left(\partial\left(\mathbb{T}^{2} \times \left[0,1\right]^{d-2}\right)\right)\right) = 0$.
	\item $G\left(\partial\left(\mathbb{T}^{2} \times \left[0,1\right]^{d-2}\right)\right) \supseteq B$.
	\item $G_{*}\left(\mu\right) = \nu$.
	\item $\mathcal{S} \circ G = G \circ \mathcal{R}$.
\end{enumerate}
\end{prop}
By the same reasoning as in \cite[Section 2.2.]{FSW} this proposition allows us to carry a construction from $\left(\mathbb{T}^{2} \times \left[0,1\right]^{d-2}, \mathcal{R}, \mu\right)$ to the general case $\left(M, \mathcal{S}, \nu\right)$. For this purpose, it is useful to introduce classes of diffeomorphisms on $\mathbb{T}^{2} \times \left[0,1\right]^{d-2}$ whose jets of all orders of the difference to a rotation decay sufficiently fast near to the boundary. To make this precise we say that a sequence $\rho = \left( \rho_n\right)_{n \in \mathbb{N}_0}$ of continuous functions $\rho_n: \mathbb{T}^{2} \times \left[0,1\right]^{d-2} \to \mathbb{R}_{\geq 0}$ is admissible if every function $\rho_n$ is positive on the interior. Let
\begin{equation*}
\begin{split}
    C^{\infty}_{\rho} \left(\mathbb{T}^{2} \times \left[0,1\right]^{d-2}\right) = & \Bigg{\{}h \in C^{\infty}\left(\mathbb{T}^{2} \times \left[0,1\right]^{d-2}\right) \, : \, \forall n \in \mathbb{N}_0 \ \exists \delta_n>0 \ \forall \text{ multiindices $\vec{a} \in \mathbb{N}^d_0$ of order $n$: } \\ 
    & \quad \abs{D_{\vec{a}}h \left(\vec{x}\right)} < \rho_n\left( \vec{x}\right) \text{ for every $\vec{x}$ outside of } \T^2 \times \left[\delta_n, 1-\delta_n\right]^{d-2} \Bigg{\}}.
    \end{split}
\end{equation*}
Using the representation in coordinate functions $f=\left( f_1, \dots, f_d \right)$ for a diffeomorphism $f \in \text{Diff}^{\infty}\left(\mathbb{T}^{2} \times \left[0,1\right]^{d-2}\right)$ we define
\begin{equation*}
    \text{Diff}^{\infty}_{\rho, \a, \a^{\prime}} \left(\mathbb{T}^{2} \times \left[0,1\right]^{d-2}\right) = \Meng{ f \in \text{Diff}^{\infty}\left(\mathbb{T}^{2} \times \left[0,1\right]^{d-2}\right)}{f_i - \left[R_{\a, \a^{\prime}} \right]_i \in C^{\infty}_{\rho} \left(\mathbb{T}^{2} \times \left[0,1\right]^{d-2}\right)}
\end{equation*}
for $\left( \a, \a^{\prime} \right) \in \T^2$.

Suppose $T: \mathbb{T}^{2} \times \left[0,1\right]^{d-2} \rightarrow \mathbb{T}^{2} \times \left[0,1\right]^{d-2}$ is a diffeomorphism obtained by $T = \lim_{n \rightarrow \infty} T_n$ with $T_n =  H^{-1}_{n}\circ R_{\alpha_{n+1}, \alpha^{\prime}_{n+1}} \circ H_n $, where $T_n = R_{\alpha_{n+1}, \alpha^{\prime}_{n+1}}$ in a neighbourhood of the boundary. Furthermore, let $\left(\a,\a^{\prime}\right) = \lim_{n \to \infty}\left(\a_n, \a^{\prime}_n \right)$. Then we define a sequence of diffeomorphisms:
\begin{equation*}
\tilde{T}_n: M \rightarrow M \ \ \ \ \ \ \tilde{T}_n\left(x\right) = \begin{cases}
G \circ T_n \circ G^{-1}\left(x\right)& \textit{if $x \in G\left(\mathbb{T}^2 \times \left(0,1\right)^{d-2}\right)$} \\
S_{\alpha_{n+1}, \a^{\prime}_{n+1}}\left(x\right)& \textit{if $x \in G\left(\partial\left(\mathbb{T}^2 \times \left[0,1\right]^{d-2}\right)\right)$}
\end{cases}
\end{equation*}
By \cite[Proposition 1.1]{K1} there exists an admissible sequence $\rho$ such that the sequence $\left( \tilde{T}_n \right)_{n \in \mathbb{N}}$ is convergent in the $C^{\infty}$-topology to the diffeomorphism 
\begin{equation*}
\tilde{T}: M \rightarrow M \ \ \ \ \ \ \tilde{T}\left(x\right) = \begin{cases}
G \circ T \circ G^{-1}\left(x\right)& \textit{if $x \in G\left(\mathbb{T}^2 \times \left(0,1\right)^{d-2}\right)$} \\
S_{\alpha, \a^{\prime}}\left(x\right)& \textit{if $x \in G\left(\partial\left(\mathbb{T}^2 \times \left[0,1\right]^{d-2}\right)\right)$}
\end{cases}
\end{equation*}
provided $T \in \text{Diff}^{\infty}_{\rho, \a, \a^{\prime}} \left(\mathbb{T}^{2} \times \left[0,1\right]^{d-2}\right)$, i.\,e. the jets of all orders of the difference between $T$ and $R_{\alpha, \a^{\prime}}$ decay sufficiently fast near to the boundary.  Moreover, for arbitrary $\left(A,B\right) \in \T^2$ and $\varepsilon > 0$, we get $d_{\infty}\left( \tilde{T}, S_{A,B} \right)< \varepsilon$ if $T$ and $R_{A,B}$ are sufficiently close in the Diff$^{\infty}$-topology and the jets of all orders of the difference between $T$ and $R_{\alpha, \a^{\prime}}$ decay sufficiently fast near to the boundary. In Proposition \ref{prop:red} we will see that these conditions on sufficient flatness and closeness can be satisfied in the constructions of this paper.  We observe that $T$ and $\tilde{T}$ are metrically isomorphic.

Moreover, we will see in Remark \ref{rem:rigid} that due to $T^{q_{n+1}q^{\prime}_{n+1}}_n = \text{id}$, the diffeomorphism $T$ is uniformly rigid along the sequence $\left(q_n q^{\prime}_n \right)_{n \in \mathbb{N}}$, provided $T_n$ and $T$ are sufficiently close. This also yields $\tilde{T}^{q_{n+1}q^{\prime}_{n+1}}_n = \text{id}$ and the uniform rigidity of $\tilde{T}$ along the sequence $\left(q_n q^{\prime}_n \right)_{n \in \mathbb{N}}$. Then both transformations have topological entropy zero by Theorem \ref{thm:Glasner}. 

Hence, it is sufficient to execute the construction of $T$ in the Main Theorem in the case of $\left(\mathbb{T}^{2} \times \left[0,1\right]^{d-2}, \mathcal{R}, \mu\right)$. In this setting we will show the following statement:
\begin{prop} \label{prop:red}
    Let $\left(A, B \right) \in \T^2$ be arbitrary. There are sequences $\left(\alpha_n\right)_{n \in \mathbb{N}}, \left(\alpha^{\prime}_n\right)_{n \in \mathbb{N}}$ of rational numbers $\alpha_n = \frac{p_n}{q_n}, \alpha^{\prime}_n = \frac{p^{\prime}_n}{q^{\prime}_n}$ converging to $\alpha$ and $\a^{\prime}$ respectively and a sequence of measure-preserving smooth diffeomorphisms $h_n$, that coincide with the identity in a neighbourhood of the boundary and satisfy $h_n \circ R_{\frac{1}{q_n}, \frac{1}{q^{\prime}_n}} = R_{\frac{1}{q_n}, \frac{1}{q^{\prime}_n}} \circ h_n$, such that the diffeomorphisms $T_n=H^{-1}_n \circ R_{\alpha_{n+1}, \a^{\prime}_{n+1}} \circ H_n$, where $H_n= h_n \circ H_{n-1}$, converge in the $C^{\infty}$-topology to a limit $T = \lim_{n \rightarrow \infty} T_n$ with topological entropy $0$, whose Cartesian product with itself is loosely Bernoulli. Moreover, for every $\varepsilon>0$ and every admissible sequence $\rho$, the parameters in the construction can be chosen in such a way that $d_{\infty}\left(T, R_{A,B}\right) < \varepsilon$ and $T \in \text{Diff}^{\infty}_{\rho, \a, \a^{\prime}}\left( \T^2 \times [0,1]^{d-2} \right)$.
\end{prop}

We will prove the genericity statement in the Main Theorem in Section \ref{sec:gen}. 

\subsection{Outline of the proof}
Since we have not been able to construct a diffeomorphism $T$ admitting an excellent approximation of type $(h,h+1)$, we cannot apply the criterion in Theorem \ref{thm:square} to deduce the loosely Bernoulli property for $T \times T$. However, we are able to construct a diffeomorphism $T$ that admits a linked approximation of type $(h,h+1)$ 
and speed $O(1/h^2)$. According to Proposition \ref{prop:crit}, which may be of independent interest, this suffices for $T\times T$
to be loosely Bernoulli.

As indicated above, we will construct the smooth volume-preserving diffeomorphism $T$ by a variant of the \textit{approximation-by-conjugation} method. We let $\alpha_1=\frac{p_1}{q_1}$ and $\alpha_1'=\frac{p_1'}{q_1'}$ be rational numbers with $q_1$ and $q_1'$ relatively prime, and for $n\in\mathbb{N}$, we let 
\begin{align*}
&\alpha_{n+1} = \frac{p_{n+1}}{q_{n+1}}=\a_n + \frac{1}{q_{n+1}}, \\
&\alpha^{\prime}_{n+1} = \frac{p^{\prime}_{n+1}}{q^{\prime}_{n+1}}=\a^{\prime}_n+\frac{1}{\bar{q}^{\prime}_{n+1}}+D_n, 
\end{align*}
with $p_{n+1}$ and $q_{n+1}$ relatively prime, $p^{\prime}_{n+1}$ and $q^{\prime}_{n+1}$ relatively prime, and $\bar{q}^{\prime}_{n+1} = q_{n+1} + q_n q^{\prime}_n$. Smooth volume-preserving diffeomorphisms $h_0,h_1,\dots$ of $M_0=\mathbb{T}^2\times \left[0,1\right]^{d-2}$ are chosen so that $h_n \circ R_{\frac{1}{q_n}, \frac{1}{q^{\prime}_n}} = R_{\frac{1}{q_n}, \frac{1}{q^{\prime}_n}} \circ h_n$ for $n\in\mathbb{N}$. For the construction of our example $T$, we let $H_0=h_0=$ id. Then we define $H_n= h_n \circ H_{n-1}$ for $n\in\mathbb{N}.$ We will choose the parameters $q_{n+1}, D_n$ in Section \ref{section:conv} in such a way that we can guarantee that the sequence $\left(T_n\right)_{n \in \N},$  where $T_n=H^{-1}_n \circ R_{\alpha_{n+1}, \a^{\prime}_{n+1}} \circ H_n$, converges in the $C^{\infty}$-topology to a limit diffeomorphism $T$ and that this limit $T$ admits approximations by towers with the required speed.
Moreover, we will see in Remark \ref{rem:rigid} that due to $T^{q_{n+1}q^{\prime}_{n+1}}_n = \text{id}$ the diffeomorphism $T$ is uniformly rigid along the sequence $\left(q_n q^{\prime}_n \right)_{n \in \mathbb{N}}$. Then $T$ has topological entropy zero by Theorem \ref{thm:Glasner}. 

In Section \ref{section:conj} we construct the conjugation map $h_n$ mentioned before. Afterwards, we show that the constructed diffeomorphism satisfies the requirements of our criterion for the loosely Bernoulli property of $T \times T$. In particular, in Subsection \ref{subsection:towers}, we give the explicit definition of two sequences of towers whose height difference is one. Hereby, we show that $T$ admits a linked approximation of type $(h,h+1)$ and speed $O(1/h^2)$. Finally, we prove the genericity statement in Section \ref{sec:gen}.

\section{Criterion for $T \times T$ loosely Bernoulli} \label{section:crit}
The following proposition is a generalization of the result of Katok
stated in Theorem \ref{thm:square}. The conclusion remains the same, but we assume
that the speed of approximation is $O(1/h^{2}),$ instead of $o(1/h^{2}).$ 
\begin{prop}[Criterion for $T \times T$ loosely Bernoulli] 
\label{prop:crit} Let $T:(X,\mu)\to(X,\mu)$ be a measure-preserving
automorphism. If $T$ admits a linked approximation of type $(h,h+1)$
and speed $O(1/h^{2}),$ or if $T$ is ergodic and admits an approximation
of type (h,h+1) and speed $O(1/h^{2}),$ then $T$ has zero entropy
and $T\times T$ is loosely Bernoulli.\end{prop}
\begin{rem}
\label{rem:wkmixing}By Theorem \ref{thm:wm},  either version of the hypothesis of Proposition
\ref{prop:crit} implies that $T$ is weakly mixing. 

Suppose that a periodic process $(\xi,\eta,\sigma)$ forms a periodic
approximation of a measure-preserving automorphism $S:(X,\mu)\to(X,\mu)$ and $\tilde{X}:=\cup_{c\in \xi}c$, the union of the levels in the towers.
Then there is a realization of $\sigma$ as a measure-preserving automorphism
of $\tilde{X}$ such that $E:=\{x\in \tilde{X}:S(x)\ne\sigma(x)\}$ has $\mu(E)=(\nicefrac{1}{2})d(\xi,S,\sigma)$.
(See part 1.1 of {[}Ka03{]}.) In this section, we will always replace
the permutation $\sigma$ with a realization of $\sigma$ as a measure-preserving
automorphism of the union of the levels of the towers. \end{rem}
\begin{dfn}
If $t$ is a tower of height $h$ for a periodic process $(\xi,\eta,\sigma)$
that forms a periodic approximation of $S,$ then we choose a realization
of $\sigma$ as above, and we define a \emph{column }of $t$ to be
a set $\{p,\sigma(p),\dots\sigma^{h-1}(p)\}$ where $p$ is a point
in the base of $t.$
\end{dfn}
\begin{dfn}
Suppose $S:(X,\mu)\to(X,\mu)$ is a measure-preserving automorphism
and $\mathcal{P}=(P_{1},\dots,P_{q})$ is a finite measurable partition
of $X.$ If $b$ and $c$ are integers with $b\le c,$ then the $S$\emph{-$\mathcal{P}$
name} of a point $x\in X$ from time $b$ to time $c$ is the finite
sequence $(a_{b},a_{b+1},\dots,a_{c})$ where $S^{i}(x)\in P_{a_{i}}$
for $b\le i\le c.$ If $b=0$ and $c=n-1$ for some $n\in\mathbb{N},$
then $(a_{0},\dots,a_{n-1})$ is called the $S$\emph{-$\mathcal{P}$-$n$
name} of $x.$ If $n\in\mathbb{N},$ and $x,y\in X,$ then the $\overline{f}$
\emph{distance} between the $S$-$\mathcal{P}$-$n$ names of $x$
and $y$ is 
\[
\overline{f}_{S,\mathcal{P},n}(x,y)=1-(m/n),
\]
where $m=\sup\big\{ j:\text{there exist }0\le k_{1}<\cdots<k_{j}<n\text{ and }0\le\ell_{1}<\cdots<\ell_{j}<n\text{ such that }\ensuremath{S^{k_{i}}x}$ and $\ensuremath{S^{\ell_{i}}y}\text{ are in the same element of }\mbox{\ensuremath{\mathcal{P}}}$
for $i=1,\dots,j\big\}.$ That is, the $\overline{f}$ distance measures
the proportion of the symbols in the $S$-$\mathcal{P}$-$n$ names
of $x$ and $y$ that cannot be matched, even if allow some ``sliding''
along the names, but require that the order of the symbols in the
names be kept the same. 
\end{dfn}
The following theorem is due to Katok and Sataev (\cite{KSa}). 
\begin{thm} \label{thm:Katok-Sataev}
\label{KScondition}Let $S:(X,\mu)\to(X,\mu)$ be an ergodic measure-preserving
automorphism. Suppose that for every finite measurable partition $\mathcal{P}$
and every $\epsilon>0,$ there exist a number $\alpha=\alpha(\epsilon)>0,$
a sequence of natural numbers $N_{n}=N_{n}(\epsilon)\to\infty,$ and
sets $K_{n}=K_{n}(\epsilon)$ such that $\mu(K_{n})>\alpha$ and for
$x,y\in K_{n},$ $\overline{f}_{S,\mathcal{P},N_{n}}(x,y)<\epsilon.$
Then $S$ has zero entropy and is loosely Bernoulli.\end{thm}
\begin{rem}
In \cite{KSa}, ``for every finite measurable partition'' is replaced
by ``for a finite generating partition.'' But if we assume the hypothesis
of Theorem \ref{KScondition} for every finite measurable partition
$\mathcal{P},$ then it follows from \cite{KSa} that the entropy of
$S$ is $0$ on $\lor_{-\infty}^{\infty}S^{-i}\mathcal{P}$. Since
this is true for every finite measurable partition $\mathcal{P},$
the entropy of $S$ is zero on the original measure space. Thus by
Krieger's finite generator theorem \cite{Kr} (which holds in case
of finite entropy), there exists a finite generating partition. Therefore
Theorem \ref{KScondition} as stated here is equivalent to the version
in \cite{KSa}. 
\end{rem}
Proposition \ref{prop:crit} will follow from Remark \ref{rem:wkmixing},
Theorem \ref{KScondition}, and Lemma \ref{fbar match on G} below. 
\begin{lem}
\label{fbar match on G}Assume that $T:(X,\mu)\to(X,\mu)$ is a measure-preserving
automorphism and admits an approximation of type $(h,h+1)$ and speed
$O(1/h^{2}).$ Let $\mathcal{P}$ be a finite measurable partition of
$X\times X.$ Then there exist positive constants $\alpha_{0}$ and
$c_{2}$ such that for $0<\alpha<\alpha_{0},$ there exist a sequence
of numbers $N_{n}=N_{n}(\alpha)$ with $N_{n}\to\infty,$ and sets
$G_{n}=G_{n}(\alpha)\subset X\times X\times X\times X$ with $(\mu\times\mu\times\mu\times\mu)(G_{n})>\alpha$
such that 
\[
\overline{f}_{T\times T,\mathcal{P},N_{n}}((x,y),(\tilde{x},\tilde{y}))<c_{2}\sqrt{\alpha},\text{ for all }(x,y,\tilde{x},\tilde{y})\in G_{n}.
\]
\end{lem}
\begin{pr}
Let $T$ be as in the hypothesis of this lemma. Let $(\xi_{n},\eta_{n},\sigma_{n})$
be an exhaustive sequence of periodic processes associated to two
substantial sequences of towers $t_{1}^{(n)}$ and $t_{2}^{(n)}$
of heights $m_{n}-1$ and $m_{n},$ respectively, where $X=t_{1}^{(n)}\cup t_{2}^{(n)}$.
Let $r,c_{1}>0$ be constants such that $\mu(t_{i}^{(n)})>r,$ for
$i=1,2,$ $n\in\mathbb{N},$ and $d(\xi_{n},T,\sigma_{n})\le c_{1}/m_{n}^{2}$
for $n\in\mathbb{N}.$ The permutation $\sigma_{n}$ of the levels
of $t_{1}^{(n)}$ and $t_{2}^{(n)}$ can be realized as a measure-preserving
automorphism of $X$ such that $E_{n}:=\{x\in X:T(x)\ne\sigma_{n}(x)\}$
has measure $(\nicefrac{1}{2})d(\xi_{n},T,\sigma_{n}).$ Thus $\mu(E_{n})\le c_{1}/(2m_{n}^{2}).$ 

Let $\alpha_{0}=\min(r^{4}/4,(.05r)^{2}/c_{1}^{2}),$ and assume $0<\alpha<\alpha_{0}.$
Let 
\[
\widehat{E}_{n}=\cup_{i=-(m_{n}-1)}^{m_{n}(1+\left\lceil \sqrt{\alpha}m_{n}\right\rceil )}\sigma_{n}^{-i}E_{n}.
\]
Then $\mu(\widehat{E}_{n})\le\big(c_{1}/(2m_{n}^{2})\big)\big(m_{n}(2+\left\lceil \sqrt{\alpha}m_{n}\right\rceil \big)\le c_{1}\sqrt{\alpha},$
for $n$ sufficiently large. (How large $n$ has to be depends on
$\alpha.)$ Let $\widetilde{E}_{n}$ be the subset of $\widehat{E}_{n}$
consisting of those points $p$ such that the entire column of $p$
in $t_{1}^{(n)}$ or $t_{2}^{(n)}$ is in $\widehat{E}_{n}.$ If $x\in F_{n}:=X\setminus\widetilde{E}_{n},$
then there is some point $x'$ in the column of $x$ that is not in
$\widehat{E}_{n}.$ Then 
\[
\cup_{i=0}^{m_{n}\left\lceil \sqrt{\alpha}m_{n}\right\rceil }\{\sigma^{i}(x)\}\subset\cup_{i=-(m_{n}-1)}^{m_{n}(1+\left\lceil \sqrt{\alpha}m_{n}\right\rceil )}\{\sigma^{i}(x')\}\subset X\setminus E_{n}.
\]
(The fact that $\widetilde{E}_{n}$ consists of entire columns of
$t_{1}^{(n)}$ and $t_{2}^{(n)}$ will be convenient for estimating
the probability that condition (\ref{eq:closelevel}) below holds.)
We have $\mu(F_{n}\cap t_{i}^{(n)})\ge r-c_{1}\sqrt{\alpha},$ for
$i=1,2$ and $n$ sufficiently large.

Let $\widetilde{G}_{n}$ be the subset of $\big(t_{1}^{(n)}\times t_{2}^{(n)}\times t_{1}^{(n)}\times t_{2}^{(n)}\big)\cap\big(F_{n}\times F_{n}\times F_{n}\times F_{n}\big)$
consisting of points $(x,y,\tilde{x},\tilde{y})$ with the additional
property (\ref{eq:closelevel}) given below. Let $0\le s,\tilde{s}\le m_{n}-1$
be such that $\sigma_{n}^{s}(y)$ is in the first level of $t_{2}^{(n)}$
and $\sigma_{n}^{\tilde{s}}(\tilde{y})$ is in the first level of
$t_{2}^{(n)}.$ For $(x,y,\tilde{x},\tilde{y})$ to be in $\widetilde{G}_{n}$
we require that
\begin{equation}
\begin{array}{c}
(\text{level of }\sigma_{n}^{s}(x)\text{ in }t_{1}^{(n)})-(\text{level of }\sigma_{n}^{\tilde{s}}(\tilde{x})\text{ in }t_{1}^{(n)})\\
\in\big\{-\left\lceil \frac{\alpha(m_{n}-1)}{r^{4}}\right\rceil ,\dots,-1,0,1,\dots,\left\lceil \frac{\alpha(m_{n}-1)}{r^{4}}\right\rceil \big\}\text{ mod }(m_{n}-1).
\end{array}\label{eq:closelevel}
\end{equation}

Condition (\ref{eq:closelevel}) happens with probability at least $\big(2\left\lceil \alpha(m_{n}-1)/r^{4}\right\rceil +1\big)/\big(m_{n}-1\big)>2\alpha/r^{4}$
among the points in $\big(t_{1}^{(n)}\times t_{2}^{(n)}\times t_{1}^{(n)}\times t_{2}^{(n)}\big)\cap\big(F_{n}\times F_{n}\times F_{n}\times F_{n}\big).$
Thus $(\mu \times \mu \times \mu \times \mu)(\widetilde{G}_{n})\ge\big(2\alpha/r^{4}\big)\big(r-c_{1}\sqrt{\alpha}\big)^{4}>\big(2\alpha/r^{4}\big)\big(.95r)^{4}>(\nicefrac{3}{2})\alpha.$ 
Let $(x,y,\tilde{x},\tilde{y})\in\widetilde{G}_{n}.$ Suppose that
\[
\big(\text{level of }\sigma_{n}^{s}(x)\text{ in }t_{1}^{(n)}\big)=\big(\text{level of }\sigma_{n}^{\tilde{s}}(\tilde{x})\text{ in }t_{1}^{(n)}\big)+k,
\]
 where $k\in\{1,\dots,\left\lceil \alpha(m_{n}-1)/r^{4}\right\rceil \} .$
(The other cases are similar.) Note that $\sigma_{n}^{\tilde{s}+m_{n}}(\tilde{y})$
is in the first level of $t_{2}^{(n)}$ and $\sigma_{n}^{\tilde{s}+m_{n}}(\tilde{x})$
is one level higher in $t_{1}^{(n)}$ than $\sigma_{n}^{\tilde{s}}(\tilde{x}).$
Continuing with this pattern, $\sigma_{n}^{\tilde{s}+km_{n}}(\tilde{y})$
is in the first level of $t_{2}^{(n)}$ and $\sigma_{n}^{\tilde{s}+km_{n}}(\tilde{x})$
is in the same level of $t_{1}^{(n)}$ as $\sigma_{n}^{s}(x).$ Thus
the $\xi_{n}\times\xi_{n}$ names of $(x,y)$ from time $s$ to time
$s+m_{n}\left\lceil \sqrt{\alpha}m_{n}\right\rceil -km_{n}$ match
exactly with the $\xi_{n}\times\xi_{n}$ names of $(\tilde{x},\tilde{y})$
from time $\tilde{s}+km_{n}$ to time $\tilde{s}+m_{n}\left\lceil \sqrt{\alpha}m_{n}\right\rceil .$
Thus 
\[
\begin{array}{ccc}
\overline{f}_{\sigma_{n}\times\sigma_{n},\xi_{n}\times\xi_{n},m_{n}\left\lceil \sqrt{\alpha}m_{n}\right\rceil }((x,y),(\tilde{x},\tilde{y})) & < & \frac{(k+1)m_{n}}{m_{n}\left\lceil \sqrt{\alpha}m_{n}\right\rceil }\\
 & < & \frac{(\alpha m_{n}^{2}/r^{4})+2m_{n}}{\sqrt{\alpha}m_{n}^{2}}\\
 & \le & \tilde{c}_{2}\sqrt{\alpha},
\end{array}
\]
for $\tilde{c}_{2}=2/r^{4},$ provided $n$ is sufficiently large.
Since $(T\times T)^{i}(x,y)=(\sigma_{n}\times\sigma_{n})^{i}(x,y)$
for $(x,y)\in F_{n}\times F_{n}$ and $0\le i<m_{n}\left\lceil \sqrt{\alpha}m_{n}\right\rceil ,$
it follows that 

\[
\overline{f}_{T\times T,\xi_{n}\times\xi_{n},m_{n}\left\lceil \sqrt{\alpha}m_{n}\right\rceil }((x,y),(\tilde{x},\tilde{y}))<\tilde{c}_{2}\sqrt{\alpha}.
\]

Since every element of $\mathcal{P}$ can be approximated by a finite
union of products of measurable subsets of $X,$ there exist $B\subset X\times X$
and finite measurable partitions $\xi$ and $\widehat{\xi}$ of $X$
such that $(\mu\times\mu)(B)>1-(\alpha^{2}/8)$, and if $(x,y),(\tilde{x},\tilde{y})\in B$
and $(x,y)$ and $(\tilde{x},\tilde{y})$ are in the same element
of $\xi\times\widehat{\xi},$ then $(x,y)$ and $(\tilde{x},\tilde{y})$
are in the same element of $\mathcal{P}.$ By the exhaustivity assumption,
if $n$ is sufficiently large, then there is a set $A_{n}\subset X$
with $\mu(A_{n})>1-(\alpha^{2}/16)$ such that if two points $x,\tilde{x}\in A_{n}$
are in the same element of $\xi_{n},$ then they are in the same element
of $\xi$ and in the same element of $\widehat{\xi}.$ Note that $(\mu\times\mu\times\mu\times\mu)\big((X\times X\times X\times X)\setminus\left( (A_{n}\times A_{n}\times A_{n}\times A_{n})\cap(B\times B) \right) \big)<4(\alpha^{2}/16)+2(\alpha^{2}/8)=\alpha^{2}/2.$
Let $G_{n}$ consist of those points $(x,y,\tilde{x},\tilde{y})$
in $\widetilde{G}_{n}$ such that the proportion of indices $i$ among
those $i$ with $0\le i\le m_{n}\left\lceil \sqrt{\alpha}m_{n}\right\rceil $
such that $(T\times T\times T\times T)^{i}(x,y,\tilde{x},\tilde{y})\notin(A_{n}\times A_{n}\times A_{n}\times A_{n})\cap(B\times B)$
is less than $\alpha.$ Then $\alpha m_{n}\left\lceil \sqrt{\alpha}m_{n}\right\rceil (\mu\times \mu \times \mu \times \mu)(\widetilde{G}_{n}\setminus G_{n})<(\alpha^{2}/2)m_{n}\left\lceil \sqrt{\alpha}m_{n}\right\rceil .$
Thus $(\mu\times \mu \times \mu \times \mu)(\widetilde{G}_{n}\setminus G_{n}\big)<\alpha/2,$ and $(\mu \times \mu \times \mu \times \mu)(G_{n})>(\nicefrac{3}{2})\alpha-(\nicefrac{1}{2})\alpha=\alpha.$ 
For $(x,y,\tilde{x},\tilde{y})\in G_{n},$ we have 
\[
\overline{f}_{T\times T,\mathcal{P},m_{n}\left\lceil \sqrt{\alpha}m_{n}\right\rceil }((x,y),(\tilde{x},\tilde{y}))<\tilde{c}_{2}\sqrt{\alpha}+2\alpha<c_{2}\sqrt{\alpha},
\]
for some constant $c_{2}>0.$ We take $N_{n}=m_{n}\left\lceil \sqrt{\alpha}m_{n}\right\rceil .$
Then $N_{n}\to\infty$ as $n\to\infty.$ The $2\alpha$ term in the
above estimate comes from considering indices $k_{1},\dots,k_{j}$
and indices $\ell_{1},\dots,\ell_{j}$ that achieve the supremum in
the definition of $\overline{f}_{T\times T,\xi_{n}\times\xi_{n},m_{n}\left\lceil \sqrt{\alpha}m_{n}\right\rceil }((x,y),(\tilde{x},\tilde{y})),$
and removing at most $\left\lfloor \alpha N_{n}\right\rfloor $ of
the $k_{i}$'s (and the corresponding $\ell_{i}$'s) for which $(T\times T)^{k_{i}}(x,y)\notin(A_{n}\times A_{n})\cap B,$
and removing at most $\left\lfloor \alpha N_{n}\right\rfloor $ of
the $\ell_{i}$'s (and the corresponding $k_{i}$'s)  for which $(T\times T)^{\ell_{i}}(\tilde{x},\tilde{y})\notin(A_{n}\times A_{n})\cap B.$
\end{pr}

\begin{proof}
[Proof of Proposition \ref{prop:crit}]If $G_{n}\subset X\times X\times X\times X$
is as in Lemma \ref{fbar match on G}, then there exists $(x_{0},y_{0})\in X\times X$
such that $K_{n}:=\{(x,y):(x_{0},y_{0},x,y)\in G_{n}\}$ has $(\mu\times\mu)(K_{n})>\alpha,$
and for $(x,y),(\tilde{x},\tilde{y})\in K_{n},$
\begin{equation*}
\overline{f}_{T\times T,\mathcal{P},N_{n}}((x,y),(\tilde{x},\tilde{y})) \leq \overline{f}_{T\times T,\mathcal{P},N_{n}}((x,y),(x_{0},y_{0}))+\overline{f}_{T\times T,\mathcal{P},N_{n}}((x_{0},y_{0}),(\tilde{x},\tilde{y}))< 2c_{2}\sqrt{\alpha}.
\end{equation*}

Thus the hypothesis of Theorem \ref{KScondition} holds for $S=T\times T,$
if we choose $\alpha(\epsilon)=\min(\alpha_{0},\epsilon^{2}/(4c_{2}^{2})).$\end{proof}

\section{Combinatorics of the construction} \label{section:comb}
In this section we describe the underlying combinatorics of our construction. For this purpose, we recall our choice of rotation numbers
\begin{equation*}
\a_{n+1}= \a_n + \frac{1}{q_{n+1}}, \ \a^{\prime}_{n+1}= \a^{\prime}_n+\frac{1}{\bar{q}^{\prime}_{n+1}}+D_n
\end{equation*}
with $p_{n+1}$ and $q_{n+1}$ relatively prime, $p^{\prime}_{n+1}$ and $q^{\prime}_{n+1}$ relatively prime as well as the relation $q_{n+1}=\bar{q}^{\prime}_{n+1}-q_n q^{\prime}_n$. We will choose the numbers $q_{n+1}$ and $D_n$ in Section \ref{section:conv} in order to get convergence of $\left(T_n \right)_{n \in \N}$ in Diff$^{\infty}(M_0,\mu)$ and small ``error terms'' in the difference of the combinatorics of $T_n$ and the limit diffeomorphism $T$. Amongst others, we will impose the following conditions:
\begin{equation} \tag{A} \label{eq:div}
q_nq^{\prime}_{n} \text{ divides } q_{n+1},
\end{equation}
\begin{equation} \tag{B} \label{eq:prime}
q_n \text{ and } q^{\prime}_n \text{ are relatively prime},
\end{equation}
\begin{equation} \tag{C} \label{eq:relation}
\bar{q}^{\prime}_{n+1} = q_{n+1}+q_n q^{\prime}_n.
\end{equation}
Additionally, we define the number
\begin{equation*}
m_n = \frac{q_{n+1}}{q_n q^{\prime}_n}+1= \frac{\bar{q}^{\prime}_{n+1}}{q_n q^{\prime}_n}.
\end{equation*}
Since $q_{n+1}$ is a multiple of $q_n q^{\prime}_n$ by equation (\ref{eq:div}), $m_n \in \N$. In fact, $m_n-1$ will be the height of the first tower and $m_n$ the height of the second one. We put
\begin{equation*}
r_n \coloneqq (m_n-1) \cdot p_n \mod q_n \ \ \text{ and } \ \ r^{\prime}_n \coloneqq (m_n-1) \cdot p^{\prime}_n \mod q^{\prime}_n.
\end{equation*}
Calculating modulo $1$, we have
\begin{align}
& (m_n-1) \cdot \a_{n+1} = \frac{r_n}{q_n}+\frac{1}{q_nq^{\prime}_n}, \label{a1} \\
& (m_n - 1) \cdot \a^{\prime}_{n+1} = \frac{r^{\prime}_n}{q^{\prime}_n} + \frac{1}{q_nq^{\prime}_n}-\frac{1}{\bar{q}^{\prime}_{n+1}}+\left(m_n - 1\right) \cdot D_n = \frac{r^{\prime}_n}{q^{\prime}_n} + \frac{1}{q_nq^{\prime}_n} - \Delta_n, \label{a2} \\
& m_n \cdot \a_{n+1} = \frac{r_n + p_n}{q_n}+\frac{1}{q_n q^{\prime}_n} + \frac{1}{q_{n+1}}, \label{a3}\\
& m_n \cdot \a^{\prime}_{n+1} = \frac{r^{\prime}_n+ p^{\prime}_n}{q^{\prime}_n}+\frac{1}{q_n q^{\prime}_n}+m_n \cdot D_n \label{a4}
\end{align}
using the notation $\Delta_n =\frac{1}{\bar{q}^{\prime}_{n+1}}-\left(m_n - 1\right) \cdot D_n$. 

For the rest of this section, we will let $R_{\alpha,\beta}$ denote the automorphism of $\mathbb{T}^2$ given by $(\theta_1,\theta_2)\mapsto (\theta_1 + \alpha,\theta_2+\beta)$. That is, $R_{\alpha,\beta}$ is as in our previous definition in Subsection \ref{subsection:First steps}, except that we only apply the map on the $\mathbb{T}^2$ factor of $\mathbb{T}^2\times [0,1]^{d-2}$. In later sections, we will go back to the original definition of $R_{\alpha,\beta}$.

\begin{rem} \label{rem:outline towers} 
We observe that after $q_nq^{\prime}_n \cdot \left(m_n - 1\right)$ iterates of $R_{\alpha_{n+1}, \alpha^{\prime}_{n+1}}$ there will be a perfect recurrence in the $\theta_1$-coordinate but some small deviation in the $\theta_2$-coordinate. In order to obtain a very small approximation error in Lemma \ref{lem:error1} we aim for very good recurrence properties in the tower base. Hence, we will define the base of the first tower (of height $m_n - 1$) using the union of $q_nq^{\prime}_n$ sets of very small $\theta_1$-width $\frac{q_n q^{\prime}_n}{q_{n+1}}$ and $\theta_2$-height about $\frac{1}{2q_nq^{\prime}_n}$ in Subsection \ref{subsection:towers}. 
On the other hand, since $m_n \cdot D_n$ will be much smaller than $\frac{1}{q_{n+1}}$ by inequalities (\ref{eq:D}) and (\ref{eq:large}) in Lemma \ref{lem:conv}, we have almost perfect recurrence in the $\theta_2$-coordinate and some slightly larger deviation in the $\theta_1$-coordinate after $q_n q^{\prime}_n \cdot m_n$ iterates of $R_{\alpha_{n+1}, \alpha^{\prime}_{n+1}}$. Thus the base of the second tower (of height $m_n$) will be defined using the union of $q_nq^{\prime}_n$ sets of very small $\theta_2$-height about $\frac{q_n q^{\prime}_n}{\bar{q}^{\prime}_{n+1}}$ and $\theta_1$-width about $\frac{1}{2q_nq^{\prime}_n}$.
\end{rem}

According to equations (\ref{a1}) to (\ref{a4}), the rotation $R_{\frac{r_n}{q_n}+\frac{1}{q_nq^{\prime}_n}, \frac{r^{\prime}_n}{q^{\prime}_n}+\frac{1}{q_nq^{\prime}_n}}$ is an approximation of $R_{\alpha_{n+1},\alpha'_{n+1}}^{m_n-1}$ and  $R_{\frac{r_n+p_n}{q_n}+\frac{1}{q_nq^{\prime}_n}, \frac{r^{\prime}_n+p^{\prime}_n}{q^{\prime}_n}+\frac{1}{q_nq^{\prime}_n}}$ is an approximation of  $R_{\alpha_{n+1},\alpha'_{n+1}}^{m_n}$. With regard to these approximations, we make the following observations:
\begin{lem} \label{lem:combidisj}
The sets 
\begin{equation*}
R^k_{\frac{r_n}{q_n}+\frac{1}{q_nq^{\prime}_n}, \frac{r^{\prime}_n}{q^{\prime}_n}+\frac{1}{q_nq^{\prime}_n}}\left( \left(0, \frac{1}{q_nq^{\prime}_n}\right) \times \left(0, \frac{1}{q_nq^{\prime}_n}\right) \right) \mod \left(\frac{1}{q_n}, \frac{1}{q^{\prime}_n}\right), \ 0\leq k < q_nq^{\prime}_n
\end{equation*}
are disjoint. The same holds true for the sets
\begin{equation*}
R^k_{\frac{r_n+p_n}{q_n}+\frac{1}{q_nq^{\prime}_n}, \frac{r^{\prime}_n+p^{\prime}_n}{q^{\prime}_n}+\frac{1}{q_nq^{\prime}_n}}\left( \left(0, \frac{1}{q_nq^{\prime}_n}\right) \times \left(0, \frac{1}{q_nq^{\prime}_n}\right) \right) \mod \left(\frac{1}{q_n}, \frac{1}{q^{\prime}_n}\right), \ 0\leq k < q_nq^{\prime}_n.
\end{equation*}
\end{lem}

\begin{pr}
In order to deduce the first statement, we prove that for every pair $(i,j)$, where $0 \leq i < q^{\prime}_n$, $0 \leq j < q_n$, there is exactly one $0 \leq k=k(i,j) < q_n q^{\prime}_n$ such that
\begin{equation} \label{eq:modulo}
k \cdot \left( \frac{r_n}{q_n} + \frac{1}{q_nq^{\prime}_n}\right) \equiv \frac{i}{q_n q^{\prime}_n} \mod \frac{1}{q_n} \ \ \text{ and } \ \ k \cdot \left( \frac{r^{\prime}_n}{q^{\prime}_n} + \frac{1}{q_nq^{\prime}_n} \right) \equiv \frac{j}{q_n q^{\prime}_n} \mod \frac{1}{q^{\prime}_n}.
\end{equation}
These equations are equivalent to 
\[
k\equiv i\text{\ mod\ }q_{n}'\text{\ \ and\ \ }k\equiv j\text{\ mod\ }q_{n}.
\]
Since $q_{n}$ and $q_{n}'$ are relatively prime, the Chinese remainder
theorem implies that there is a unique solution $k$ with $0\le k<q_{n}q_{n}'$. 

The proof of the second statement in the lemma is similar. In fact,
the same $k=k(i,j)$ is also a solution to the equations in (\ref{eq:modulo}) if
we replace $r_{n}$ and $r_{n}'$ by $r_{n}+p_{n}$ and $r_{n}'+p_{n}'$, respectively.
\end{pr}

We consider the family of rectangles
\[
\begin{array}{cccc}
S_{i,j}^{(1)} & = & \left[\frac{i}{q_{n}q_{n}'},\frac{i+1}{q_{n}q_{n}'}\right]\times\left[\frac{j+(1/2)}{q_{n}q_{n}'},\frac{j+1}{q_{n}q_{n}'}\right], & \text{\ for\ }0\le i<q_{n},0\le j<q_{n}',\\
S_{i,j}^{(2)} & = & \left[\frac{i}{q_{n}q_{n}'},\frac{i+1}{q_{n}q_{n}'}\right]\times\left[\frac{j}{q_{n}q_{n}'},\frac{j+(1/2)}{q_{n}q_{n}'}\right], & \text{\ for\ }0\le i<q_{n},0\le j<q_{n}',
\end{array}
\]
in $\left[0,\frac{1}{q_{n}}\right]\times\left[0,\frac{1}{q_{n}'}\right]$.
Moreover, we consider the iterates 
\[
\begin{array}{cccc}
S_{k}^{(1)}: & = & R_{\frac{r_{n}}{q_{n}}+\frac{1}{q_{n}q_{n'}},\frac{r_{n}'}{q_{n}'}+\frac{1}{q_{n}q_{n}'}}^{k}(S_{0,0}^{(1)}), & \text{\ for\ }0\le k<q_{n}q_{n}',\\
S_{k}^{(2)}: & = & R_{\frac{r_{n}+p_{n}}{q_{n}}+\frac{1}{q_{n}q_{n'}},\frac{r_{n}'+p_{n}'}{q_{n}'}+\frac{1}{q_{n}q_{n}'}}^{k}(S_{0,0}^{(2)}), & \text{\ for\ }0\le k<q_{n}q_{n}'.
\end{array}
\]

\begin{figure}[hbtp] 
\begin{center}
\includegraphics[scale=0.8]{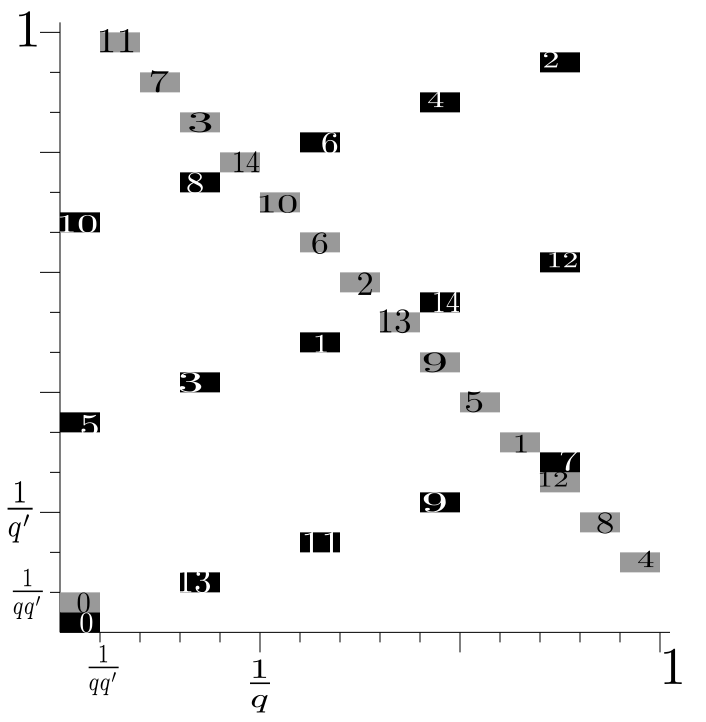}
\caption{Schematic representation of the combinatorics (with $p_{n}=2$, $p_{n}'=1$, $q_{n}=3$, $q_{n}'=5$, $r_{n}=2$, $r_{n}'=1$): The sets $S^{(1)}_{k}$ are colored in grey, the sets $S^{(2)}_{k}$ in black and both types
of sets are labelled by $k$. } \label{figure:combi}
\end{center}
\end{figure}

See also Figure \ref{figure:combi} for a sketch of these iterates. By the Lemma \ref{lem:combidisj}, for every pair $(i,j),$ where $0\le i<q_{n}',$ $0\le j<q_{n},$
there is exactly one $0\le k=k(i,j)<q_{n}q_{n}'$ such that $S_{k}^{(1)}$
has the same position as $S_{i,j}^{(1)}$ mod $\left(\frac{1}{q_{n}},\frac{1}{q_{n}'}\right).$
Similarly, for every pair $(i,j),$ where $0\le i<q_{n}',$ $0\le j<q_{n},$
the same $k=k(i,j)$ is the unique $k$ such that $S_{k}^{(2)}$ has
the same position as $S_{i,j}^{(2)}$ mod $\left(\frac{1}{q_{n}},\frac{1}{q_{n}'}\right).$
In particular, there are assignments $0\le a_{n,s}(i,j)<q_{n},$ and
$0\le a_{n,s}'(i,j)<q_{n}'$ , $s=1,2,$ such that
\[
\begin{array}{ccc}
S_{k(i,j)}^{(s)} & = & R_{\frac{a_{n,s}(i,j)}{q_{n}},\frac{a_{n,s}'(i,j)}{q_{n}'}}(S_{i,j}^{(s)}),\end{array}
\]
for all $0\le i<q_{n}'$, $0\le j<q_{n},s=1,2.$ \\
 
As mentioned in Remark \ref{rem:outline towers}, two towers will be defined in Subsection \ref{subsection:towers}. The $\mathbb{T}^2$ factor of the base of the first tower will consist of $q_nq^{\prime}_n$ sets of $\theta_1$-width $\frac{q_n q^{\prime}_n}{q_{n+1}}$ and $\theta_2$-height about $\frac{1}{2q_nq^{\prime}_n}$, one in each $S^{(1)}_{k(i,j)}$, $0\le i<q_{n}',$ $0\le j<q_{n}$ (the grey rectangles in Figure \ref{figure:combi}). On the other hand, the $\mathbb{T}^2$ factor of the base of the second tower will consist of $q_nq^{\prime}_n$ sets of $\theta_2$-height about $\frac{q_n q^{\prime}_n}{\bar{q}^{\prime}_{n+1}}$ and $\theta_1$-width about $\frac{1}{2q_nq^{\prime}_n}$, one in each $S^{(2)}_{k(i,j)}$, $0\le i<q_{n}',$ $0\le j<q_{n}$ (the black rectangles in Figure \ref{figure:combi}).

\section{Construction of conjugation maps} \label{section:conj}
We present step $n$ in our inductive process of construction. Accordingly, we assume that we have already defined the rational numbers $\alpha_1,...,\alpha_{n}, \alpha^{\prime}_1,..., \alpha^{\prime}_{n} \in \mathbb{S}^1$ and the conjugation map $H_{n-1}= h_{n-1} \circ ... \circ h_{0} \in \text{Diff}^{\infty}(M_0, \mu)$. We will construct the new conjugation map $h_n$ as a composition $h_n=h_{n,2} \circ h_{n,1}$. Both maps $h_{n,1}$ and $h_{n,2}$ will be constructed with a precise description of so-called ``good domains'' where we have good control on the action of the particular map. This will be important when we prove that the sequence of partitions consisting of the tower levels is generating and calculate the speed of approximation. 

We will make use of Lemma \ref{lem:Zauberlemma} below, which is essentially the same as \cite[Lemma 1.1]{AK}. First we need the following definition.
\begin{dfn}
Let $B_r^d=\{x=(x_1,\dots,x_d)\in \mathbb{R}^d:x_1^2+\cdots+x_d^2\le r^2\}.$ A $d$-\emph{cell} in a manifold $X$ is a set $F\subset X$ such that for some $\epsilon>0$, there exists a one-to-one non-singular $C^{\infty}$ map $\phi:\,$int($B_{1+\epsilon}^d)\to X$ with $\phi(B_1^d)=F.$
\end{dfn}
\begin{lem} \label{lem:Zauberlemma}
Let $d\ge 2$ and let $X$ be a connected smooth $d$-dimensional orientable manifold, possibly with boundary, with a smooth volume $\omega.$ Let $F_i$ and $G_i$, $i=1,\dots, k$, be two systems of $d$-cells in $X$  such that $F_i \cap F_j = \emptyset$ and $G_i \cap G_j = \emptyset$ for $i\neq j$. Moreover, for $i=1, \dots, k$, suppose that $\Theta_i$ is an orientation-preserving volume-preserving $C^{\infty}$ diffeomorphism from $F_i$ onto $G_i$ extendable to a diffeomorphism between open neighborhoods of $F_i$ and $G_i$. Then there exists an orientation-preserving diffeomorphism $\Theta \in \text{Diff}^{\infty}\left(X,\omega\right)$ which coincides with the identity outside of some compact set  $N\subset X\setminus \partial X$ and satisfies $\Theta|_{F_i} =\Theta_i$ for $i=1,\dots,k$.
\end{lem}

\begin{pr} 
We give an outline of the proof that is in \cite{AK}. By \cite[Theorem B]{Pa} there is an orientation-preserving $C^{\infty}$ diffeomorphism $\tilde{\Theta}$ on $X$ coinciding with $\Theta_i$ on $F_i$. In fact, by \cite[Corollary 1]{Pa} $\tilde{\Theta}$ can be taken to be the identity outside a compact set $N\subset X\setminus \partial X$.  Then we apply ``Moser's trick'' (\cite{Mo}) to obtain a volume-preserving $C^{\infty}$ diffeomorphism $\Theta$ on $X$ which coincides with $\tilde{\Theta}$ on the sets $F_i$, $i= 1, \dots, k$, and $X \setminus N$.
\end{pr}

In the following constructions we choose
\begin{equation} \label{eq:eps}
\varepsilon_n = \frac{2}{nq_n q^{\prime}_n}.
\end{equation}

\subsection{The conjugation map $h_{n,1}$} \label{subsection:h1}
Our goal is to find a smooth volume-preserving diffeomorphism $h_{n,1}$ from $\mathbb{T}^2\times [0,1]^{d-2}$ to itself that maps $S_{i, j}^{(s)}\times [0,1]^{d-2}$ to $S^{(s)}_{k(i,j)}\times [0,1]^{d-2}$ on large parts of these domains and which is  $\left(\frac{1}{q_n}, \frac{1}{q^{\prime}_n} \right)$-equivariant (that is, it commutes with $R_{\frac{1}{q_n},\frac{1}{q_n'}}\times \text{Id}$). Additionally, we divide each  $S_{i,j}^{(s)} \times [0,1]^{d-2}$  into subcuboids of the type $\left[ \frac{i_1}{n^{d-2} \cdot \left(q_n q^{\prime}_n\right)^d}, \frac{i_1 + 1}{n^{d-2} \cdot \left(q_n q^{\prime}_n\right)^d}\right] \times \left[\frac{i_2}{2q_nq^{\prime}_n},\frac{i_2 + 1}{2q_nq^{\prime}_n} \right] \times [0,1]^{d-2}$, where $i_1, i_2 \in \mathbb{Z}$, and we require the images under $h^{-1}_{n,1}$ of all but small parts of each of these subcuboids to have diameter of order $\frac{1}{q_n q^{\prime}_n}$ (see Lemma \ref{lem:prop h1} and Remark \ref{rem:small diameter} for the precise statements). Later, we will exploit this property in proving that the sequence of partitions consisting of the tower levels is generating (see Lemmas \ref{lem:xi} and \ref{lem:eta}). 

As in Section \ref{section:comb}, $S_{i,j}^{(s)}$ is a closed rectangle with $\theta_{1}$-length
$\frac{1}{q_{n}q_{n}'}$ and $\theta_{2}$-length $\frac{1}{2q_{n}q_{n}'}.$ Define $\widetilde{\varepsilon}_n:=\varepsilon_n(nq_nq_n')^{2-d}.$
Let $S_{i,j,\varepsilon_{n}}^{(s)}$ be a closed rectangle with the
same center as $S_{i,j}^{(s)}$ and with $\theta_{1}$-length $\frac{1-(\widetilde{\varepsilon}_{n}/2)}{q_{n}q_{n}'}$
and $\theta_{2}$-length $\frac{1-\varepsilon_{n}}{2q_{n}q_{n}'}.$ Let
$S_{k(i,j),\varepsilon_{n}}^{(1)}$ and $S_{k(i,j),\varepsilon_{n}}^{(2)}$
be obtained from $S_{k(i,j)}^{(1)}$ and $S_{k(i,j)}^{(2)}$ in the
same way. More precisely, we let 
\[
\begin{array}{ccc}
S_{i,j,\varepsilon_{n}}^{(s)} & := & \left[\frac{i+(\widetilde{\varepsilon}_{n}/4)}{q_{n}q_{n}'},\frac{i+1-(\widetilde{\varepsilon}_{n}/4)}{q_{n}q_{n}'}\right]\times\left[\frac{j+(1/2)(2-s)+(\varepsilon_{n}/4)}{q_{n}q_{n}'},\frac{j+(1/2)(3-s)-(\varepsilon_{n}/4)}{q_{n}q_{n}'}\right],\\
S_{k(i,j),\varepsilon_{n}}^{(s)} & := & R_{\frac{a_{n,s}(i,j)}{q_{n}},\frac{a_{n,s}'(i,j)}{q_{n}'}}(S_{i,j,\varepsilon_{n}}^{(s)})
\end{array}
\]
 for $0\le i<q'_{n},0\le j<q_{n},s=1,2.$ 
\begin{lem} \label{lem:translationlemma}
There exists a measure-preserving $C^{\infty}$ diffeomorphism $\varphi_{n}:\mathbb{T}^{2}\to\mathbb{T}^{2}$
that is $\left(\frac{1}{q_{n}},\frac{1}{q_{n}'}\right)$-equivariant
and $\varphi_{n}(S_{i,j,\varepsilon_{n}}^{(s)})=S_{k(i,j),\varepsilon_{n}}^{(s)}$
for all $0\le i<q_{n}'$, $0\le j<q_{n},$ $s=1,2$. Moreover the
restriction of $\varphi_{n}$ to each $S_{i,j,\varepsilon_{n}}^{(s)}$
is a translation. \end{lem}
\begin{proof}
The lemma will follow by applying the claim below, as well as an analogous statement for a translation by $\left(0, \frac{1}{q^{\prime}_n} \right)$, finitely many times and composing the diffeomorphisms that are obtained.

Fix a choice of $i_{0},j_{0},s_{0}.$ 
\begin{claim*}
There exists a measure-preserving $\left(\frac{1}{q_{n}},\frac{1}{q_{n}'}\right)$-equivariant
$C^{\infty}$ diffeomorphism $\phi:\mathbb{T}^{2}\to\mathbb{T}^{2}$
such that $\phi|S_{i,j,\varepsilon_{n}}^{(s)}=\text{Id}|S_{i,j,\varepsilon_{n}}^{(s)}$
for $(i,j,s)\ne(i_{0},j_{0},s_{0})$ and $\phi|S_{i_{0},j_{0},\varepsilon_{n}}^{(s_{0})}$
is a translation by $\left(\frac{1}{q_{n}},0\right).$ 
\end{claim*}

\noindent\emph{Proof of the Claim.} 
Let $\phi_{1}:\left[\frac{i_{0}}{q_{n}q_{n}'},\frac{1}{q_{n}}+\frac{i_{0}}{q_{n}q_{n}'}\right]\times\left[0,\frac{1}{q_{n}'}\right]\to\left[\frac{i_{0}}{q_{n}q_{n}'},\frac{1}{q_{n}}+\frac{i_{0}}{q_{n}q_{n}'}\right]\times\left[0,\frac{1}{q_{n}'}\right]$
be the measure-preserving diffeomorphism obtained from Lemma \ref{lem:Zauberlemma}
such that $\phi_{1}$ is the identity on $S_{i,j,\varepsilon_{n}}^{(s)}+\left(\frac{i_{0}}{q_{n}q_{n}'},0\right)$
if $(j,s)\ne(j_{0},s_{0})$ and $\phi_{1}:S_{i,j_{0},\varepsilon_{n}}^{(s)}+\left(\frac{i_{0}}{q_{n}q_{n}'},0\right)\to S_{i-1,j_{0},\varepsilon_{n}}^{(s)}+\left(\frac{i_{0}}{q_{n}q_{n}'},0\right)$
is a translation, where $i-1$ is taken mod $q'_{n}$ and chosen so
that $0\le i-1<q'_{n.}$ The diffeomorphism $\phi_{1}$ is chosen
to agree with the identity map in a neighborhood of the boundary of
$\left[\frac{i_{0}}{q_{n}q_{n}'},\frac{1}{q_{n}}+\frac{i_{0}}{q_{n}q_{n}'}\right]\times\left[0,\frac{1}{q_{n}'}\right].$
We extend $\phi_{1}$ to a $\left(\frac{1}{q_{n}},\frac{1}{q_{n}'}\right)$-equivariant
diffeomorphism of $\mathbb{T}^{2}$ to itself by defining $\phi_{1}(\theta_{1}+\frac{\lambda}{q_{n}},\theta_{2}+\frac{\lambda'}{q_{n}'})=\phi_{1}(\theta_{1},\theta_{2})+\left(\frac{\lambda}{q_{n}},\frac{\lambda'}{q_{n}'}\right)$,
for $0\le\lambda<q_{n},$ $0\le\lambda'<q_{n}'.$ Then we choose a
$C^{\infty}$ function $f:[0,\frac{1}{q_{n}'}]\to\mathbb{T}^{1}$
so that $f$ is zero on the complement of $\pi_{2}(S_{i_{0},j_{0}}^{(s)})$
and $f(\theta_{2})=\frac{1}{q_{n}q_{n}'}$ for $\theta\in\pi_{2}(S_{i_{0},j_{0},\varepsilon_{n}}^{(s)}).$
Here $\pi_{2}$ denotes projection onto the second coordinate. Extend
$f$ so that $f:\mathbb{T}^{1}\to\mathbb{T}^{1}$ and $f$ is periodic
with period $\frac{1}{q_{n}'}.$ Let $\phi_{2}(\theta_{1},\theta_{2})=(\theta_{1}+f(\theta_{2}),\theta_{2})$
on $\mathbb{T}^{2}.$ Then $\phi:=\phi_{2}\circ\phi_{1}$ satisfies
the Claim. 
\end{proof}
If $d=2,$ we define $h_{n,1}:=\varphi_{n}$ as in Lemma \ref{lem:translationlemma}. For
the rest of Section \ref{subsection:h1}, we assume $d>2.$

Divide each $\pi_{1}(S_{i,j}^{(s)})$ into $(nq_{n}q_{n}')^{d-2}$
equal intervals of $\theta_{1}$-length $\frac{1}{n^{d-2}(q_{n}q_{n}')^{d-1}}$,
and take the products of these intervals with $\pi_{2}(S_{i,j}^{(s)})\times\Pi_{r=1}^{d-2}[0,1].$
Denote the cuboids obtained in this way by $S_{i,j,\ell}^{(s)},$
$0\le\ell<(nq_{n}q_{n}')^{d-2}.$ Let $S_{i,j,\ell,\varepsilon_{n}}^{(s)}$
be the cuboid with the same center as $S_{i,j,\ell}^{(s)}$ with $\theta_{1}$-length
$\frac{1-\varepsilon_{n}}{n^{d-2}(q_{n}q_{n}')^{d-1}},$ $\theta_{2}$-length
$\frac{1-2\varepsilon_{n}}{2q_{n}q_{n}'}$, and $r_{1},\dots,r_{d-2}$-lengths
$1-2\varepsilon_{n}.$ Let $S^{(s)}_{k(i,j),\ell,\varepsilon_n}$ be obtained from $S^{(s)}_{k(i,j),\varepsilon_n}$ in the same way. More precisely, 
\[
\begin{array}{cccccc}
S_{i,j,\ell,\varepsilon_{n}}^{(s)}: & = & \left[\frac{i}{q_{n}q_{n}'}+\frac{\ell+(\varepsilon_{n}/2)}{n^{d-2}(q_{n}q_{n}')^{d-1}},\frac{i}{q_{n}q_{n}'}+\frac{\ell+1-(\varepsilon_{n}/2)}{n^{d-2}(q_{n}q_{n}')^{d-1}}\right]\ \ \ \ \ \ \ \ \ \ \ \ \ \ \ \ \ \ \ \ \ \\
 &  & \times\left[\frac{j+(1/2)(2-s)+(\varepsilon_{n}/2)}{q_{n}q_{n}'},\frac{j+(1/2)(3-s)-(\varepsilon_{n}/2)}{q_{n}q_{n}'}\right]\times\prod_{r=1}^{d-2}[\varepsilon_{n},1-\varepsilon_{n}],\\
S_{k(i,j),\ell,\varepsilon_{n}}^{(s)}: & = & R_{\frac{a_{n,s}(i,j)}{q_{n}},\frac{a_{n,s}'(i,j)}{q_{n}'}}(S_{i,j,\ell,\varepsilon_{n}}^{(s)}),
\end{array}
\]
for $0\le i<q'_{n},0\le j<q_{n},s=1,2.$
 
Note that $S_{i,j,\ell,\varepsilon_{n}}^{(s)}\subset\text{int }\left(S_{i,j,\varepsilon_{n}}^{(s)}\times\Pi_{r=1}^{d-2}[0,1]\right),$
for $0\le\ell<(nq_{n}q_{n}')^{d-2}$, $0\le i<q_{n}',0\le j<q_{n},s=1,2.$ 

Now take the product of a rectangle in the $(\theta_{1},\theta_{2})$-coordinates
with the same center as $S_{i,j}^{(s)}$ but $\theta_{1}$-length
$\frac{1-\varepsilon_{n}}{q_{n}q_{n}'}$ and $\theta_{2}$-length $\frac{1-2\varepsilon_{n}}{2q_{n}q_{n}'}$
with $\Pi_{r=1}^{d-2}[0,1].$ Divide the resulting cuboid into
$(nq_{n}q_{n}')^{d-2}$ congruent subcuboids by taking products of
intervals of length $\frac{1}{nq_{n}q_{n}'}$ in each of the $r_{1},\dots,r_{d-2}$
directions. Denote the resulting subcuboids by $\widetilde{S}_{i,j,\ell}^{(s)},$
$0\le\ell<(nq_{n}q_{n}')^{d-2}.$ Let $\widetilde{S}_{i,j,\ell,\varepsilon_{n}}^{(s)}$
be the cuboid with the same center as $\widetilde{S}_{i,j,\ell}^{(s)}$
and $\theta_{1}$-length $\frac{1-\epsilon_{n}}{q_{n}q_{n}'},$ $\theta_{2}$-length
$\frac{1-2\varepsilon_{n}}{2q_{n}q_{n}'},$ and $r_{1},\dots,r_{d-2}$-lengths
$\frac{1-2\varepsilon_{n}}{nq_{n}q_{n}'}.$ That is,
\[
\begin{array}{cccccc}
\widetilde{S}_{i,j,\ell,\varepsilon_{n}}^{(s)} & := & \left[\frac{i+(\varepsilon_{n}/2)}{q_{n}q_{n}'},\frac{i+1-(\varepsilon_{n}/2)}{q_{n}q_{n}'}\right]\times\left[\frac{j+(1/2)(2-s)+(\varepsilon_{n}/2)}{q_{n}q_{n}'},\frac{j+(1/2)(3-s)-(\varepsilon_{n}/2)}{q_{n}q_{n}'}\right]\\
 &  & \times\prod_{r=1}^{d-2}\left[\frac{\ell_{r}+\varepsilon_{n}}{nq_{n}q_{n}'},\frac{\ell_{r}+1-\varepsilon_{n}}{nq_{n}q_{n}'}\right],
\end{array}
\]
where each $\ell$ satisfying $0\le\ell<(nq_{n}q_{n}')^{d-2}$ is
identified with a unique $(d-2)$-tuple, $(\ell_{1},\dots,\ell_{d-2}),$
where $0\le\ell_{r}<nq_{n}q_{n}'$ for $r=1,\dots,d-2,$ so that $\ell=\Sigma_{r=1}^{d-2}\ell_r(nq_nq_n')^{r-1}.$ Then $\widetilde{S}_{i,j,\ell,\varepsilon_{n}}^{(s)}\subset\text{int }\left(S_{i,j,\varepsilon_{n}}^{(s)}\times\prod_{r=1}^{d-2}[0,1]\right),$ and the diameter of $\widetilde{S}_{i,j,\ell,\varepsilon_{n}}^{(s)}$ is less than $\sqrt{(5/4)+(d-2)n^{-2}}(q_nq_n')^{-1}.$
Moreover, vol($\widetilde{S}_{i,j,\ell,\varepsilon_{n}}^{(s)}$)=vol($S_{i,j,\ell,\varepsilon_{n}}^{(s)}).$
We apply Lemma \ref{lem:Zauberlemma} to obtain a measure-preserving $C^{\infty}$
diffeomorphism $\psi_{n}$ from $S_{0,0,\varepsilon_{n}}^{(2)}\times\prod_{r=1}^{d-2}[0,1]$
to itself which is the identity in a neighborhood of the boundary
and is an affine map from $\widetilde{S}_{0,0,\ell,\varepsilon_{n}}^{(2)}$
to $S_{0,0,\ell,\varepsilon_{n}}^{(2)},$ for $0\le\ell<(nq_{n}q_{n}')^{d-2}.$
We extend $\psi_{n}$ to be the identity on $\left(S_{0,0}^{(2)}\setminus S_{0,0,\varepsilon_{n}}^{(2)}\right)\times\prod_{r=1}^{d-2}[0,1].$
Then we further extend $\psi_{n}$ to a $\left(\frac{1}{q_{n}q_{n}'},\frac{1}{2q_{n}q_{n}'}\right)$-
equivariant diffeomorphism from $\mathbb{T}^{2}\times\prod_{r=1}^{d-2}[0,1]$
to itself. Then $\psi_n:\widetilde{S}^{(s)}_{i,j,\ell,\varepsilon_{n}}\to S^{(s)}_{i,j,\ell,\varepsilon_{n}}$ is an affine map, for $0\le i<q_n', 0\le j<q_n, 0\le \ell<(nq_nq_n')^{d-2}, s=1,2.$

We replace $\varphi_{n}:\mathbb{T}^{2}\to\mathbb{T}^{2}$ from 
Lemma \ref{lem:translationlemma} by $\varphi_{n}\times\text{Id}$, where Id denotes the identity
map on $\prod_{r=1}^{d-2}[0,1]$, so that the new $\varphi_{n}$ is
a diffeomorphism from $\mathbb{T}^{2}\times\prod_{r=1}^{d-2}[0,1]$ to
itself. Then we define $h_{n,1}:=\varphi_{n}\circ\psi_{n}.$

\begin{dfn} \label{dfn:goodh1}
The ``good domain'' of $h_{n,1}^{-1}$ is defined to be
\[
\mathscr{D}(h_{n,1}^{-1}):=\bigcup_{\begin{array}{c}
0\le b<q_{n}\\
0\le b'<q_{n}'
\end{array}}R_{\frac{b}{q_{n}},\frac{b'}{q_{n}'}}\left(\bigcup_{s=1,2}\bigcup_{\begin{array}{c}
0\le i<q_{n}'\\
0\le j<q_{n}
\end{array}}\bigcup_{0\le\ell<(nq_{n}q_{n}')^{d-2}}S_{i,j,\ell,\varepsilon_{n}}^{(s)}\right).
\]
\end{dfn}
From Lemma \ref{lem:translationlemma} and the definition of $\psi_{n}$ we obtain the following
lemma, which tells us that $h_{n,1}^{-1}$ is an affine map on each
component of $\mathscr{D}(h_{n,1}^{-1}).$
\begin{lem} 
\label{lem:prop h1} The restriction of the diffeomorphism $h_{n,1}^{-1}$
to $S_{i,j,\ell,\varepsilon_{n}}^{(s)}$ is an affine map such that

\[
\begin{array}{cccc}
h_{n,1}^{-1}(S_{i,j,\ell,\varepsilon_{n}}^{(s)}) & = & R_{\frac{-a_{n,s}(i,j)}{q_{n}},\frac{-a_{n,s}'(i,j)}{q_{n}'}}\widetilde{S}_{i,j,\ell,\varepsilon_{n}}^{(s)}, & \text{or equivalently,}\\
h_{n,1}(\widetilde{S}_{i,j,\ell,\varepsilon_{n}}^{(s)}) & = & R_{\frac{a_{n,s}(i,j)}{q_{n}},\frac{a_{n,s}'(i,j)}{q_{n}'}}S_{i,j,\ell,\varepsilon_{n}}^{(s)}=\ \ \ S^{(s)}_{k(i,j),\ell,\varepsilon_n}.
\end{array}
\]
\end{lem}

\begin{rem}
\label{rem:small diameter} Note that $R_{\frac{a_{n,s}(i,j)}{q_{n}},\frac{a_{n,s}'(i,j)}{q_{n}'}}S_{i,j,\ell,\varepsilon_{n}}^{(s)}$
is a subcuboid of $S^{(s)}_{k(i,j),\varepsilon_{n}}$. Moreover, the diameter of $h_{n,1}^{-1}(S_{i,j,\ell,\varepsilon_{n}}^{(s)})$ is
less than $\sqrt{(5/4)+(d-2)n^{-2}}(q_nq_n')^{-1}.$ Thus the map $h_{n,1}$ achieves the goals stated at the beginning of Subsection \ref{subsection:h1}. 
\end{rem}

\subsection{The conjugation map $h_{n,2}$} \label{subsection:h2}
In Remark \ref{rem:outline towers} we gave a preview of the shape of the sets in the definition
of our tower bases. In case of the first tower these sets will be
a collection of $q_{n}q_{n}'$ rectangular boxes with $\theta_{1}$-width
$\frac{q_{n}q_{n}'}{q_{n+1}}$ and $\theta_{2}$-height slightly less
than $\frac{1}{2q_{n}q_{n}^{\prime}},$ while in case of the second
tower the boxes will have $\theta_{1}$-width slightly less than $\frac{1}{2q_{n}q_{n}^{\prime}}$
and $\theta_{2}$-height slightly less than $\frac{q_{n}q_{n'}}{\overline{q}'_{n+1}}$
. The levels of the towers will be the images of these bases under
$R_{\alpha_{n+1},\alpha_{n+1}^{\prime}}^{i}$, where $0\leq i<m_{n}-1$
in case of the first tower, and $0\le i<m_{n}$ in case of the second
tower. If we consider one of the boxes in the base of the first tower
(for example, $C_{0}^{(1)}$ in Figure \ref{fig:tower}), then the images of this
box under $R^{jq_{n}q_{n}'}_{\alpha_{n+1},\alpha_{n+1}^{\prime}},$ for $j$ ranging from $0$ to approximately
$\frac{q_{n+1}}{(q_{n}q_{n}')^{2}}$, are adjacent to each other and
they remain in the upper parallelogram in Figure \ref{fig:tower}. A similar statement
holds for the images of one of the boxes in the base of the second
tower (for example, $C_{0}^{(2)}$ in Figure \ref{fig:tower}) under $R^{jq_{n}q_{n}'}_{\alpha_{n+1},\alpha_{n+1}^{\prime}}$
, which remain in the lower parallelogram in Figure \ref{fig:tower}. However the
small boxes in Figure \ref{fig:tower} leave the standard domains of the form $\left[\frac{i_{1}}{q_{n}q_{n}^{\prime}},\frac{i_{1}+1}{q_{n}q_{n}^{\prime}}\right]\times\left[\frac{i_{2}}{q_{n}q_{n}^{\prime}},\frac{i_{2}+1}{q_{n}q_{n}^{\prime}}\right]$.
This would cause difficulties in proving that the family of tower
levels establishes a generating sequence of partitions, especially
with regard to a linked approximation of type $(h,h+1)$ (see Lemma
\ref{lem:eta}). Hence, the goal of the map $h_{n,2}$ is to deform the standard
domain in such a way that the iterates of the boxes in the tower bases
under $R^{jq_{n}q_{n}'}_{\alpha_{n+1},\alpha_{n+1}^{\prime}}$, for $j$ ranging from $0$ to approximately
$\frac{q_{n+1}}{(q_{n}q_{n}')^{2}}$ in the case of the first tower,
and from $0$ to approximately $\frac{\overline{q}_{n+1}'}{(q_{n}q_{n}')^{2}}$
in the case of the second tower, remain in the same domain. (See Figures
\ref{figure:shapes} and \ref{fig:domains} for sketchs of the intended shapes and Figure \ref{fig:tower} for a sketch
of the positioning of the tower levels.)

\subsubsection{The maps $\Xi_{j,\lambda,\gamma, \varepsilon, \delta}$}
For $\delta>0$ we let $\sigma_{\delta}: \mathbb{R}\rightarrow \left[0,1\right]$ be a smooth map satisfying $\sigma_{\delta}\left(x\right)=0$ for $x \leq \frac{\delta}{2}$, $\sigma_{\delta}\left(x\right)=1$ for $\delta \leq x\leq 1- \delta$ and $\sigma_{\delta}\left(x\right)=0$ for $x\geq 1 - \frac{\delta}{2}$. Additionally, for any $0<\rho<1$ we choose a $C^{\infty}$-function $\tilde{\b}_{\rho}: [0,1] \to \left[ 0, \frac{2-\rho}{4}\right)$ satisfying $\tilde{\b}_{\rho}(x) = 0$ for $x \leq \frac{1}{2}$, $\tilde{\b}_{\rho}(x) = x-\frac{1}{2}$ for $\frac{1+\rho}{2} \leq x \leq  1-\frac{\rho}{2}$ and $\tilde{\b}_{\rho}(x)=0$ for $x \geq 1-\frac{\rho}{4}$. We also require that $\tilde{\b}_{\rho}'(x)\ge 0$ for $\frac{1}{2}\le x\le \frac{1+\rho}{2}$, $\tilde{\b}_{\rho}(x)<x-\frac{1}{2}$ for $1-\frac{\rho}{2}<x\le 1-\frac{\rho}{4},$ and $\tilde{\b}_{\rho}(\frac{2+\rho}{4}+x)> x$ for $0\le x\le \frac{\rho}{4}$.
\begin{figure}[hbtp] 
\begin{center}
\includegraphics[scale=0.7]{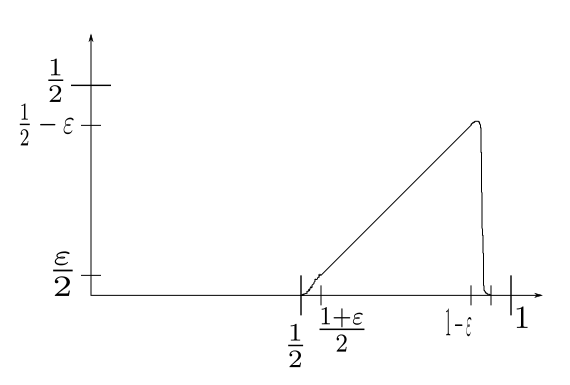}
\caption{Qualitative representation of the graph of $\tilde{\b}_{\rho}$.}
\end{center}
\end{figure}
We extend $\tilde{\b}_{\rho}$ to a $1$-periodic $C^{\infty}$-map $\tilde{\b}_{\rho}: \R \to \left[ 0, \frac{2-\rho}{4}\right)$. Then we define two measure-preserving diffeomorphisms $\Xi_{j,\rho, \delta}: \R^2 \times [0,1]^{d-2}\to \R^2\times [0,1]^{d-2}$, $j=1,2$,
\begin{align*}
\Xi_{1,\rho, \delta} \left(x_1, x_2, r_1,\dots,r_{d-2}\right) & = \left( x_1 + \tilde{\b}_{\rho}\left(x_2\right)\cdot \sigma_{\delta}\left(r_1\right) \cdot ...\cdot \sigma_{\delta}\left(r_{d-2}\right), x_2, r_1,\dots,r_{d-2} \right), \\
\Xi_{2,\rho, \delta} \left(x_1, x_2, r_1,\dots,r_{d-2}\right) & = \left( x_1 , x_2+ \tilde{\b}_{\rho}\left(x_1\right)\cdot \sigma_{\delta}\left(r_1\right) \cdot ...\cdot \sigma_{\delta}\left(r_{d-2}\right), r_1,\dots,r_{d-2}\right).
\end{align*} 
If $d=2$, then neither the coordinates $r_i$ nor the maps $\sigma_{\delta}$ occur. 

Let $\lambda, \gamma \in \N$. Then we define $\b_{\lambda, \gamma,  \varepsilon}: \left[0, \frac{1}{\lambda}\right] \to \left[ 0, \frac{2-(\varepsilon/\gamma)}{4\lambda} \right)$ by $\b_{\lambda, \gamma, \varepsilon}(x) = \frac{1}{\lambda}\tilde{\b}_{\rho}(\lambda \cdot x),$ where $\rho=\frac{\varepsilon}{\gamma}$. Since this map takes the value $0$ in a neighbourhood of the boundary, we can extend it smoothly to $\mathbb{T}$ by the rule $\b_{\lambda, \gamma,  \varepsilon}\left(x+\frac{1}{\lambda}\right) = \b_{\lambda, \gamma,  \varepsilon}\left(x\right)$. Then we define two measure-preserving diffeomorphisms $\Xi_{j,\lambda,\varepsilon, \delta}: \T^2 \times [0,1]^{d-2} \to \T^2 \times [0,1]^{d-2}$, $j=1,2$,
\begin{align*}
\Xi_{1,\lambda,\gamma,\varepsilon, \delta} \left(\theta_1, \theta_2, r_1,\dots,r_{d-2} \right) & = \left( \theta_1 + \b_{\lambda,\gamma, \varepsilon}\left(\theta_2\right)\cdot \sigma_{\delta}\left(r_1\right) \cdot ...\cdot \sigma_{\delta}\left(r_{d-2}\right), \theta_2,r_1,\dots,r_{d-2} \right), \\
\Xi_{2,\lambda,\gamma \varepsilon, \delta} \left(\theta_1, \theta_2, r_1,\dots,r_{d-2} \right) & = \left( \theta_1 , \theta_2+ \b_{\lambda,\gamma, \varepsilon}\left(\theta_1\right)\cdot \sigma_{\delta}\left(r_1\right) \cdot ...\cdot \sigma_{\delta}\left(r_{d-2}\right),r_1,\dots,r_{d-2} \right).
\end{align*}
The maps $\sigma_{\delta}$ are introduced to guarantee that both maps $\Xi_{j,\lambda,\gamma,\varepsilon, \delta}$, $j=1,2$,  coincide with the identity in a neighbourhood of the boundary.

\subsubsection{The map $\Theta_{\lambda, \gamma, \varepsilon,\delta}$} \label{subsubsec:Theta}
 We consider the sets

\begin{equation} \label{eq:Gtilde}
\begin{split} 
\tilde{G}_1= & \Meng{(x,y) \in \R^2}{\frac{\varepsilon}{2\gamma}\leq x \leq 1-\frac{\varepsilon}{2\gamma},\ \frac{\varepsilon}{2} + x \leq y \leq \frac{1-\varepsilon}{2}+x}, \\
\tilde{G}_2= & \Meng{(x,y) \in \R^2}{\frac{\varepsilon}{2\gamma} \leq y \leq 1-\frac{\varepsilon}{2\gamma},\ \frac{\varepsilon}{2} + y \leq x \leq \frac{1-\varepsilon}{2}+y},
\end{split}
\end{equation}
where $\gamma\in \mathbb{N}$ and $0<\delta<1/2.$ Later we will take $\gamma=(nq_nq_n')^{d-2}$ and $\delta=\varepsilon$. 

In order to give the idea of the construction we present it in the case $d=2$ first. In this case, $\gamma = 1.$
\begin{figure}[hbtp]
\begin{center}
\includegraphics[scale=0.7]{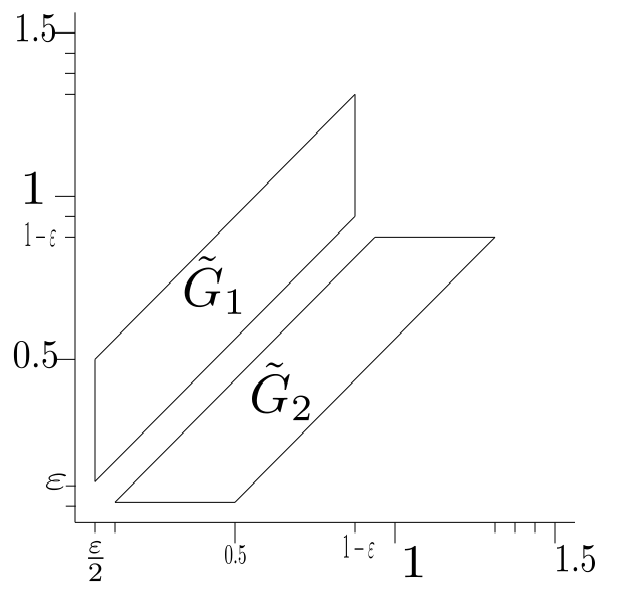}
\caption{Shape of the sets $\tilde{G}_1$ and $\tilde{G}_2$ in the case $d=2$.} \label{figure:shapes}
\end{center}
\end{figure}
As the set $G_s$ we take the set $\Xi^{-1}_{1,\rho, \delta} \circ \Xi^{-1}_{2,\rho, \delta} \left(\tilde{G}_s\right)$, $s=1,2$, where $\rho=\frac{\varepsilon}{\gamma} = \varepsilon$. We observe that $G_s\subset \left(\frac{\varepsilon}{4}, 1-\frac{\varepsilon}{4} \right)^2$.

\begin{figure}[hbtp] 
\begin{center}
\includegraphics[scale=0.7]{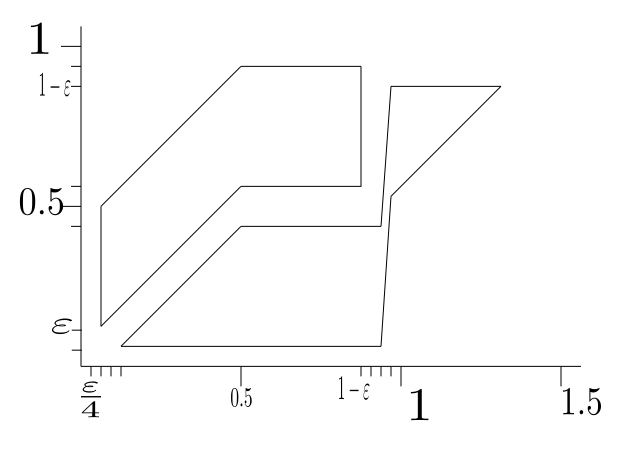}
\caption{Qualitative representation of the action of $\Xi^{-1}_{2,\varepsilon, \delta}$ on the sets $\tilde{G}_1$ and $\tilde{G}_2$. We note that the condition $\tilde{\b}_{\varepsilon}(\frac{2+\varepsilon}{4}+x)> x$ for $0\le x\le \frac{\varepsilon}{4}$ keeps the figure below $y=1-\frac{\varepsilon}{4}$.}
\end{center}
\end{figure}
\begin{figure}[hbtp] 
\begin{center}
\includegraphics[scale=0.7]{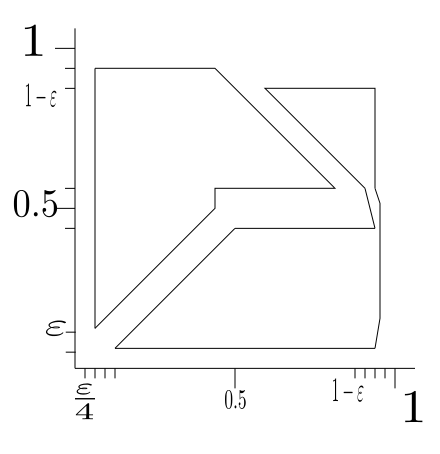}
\caption{Qualitative representation of the action of $\Xi^{-1}_{1,\varepsilon, \delta}$ on the sets $\Xi^{-1}_{2, \varepsilon, \delta}\left(\tilde{G}_1\right)$ and $\Xi^{-1}_{2,\varepsilon, \delta}\left(\tilde{G}_2\right)$. Hence, we see the shape of the sets $G_1$ and $G_2$.}
\end{center}
\end{figure}

We take a set 
\begin{equation*}
F_1=\left[\frac{\varepsilon}{2}, 1-\frac{\varepsilon}{2} \right] \times \left[ \frac{1+\varepsilon}{2}, 1-\frac{\varepsilon}{2} \right], 
\end{equation*}
as well as a set  
\begin{equation*}
F_2=\left[\frac{\varepsilon}{2}, 1-\frac{\varepsilon}{2} \right] \times \left[ \frac{\varepsilon}{2}, \frac{1-\varepsilon}{2} \right] .
\end{equation*}
Since $F_s$ and $\tilde{G}_s$ are parallelograms with $\mu(F_s)=\mu(\tilde{G}_s)$, there is an orientation-preserving measure-preserving affine map $A_s:\mathbb{R}^2\to\mathbb{R}^2$ with $A_s(F_s)=\tilde{G}_s$. Let $F_{s,+}$ and $\tilde{G}_{s,+}$ be 2-cells with $F_s\subset$ int($F_{s,+})$, $\tilde{G}_s\subset$ int($\tilde{G}_{s,+})$ such that $A_s(F_{s,+})=\tilde{G}_{s,+}.$ Let $G_{s,+}=\Xi^{-1}_{1,\varepsilon, \delta} \circ \Xi^{-1}_{2,\varepsilon, \delta} \left(\tilde{G}_{s,+}\right)$, $s=1,2$, where we assume $F_{s,+}$ was chosen small enough so that $G_{1,+}\cap G_{2,+} =\emptyset$ and $F_{s,+}\cup G_{s,+}\subset \left(\frac{\varepsilon}{4},1-\frac{\varepsilon}{4}\right)^2$. By Lemma \ref{lem:Zauberlemma}, there is a volume-preserving orientation-preserving smooth diffeomorphism $\Theta_{\varepsilon}$ of $[0,1]^2$ such that $\Theta_{\varepsilon}$ agrees with the identity outside $\left[\frac{\varepsilon}{4},1-\frac{\varepsilon}{4}\right]^2$ and $\Theta_{\varepsilon}|_{F_{s,+}}=A_s|_{F_{s,+}}$; in particular, $\Theta_{\varepsilon}(F_s)=G_s$. (The sets $F_{s,+}$ and $G_{s,+}$ were introduced because, technically, Lemma \ref{lem:Zauberlemma} cannot be applied directly to the sets $F_s$ and $G_s$, which have corner points. In the application of Lemma \ref{lem:Zauberlemma}, we can also take $M=\left(\frac{\varepsilon}{4},1-\frac{\varepsilon}{4}\right)^2$ to avoid corner points.)

Now we describe the construction in case of $d>2$. Let $\gamma\in \mathbb{N}.$ This time we consider the sets
\begin{align*}
\tilde{G}_{1,\ell}= & \Meng{(x,y) \in \R^2}{\frac{\ell}{\gamma}+\frac{\varepsilon}{2\gamma}\leq x \leq \frac{\ell+1}{\gamma}-\frac{\varepsilon}{2\gamma},\ \frac{\varepsilon}{2} + x \leq y \leq \frac{1-\varepsilon}{2}+x}\subset \widetilde{G}_1, \\
\tilde{G}_{2,\ell}= & \Meng{(x,y) \in \R^2}{\frac{\ell}{\gamma}+\frac{\varepsilon}{2\gamma}\leq y \leq \frac{\ell+1}{\gamma}-\frac{\varepsilon}{2\gamma},\ \frac{\varepsilon}{2} + y \leq x \leq \frac{1-\varepsilon}{2}+y}\subset \widetilde{G}_2,
\end{align*}
for $0 \leq \ell < \gamma$.
We let $G_{s,\ell}=\Xi^{-1}_{1,\rho, \delta} \circ \Xi^{-1}_{2,\rho, \delta} \left(\tilde{G}_{s,\ell} \times \left[ \delta, 1-\delta \right]^{d-2}\right)$, where $\rho=\frac{\varepsilon}{\gamma}$, $s=1,2$ and $0 \leq \ell < \gamma$. As above, we observe that $G_{s,\ell} \subset \left(\frac{\varepsilon}{4\gamma}, 1-\frac{\varepsilon}{4\gamma} \right)^2 \times [\delta,1-\delta]^{d-2}$. On the other hand, we take a set
\begin{equation*}
F_{1,\ell} = \left[\frac{\ell}{\gamma}+\frac{\varepsilon}{2\gamma}, \frac{\ell+1}{\gamma}-\frac{\varepsilon}{2\gamma} \right] \times \left[ \frac{1+\varepsilon}{2}, 1-\frac{\varepsilon}{2} \right] \times \left[ \delta, 1-\delta \right]^{d-2} ,  
\end{equation*}
as well as a set 
\begin{equation*}
F_{2,\ell}=\left[\frac{\ell}{\gamma}+\frac{\varepsilon}{2\gamma}, \frac{\ell+1}{\gamma} - \frac{\varepsilon}{2\gamma}\right] \times \left[ \frac{\varepsilon}{2}, \frac{1-\varepsilon}{2} \right] \times \left[ \delta, 1-\delta \right]^{d-2}   .
\end{equation*}
Since $F_{s,\ell}$ and $\tilde{G}_{s,\ell}$ are parallelepipeds with $\mu(F_{s,\ell})=\mu(\tilde{G}_{s,\ell})$, there are orientation-preserving measure-preserving affine maps $A_{s,\ell}:\mathbb{R}^d\to \mathbb{R}^d$ such that $A_s(F_{s,\ell})=\tilde{G}_{s,\ell}.$ Thus we can apply Lemma \ref{lem:Zauberlemma} as in the two-dimensional case to obtain a volume-preserving smooth diffeomorphism $\Theta_{\gamma,\varepsilon,  \delta}: \left[0,1\right]^d \to \left[0,1\right]^d$ with $\Theta_{\gamma,\varepsilon , \delta} \left( F_{s,\ell} \right) = G_{s,\ell}$, $s=1,2$, $0 \leq \ell < \gamma$, and which coincides with the identity outside of $\left[\frac{\varepsilon}{4\gamma},1-\frac{\varepsilon}{4\gamma}\right]^2 \times \left[ \frac{\delta}{2}, 1-\frac{\delta}{2}\right]^{d-2}$. 


In the next step, let $\lambda \in \N$. Then we define the stretching $C_{\lambda}: \left[0,\frac{1}{\lambda}\right]^2 \times \left[0,1\right]^{d-2} \to \left[0,1\right]^d$ by $C_{\lambda}(x_1,x_2,x_3, \dots,x_d) = (\lambda x_1, \lambda x_2, x_3, \dots, x_d)$ and $\Theta_{\lambda,\gamma, \varepsilon, \delta}: \left[0,\frac{1}{\lambda}\right]^2 \times \left[0,1\right]^{d-2} \to \left[0,\frac{1}{\lambda}\right]^2 \times \left[0,1\right]^{d-2}$ by $\Theta_{\lambda, \gamma, \varepsilon, \delta}=C^{-1}_{\lambda} \circ \Theta_{ \gamma,\varepsilon, \delta} \circ C_{\lambda}$. Since $\Theta_{\lambda, \gamma, \varepsilon, \delta}$ coincides with the identity in a neighbourhood of the boundary of its domain, we can extend it to a smooth diffeomorphism on $\T^2 \times [0,1]^{d-2}$ by
\begin{align*}
& \Theta_{\lambda, \gamma, \varepsilon, \delta}: \prod^2_{i=1} \left[\frac{j_i}{\lambda}, \frac{j_i + 1}{\lambda}\right] \times \left[0,1\right]^{d-2} \to \prod^2_{i=1} \left[\frac{j_i}{\lambda}, \frac{j_i + 1}{\lambda}\right]\times \left[0,1\right]^{d-2}, \\
& \Theta_{\lambda, \gamma, \varepsilon, \delta}\left( \frac{j_1}{\lambda}+x_1, \frac{j_2}{\lambda}+x_2, x_3, \dots,x_d \right) = \left(\frac{j_1}{\lambda}, \frac{j_2}{\lambda},0, \dots,0\right)+\Theta_{\lambda, \gamma, \varepsilon, \delta}\left(x_1,\dots,x_d\right),
\end{align*}
for $0\le j_1,j_2<\lambda$, $(x_1,\dots,x_d)\in \left[0,\frac{1}{\lambda}\right]^2\times \left[0,1\right]^{d-2}.$

\subsubsection{Definition of $h_{n,2}$} \label{subsubsection:h2}
Using the maps constructed in the previous subsections and choosing the parameters $\lambda=q_nq_n'$, $\gamma=(nq_nq_n')^{d-2}$, and $\delta=\varepsilon_n$, we define
\begin{equation*}
h_{n,2} = \Xi_{2, q_nq^{\prime}_n, (nq_nq_n')^{d-2},\varepsilon_n, \varepsilon_n} \circ \Xi_{1, q_nq^{\prime}_n, (nq_nq_n')^{d-2}, \varepsilon_n, \varepsilon_n} \circ \Theta_{q_n q^{\prime}_n, \left(nq_nq^{\prime}_n \right)^{d-2}, \varepsilon_n, \varepsilon_n},
\end{equation*}
and observe
\begin{equation*}
h_{n,2} \circ R_{\frac{i}{q_nq^{\prime}_n},\frac{j}{q_n q^{\prime}_n}} = R_{\frac{i}{q_nq^{\prime}_n},\frac{j}{q_n q^{\prime}_n}} \circ h_{n,2} \text{ for every } i,j \in \Z.
\end{equation*}


By construction we have the following properties:
\begin{lem}[Properties of $h^{-1}_{n,2}$] \label{lem:prop h2}
Let $0\le j_{1},j_{2}<q_{n}q_{n}'.$ If $d>2,$ let $0\le\ell<(nq_{n}q_{n}')^{d-2}.$
If $d=2,$ then $\ell=0$ and we let $I_{j_{1},j_{2}}^{(n,s)}:=I_{j_{1},j_{2},0}^{(n,s)}$
(as defined below). The smooth diffeomorphism $h_{n,2}^{-1}$ maps
the set $I_{j_{1},j_{2},\ell}^{(n,1)}$ of points $\left(\theta_{1},\theta_{2},r_{1},\dots,r_{d-2}\right)\in\mathbb{T}^{2}\times\left[\varepsilon_{n},1-\varepsilon_{n}\right]^{d-2}$
with
\[
\begin{split} & \theta_{1}\in\Bigg[\frac{j_{1}}{q_{n}q_{n}^{\prime}}+\frac{\ell+(\varepsilon_{n}/2)}{n^{d-2}\cdot\left(q_{n}q_{n}^{\prime}\right)^{d-1}},\frac{j_{1}}{q_{n}q_{n}^{\prime}}+\frac{\ell+1-(\varepsilon_{n}/2)}{n^{d-2}\cdot\left(q_{n}q_{n}^{\prime}\right)^{d-1}}\Bigg],\\
 & \frac{2j_{2}+\varepsilon_{n}}{2q_{n}q_{n}^{\prime}}+\theta_{1}-\frac{j_{1}}{q_{n}q_{n}^{\prime}}\leq\theta_{2}\leq\frac{2j_{2}+1-\varepsilon_{n}}{2q_{n}q_{n}^{\prime}}+\theta_{1}-\frac{j_{1}}{q_{n}q_{n}^{\prime}}
\end{split}
\]
 to
\begin{equation} \label{eq:part1goodh2}
\begin{split} & \Bigg[\frac{j_{1}}{q_{n}q_{n}^{\prime}}+\frac{\ell+(\varepsilon_{n}/2)}{n^{d-2}\cdot\left(q_{n}q_{n}^{\prime}\right)^{d-1}},\frac{j_{1}}{q_{n}q_{n}^{\prime}}+\frac{\ell+1-(\varepsilon_{n}/2)}{n^{d-2}\cdot\left(q_{n}q_{n}^{\prime}\right)^{d-1}}\Bigg]\\
\times & \left[\frac{2j_{2}+1+\varepsilon_{n}}{2q_{n}q_{n}^{\prime}},\frac{2(j_{2}+1)-\varepsilon_{n}}{2q_{n}q_{n}^{\prime}}\right]\times\prod_{i=3}^{d}\left[\varepsilon_{n},1-\varepsilon_{n}\right].
\end{split}
\end{equation}
Moreover, $h_{n,2}^{-1}$ maps the set $I_{j_{1},j_{2},\ell}^{(n,2)}$
of points $\left(\theta_{1},\theta_{2},r_{1},\dots,r_{d-2}\right)\in\mathbb{T}^{2}\times\left[\varepsilon_{n},1-\varepsilon_{n}\right]^{d-2}$
with
\[
\begin{split} & \theta_{2}\in\Bigg[\frac{j_{2}}{q_{n}q_{n}^{\prime}}+\frac{\ell+(\varepsilon_{n}/2)}{n^{d-2}\cdot\left(q_{n}q_{n}^{\prime}\right)^{d-1}},\frac{j_{2}}{q_{n}q_{n}^{\prime}}+\frac{\ell+1-(\varepsilon_{n}/2)}{n^{d-2}\cdot\left(q_{n}q_{n}^{\prime}\right)^{d-1}}\Bigg],\\
 & \frac{2j_{1}+\varepsilon_{n}}{2q_{n}q_{n}^{\prime}}+\theta_{2}-\frac{j_{2}}{q_{n}q_{n}^{\prime}}\leq\theta_{1}\leq\frac{2j_{1}+1-\varepsilon_{n}}{2q_{n}q_{n}^{\prime}}+\theta_{2}-\frac{j_{2}}{q_{n}q_{n}^{\prime}}
\end{split}
\]
 to 
\begin{equation}\label{eq:part2goodh2}
\begin{split} & \Bigg[\frac{j_{1}}{q_{n}q_{n}^{\prime}}+\frac{\ell+(\varepsilon_{n}/2)}{n^{d-2}\cdot\left(q_{n}q_{n}^{\prime}\right)^{d-1}},\frac{j_{1}}{q_{n}q_{n}^{\prime}}+\frac{\ell+1-(\varepsilon_{n}/2)}{n^{d-2}\cdot\left(q_{n}q_{n}^{\prime}\right)^{d-1}}\Bigg]\\
\times & \left[\frac{2j_{2}+\varepsilon_{n}}{2q_{n}q_{n}^{\prime}},\frac{2j_{2}+1-\varepsilon_{n}}{2q_{n}q_{n}^{\prime}}\right]\times\prod_{i=3}^{d}\left[\varepsilon_{n},1-\varepsilon_{n}\right].
\end{split}
\end{equation}\end{lem}

\begin{pr}
Let $\delta=\varepsilon_{n}$ and $\gamma=(nq_{n}q_{n}')^{d-2}$ in
the definitions of $\widetilde{G}_{s,\ell}$ and $F_{s,\ell}$, and let $\lambda=q_nq_n'$ in the definition of the map $C_{\lambda}$. Note that the collection
of  sets $I_{j_{1},j_{2},\ell}^{(n,s)}$
for fixed $j_{1},j_{2},$ and $\ell$ varying as above, is exactly
the same as the collection of scaled and translated versions of the sets  $\widetilde{G}_{s,\ell}\times \left[ \varepsilon_n, 1-\varepsilon_n \right]^{d-2}$ given by  $C_{\lambda}^{-1}(\widetilde{G}_{s,\ell}\times \left[ \varepsilon_n, 1-\varepsilon_n \right]^{d-2})+(\frac{j_1}{q_nq_n'},\frac{j_2}{q_nq_n'},0,\dots,0)$.  Likewise, the sets in (\ref{eq:part1goodh2}) and (\ref{eq:part2goodh2}) are the same as the sets  $C_{\lambda}^{-1}({F}_{s,\ell})+(\frac{j_1}{q_nq_n'},\frac{j_2}{q_nq_n'},0,\dots,0)$, for $s=1,2,$ respectively. 
\end{pr}

    We let the union of the scaled and translated versions of the set  $\widetilde{G}_{s,\ell}\times \left[ \varepsilon_n, 1-\varepsilon_n \right]^{d-2}$ that occur in the proof of Lemma \ref{lem:prop h2} be denoted $\widehat{G}_s.$ That is,
\[
\widehat{G}_{s}:=\bigcup_{
0\le j_{1},j_{2}<q_{n}q_{n}'}\ \ \bigcup_{{0\le\ell<(nq_{n}q_{n}')^{d-2}}}\left[C_{\lambda}^{-1}(\widetilde{G}_{s, \ell} \times \left[ \varepsilon_n, 1-\varepsilon_n \right]^{d-2})+\left(\frac{j_{1}}{q_{n}q_{n}'},\frac{j_{2}}{q_{n}q_{n}'},0,\dots,0\right)\right].
\]
See Figure \ref{fig:domains}.

\begin{dfn} \label{dfn:goodh2}
The ``good domain'' of $h_{n,2}^{-1}$ is 
\[
\mathscr{D}(h_{n,2}^{-1}):=\bigcup_{s=1,2}\ \ \bigcup_{0\le j_{1},j_{2}<q_{n}q_{n}'}\ \ \bigcup_{0\le\ell<(nq_{n}q_{n}')^{d-2}}I_{j_{1},j_{2},\ell}^{(n,s)}=\bigcup_{s=1,2} \widehat{G}_s.
\]
\end{dfn}
\subsection{Definition of $h_n$}
The conjugation map $h_n$ is the composition
\begin{equation*}
h_n = h_{n,2} \circ h_{n,1}.
\end{equation*}

\begin{rem} \label{rem:identity}
We note that $h_n$ coincides with the identity outside of $\mathscr{R}_n \coloneqq \T^2 \times \left[\frac{\varepsilon_n}{2}, 1 - \frac{\varepsilon_n}{2} \right]^{d-2}$.
\end{rem} 

\begin{rem} \label{rem:hn2behavior}According to Lemma \ref{lem:prop h2}, we have 

\[
h_{n,2}^{-1}\left(I_{j_{1},j_{2},\ell}^{(n,s)}\right)=R_{\frac{b}{q_{n}},\frac{b'}{q_{n}'}}S_{i,j,\ell,\varepsilon_{n}}^{(s)},
\]
where
\begin{equation}\label{eq:j1j2}
\frac{i}{q_{n}q_{n}'}+\frac{b}{q_{n}}=\frac{j_{1}}{q_{n}q_{n}'}\text{\ and\ }\frac{j}{q_{n}q_{n}'}+\frac{b'}{q_{n}'}=\frac{j_{2}}{q_{n}q_{n}'}.
\end{equation}
In particular, if  
$b=a_{n,s}(i,j)$ and  $b'=a'_{n,s}(i,j)$, then  
\[
h_{n,2}(S^{(s)}_{k(i,j),\ell,\varepsilon_n})=I_{j_{1},j_{2},\ell}^{(n,s)},
\]
and thus, by Lemma \ref{lem:prop h1},
\[
h_{n}(\widetilde{S}^{(s)}_{i,j,\ell,\varepsilon_n})=I_{j_{1},j_{2},\ell}^{(n,s)}.
\]
\end{rem}
By Lemma \ref{lem:prop h2} we have $h_{n,2}^{-1}(\mathscr{D}(h_{n,2}^{-1}))=\mathscr{D}(h_{n,1}^{-1}). $ 
Thus we make the following definition.
\begin{dfn} \label{dfn:goodh} The ``good domain'' of $h_n^{-1}$ is $\mathscr{D}(h_n^{-1}):=\mathscr{D}(h_{n,2}^{-1}),$
as in Definition \ref{dfn:goodh2}.
\end{dfn}

\begin{figure}[hbtp] 
\begin{center}
\includegraphics[scale=0.7]{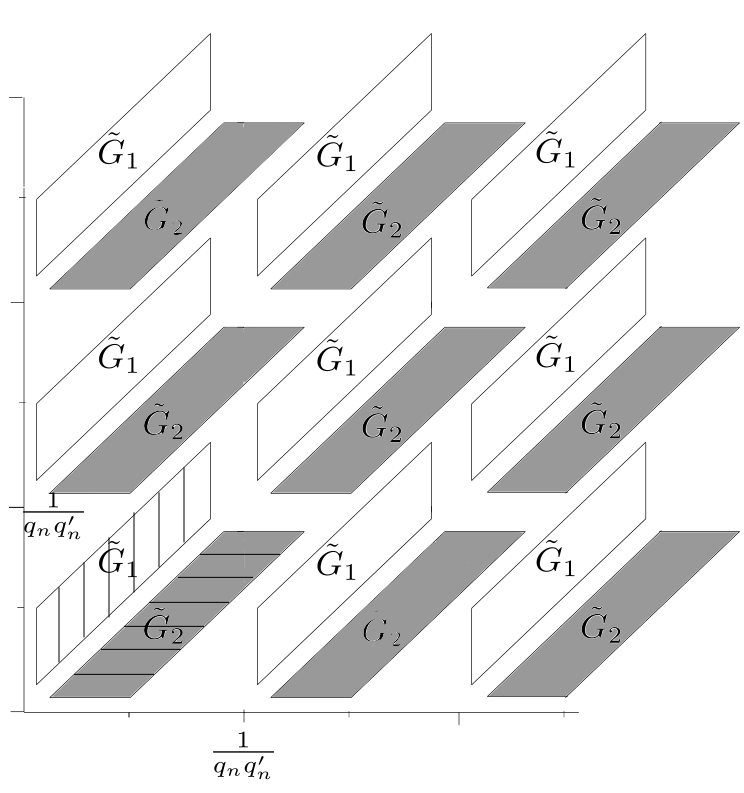}
\caption{``Good domains'' of the map $h^{-1}_{n}$ in the two-dimensional case. In the domains in the bottom left corner we have also indicated the additional subdivision that occurs in the $(\theta_1,\theta_2)-$factor of $\hat{G}_s$ in the higher dimensional case.}\label{fig:domains}
\end{center}
\end{figure}

\section{Convergence of $\left(T_n\right)_{n \in \N}$ in Diff$^{\infty}(M_0,\mu)$} \label{section:conv}
In the following we show that the sequence of constructed measure-preserving smooth diffeomorphisms $T_n = H^{-1}_n \circ R_{\alpha_{n+1}, \a^{\prime}_{n+1}} \circ H_{n}$ converges. For this purpose, the next result is very useful.

\begin{lem} \label{lem:konj}
Let $k \in \mathbb{N}_0$ and $h$ be a C$^{\infty}$-diffeomorphism on $M_0$. Then we get for every $\alpha, \a^{\prime},\beta, \b^{\prime} \in \mathbb{R}$:
\begin{equation*}
d_k\left(h \circ R_{\alpha, \a^{\prime}} \circ h^{-1}, h \circ R_{\beta, \b^{\prime}} \circ h^{-1}\right) \leq C_k \cdot ||| h |||^{k+1}_{k+1} \cdot \max \left(\left| \alpha - \beta \right|, \left| \alpha^{\prime} - \beta^{\prime} \right| \right)
\end{equation*}
where the constant $C_k$ depends solely on $k$. In particular $C_0 = 1$.
\end{lem}

\begin{pr}
The proof is similar to the one of \cite[Lemma 4]{FSW} and \cite[Lemma 6.4]{K-rigid2}. 

First of all, we will prove the following observation by induction on $k \in \mathbb{N}$: \\
\textbf{Claim: }\textit{For any multiindex $\vec{a} \in \mathbb{N}^d_0$ with $\left|\vec{a}\right|=k$ and $i\in \left\{1,...,d\right\}$ the partial derivative $D_{\vec{a}} \left[h \circ R_{\alpha, \a^{\prime}} \circ h^{-1}\right]_i$ consists of at most $\frac{(d+k-1)!}{(d-1)!}$ summands, where each summand is the product of one single partial derivative $\left(D_{\vec{b}}\left[h\right]_i\right)\circ R_{\alpha, \a^{\prime}} \circ h^{-1}$ of order at most $\big{|}\vec{b}\big{|} \leq k$ and at most $k$ derivatives $D_{\vec{b}} \left[ h^{-1}\right]_j$ with $\big{|}\vec{b}\big{|} \leq k$.}
\begin{itemize}
	\item \textit{Start:} $k=1$ \\
	For $i_1, i \in \left\{1,...,d\right\}$ we compute:
	\begin{equation*}
	D_{x_{i_1}} \left[h \circ R_{\alpha, \a^{\prime}} \circ h^{-1}\right]_i = \sum^{d}_{j_1=1} \left(D_{x_{j_1}} \left[h\right]_i\right) \circ  R_{\alpha, \a^{\prime}} \circ h^{-1} \cdot D_{x_{i_1}} \left[h^{-1}\right]_{j_1}.
	\end{equation*}
	Hence, this derivative consists of $d = \frac{(d+1-1)!}{(d-1)!}$ summands and each summand has the announced form.
	\item \textit{Induction assumption:} The claim holds for $k \in \mathbb{N}$.
	\item \textit{Induction step:} $k\rightarrow k+1$ \\
	Let $i \in \left\{1,..,d\right\}$ and $\vec{b} \in \mathbb{N}^d_0$ be any multiindex of order $\big{|}\vec{b}\big{|}=k+1$. There are $j \in \left\{1,...,d\right\}$ and a multiindex $\vec{a}$ of order $\left|\vec{a}\right|=k$ such that $D_{\vec{b}}= D_{x_j}D_{\vec{a}}$. By the induction assumption the partial derivative $D_{\vec{a}}\left[h \circ R_{\alpha, \a^{\prime}} \circ h^{-1}\right]_i$ consists of at most $\frac{(d+k-1)!}{(d-1)!}$ summands, such that the summands with the most factors are of the following form:
	\begin{equation*}
	\left(D_{\vec{c}_1} \left[h\right]_i\right) \circ R_{\alpha, \a^{\prime}} \circ h^{-1} \cdot D_{\vec{c}_2}\left[h^{-1}\right]_{i_2} \cdot \dots \cdot D_{\vec{c}_{k+1}}\left[h^{-1}\right]_{i_{k+1}},
	\end{equation*}
	where each $\vec{c}_i$ is of order at most $k$. Using the product rule we compute how the derivative $D_{x_j}$ acts on such a summand: 
	\begin{align*}
	&\sum^{d}_{j_1=1}\left( D_{x_{j_1}}D_{\vec{c}_1} \left[h\right]_i \right) \circ  R_{\alpha, \a^{\prime}} \circ h^{-1} \cdot D_{x_j}\left[ h^{-1} \right]_{j_1} D_{\vec{c}_2}\left[h^{-1}\right]_{i_2} \cdot \dots \cdot D_{\vec{c}_{k+1}}\left[h^{-1}\right]_{i_{k+1}}  +  \\
	&\left( D_{\vec{c}_1} \left[h\right]_i \right) \circ  R_{\alpha, \a^{\prime}} \circ h^{-1} \cdot D_{x_j}D_{\vec{c}_2}\left[h^{-1}\right]_{i_2}\cdot \dots \cdot D_{\vec{c}_{k+1}}\left[h^{-1}\right]_{i_{k+1}}+\dots + \\
	& \left( D_{\vec{c}_1} \left[h\right]_i \right) \circ  R_{\alpha, \a^{\prime}} \circ h^{-1} \cdot D_{\vec{c}_2}\left[h^{-1}\right]_{i_2} \cdot \dots \cdot D_{x_j}D_{\vec{c}_{k+1}}\left[h^{-1}\right]_{i_{k+1}}
	\end{align*}
	Thus, each summand is the product of one derivative of $h$ of order at most $k+1$ and at most $k+1$ derivatives of $h^{-1}$ of order at most $k+1$. Moreover, we observe that $d + k$ summands arise out of one. So the number of summands can be estimated by $\left(d+k\right) \cdot \frac{(d+k-1)!}{(d-1)!} = \frac{(d+k)!}{(d-1)!}$ and the claim is verified. \qed
\end{itemize}

Furthermore, with the aid of the mean value theorem we can estimate for any multiindex $\vec{a} \in \mathbb{N}^d_0$ with $\left|\vec{a}\right|\leq k$ and $i\in \left\{1,...,d\right\}$:
\begin{align*}
& \left|D_{\vec{a}}\left[h\right]_i\left( R_{\alpha, \a^{\prime}} \circ h^{-1}\left(x_1,...,x_d\right)\right)-D_{\vec{a}}\left[h\right]_i\left( R_{\b, \b^{\prime}} \circ h^{-1} \left(x_1,...,x_d\right)\right)\right| \\ 
\leq & |||h|||_{k+1} \cdot \max \left( \left|\alpha-\beta\right|, \left|\alpha^{\prime}-\beta^{\prime}\right| \right).
\end{align*}
Since $\left(h \circ R_{\alpha, \a^{\prime}} \circ h^{-1}\right)^{-1}= h \circ R_{-\alpha, -\a^{\prime}} \circ h^{-1}$ is of the same form, we obtain in conclusion:
\begin{align*}
d_k\left( h \circ R_{\alpha, \a^{\prime}} \circ h^{-1}, h \circ R_{\b, \b^{\prime}} \circ h^{-1}\right) & \leq \frac{(d+k-1)!}{(d-1)!} \cdot |||h|||_{k+1} \cdot |||h|||^k_{k} \cdot \max \left( \left|\alpha-\beta\right|, \left|\alpha^{\prime}-\beta^{\prime}\right| \right) \\
& \leq \frac{(d+k-1)!}{(d-1)!} \cdot |||h|||^{k+1}_{k+1} \cdot \max \left( \left|\alpha-\beta\right|, \left|\alpha^{\prime}-\beta^{\prime}\right| \right).
\end{align*}
\end{pr}

Under some conditions on the proximity of $\alpha_n$ to $\alpha_{n+1}$ as well as $\a^{\prime}_n$ to $\a^{\prime}_{n+1}$ we can prove convergence: 
\begin{lem} \label{lem:convgen}
Let $\varepsilon>0$, let $\varepsilon_n$ as in equation (\ref{eq:eps}) and $\left( l_n \right)_{n \in \N}$ be an increasing sequence of natural numbers satisfying $\sum^{\infty}_{i=1} \frac{1}{l_i} < \varepsilon$ and $\sum^{\infty}_{i=n+1} \frac{1}{l_i} < \varepsilon_n$. We assume that in our constructions the following conditions are fulfilled:
\begin{equation*}
\max\left( \left| \alpha_{n+1} - \alpha_n \right|, \left| \alpha^{\prime}_{n+1} - \alpha^{\prime}_n \right| \right) \leq \frac{1}{4 \cdot l_n \cdot C_{l_n} \cdot q_n q^{\prime}_n \cdot ||| H_n |||^{l_n+1}_{l_n+1}} 
\end{equation*}
for every  $n \in \mathbb{N}$, where $C_{n}$ are the constants from Lemma \ref{lem:konj}.
\begin{enumerate}
	\item Then the sequence of diffeomorphisms $T_n = H^{-1}_n \circ R_{\alpha_{n+1}, \a^{\prime}_{n+1}} \circ H_{n}$ converges in the Diff$^{\infty}(M_0)$-topology to a measure-preserving smooth diffeomorphism $T$.
	\item Let $\left( A, B \right) \in \T^2$ be arbitrary. Moreover, we assume $\max \left( \left| A-\a_1 \right|, \left| B-\a^{\prime}_1\right| \right)< \varepsilon$. Then we have 
\begin{equation*}
d_{\infty}\left(T,R_{A, B}\right)< 4 \cdot \varepsilon.
\end{equation*}
	\item We have for every $n \in \mathbb{N}$ and $m \leq q_{n+1}q^{\prime}_{n+1}$:
	\begin{equation*}
	d_0\left(T^{m}, T^{m}_n\right) < \frac{\varepsilon_n}{4} .
	\end{equation*} 
	\item Let $\rho = \left( \rho_n \right)_{n \in \mathbb{N}_0}$ be an admissible sequence. Under the additional assumption
	\begin{equation*}
	    \sum^{\infty}_{m=n+1} \frac{1}{l_m} < \min_{k=0,...,n} \min_{\vec{x} \in \mathscr{R}_{n+1} \setminus \mathscr{R}_n} \rho_k \left( \vec{x} \right)
	\end{equation*}
	we have $T \in \text{Diff}^{\infty}_{\rho, \a, \a^{\prime}} \left( \T^2 \times [0,1]^{d-2} \right)$, where $\a = \lim_{n \rightarrow \infty} \a_n$ and $\a^{\prime}= \lim_{n \rightarrow \infty} \a^{\prime}_n$.
\end{enumerate}
\end{lem}

\begin{pr}
\begin{enumerate}
	\item According to our construction we have $h_n \circ R_{\alpha_n, \a^{\prime}_n} = R_{\alpha_n, \a^{\prime}_n} \circ h_n$, and hence we can apply Lemma \ref{lem:konj} for every  $k \in \mathbb{N}_0$, $n \in \mathbb{N}$:
	\begin{align*}
	d_k\left( T_n, T_{n-1}\right) & = d_k\left(H^{-1}_n \circ R_{\alpha_{n+1}, \a^{\prime}_{n+1}} \circ H_{n}, H^{-1}_n \circ R_{\alpha_n, \a^{\prime}_n} \circ H_{n}\right) \\
	& \leq C_k \cdot ||| H_n |||^{k+1}_{k+1} \cdot \max \left( \left| \alpha_{n+1} - \alpha_n \right|, \left| \a^{\prime}_{n+1}- \a^{\prime}_n \right| \right).
	\end{align*}
	By the assumptions of this Lemma it follows for every $k \leq l_n$:
	\begin{equation} \label{est4}
	d_k\left(T_n,T_{n-1}\right) \leq d_{l_n}\left(T_n,T_{n-1}\right) \leq C_{l_n} \cdot ||| H_n |||^{l_n+1}_{l_n +1} \cdot \frac{1}{4 l_n \cdot C_{l_n} \cdot q_nq^{\prime}_n \cdot ||| H_n|||^{l_n+1}_{l_n +1} } < \frac{1}{4l_n}.
	\end{equation}
	In the next step we show that for arbitrary $k \in \mathbb{N}_0$, $\left(T_n\right)_{n \in \mathbb{N}}$ is a Cauchy sequence in Diff$^k\left(M_0\right)$, i.\ e. $\lim_{n,m\rightarrow \infty} d_k\left(T_n,T_m\right) = 0$. For this purpose, we calculate:
	\begin{equation} \label{est2}
	\lim_{n \rightarrow \infty} d_k\left(T_n,T_m\right) \leq \lim_{n \rightarrow \infty} \sum^{n}_{i=m+1} d_k\left(T_i, T_{i-1}\right) = \sum^{\infty}_{i=m+1}d_k\left( T_i, T_{i-1}\right).
	\end{equation}
	We consider the limit process $ m \rightarrow \infty$, i.e. we can assume $k\leq l_m$ and obtain from equations (\ref{est4}) and (\ref{est2}):
	\begin{equation*}
	\lim_{n,m \rightarrow \infty} d_k\left( T_n,T_m\right) \leq \lim_{m \rightarrow \infty} \sum^{\infty}_{i=m+1} \frac{1}{4l_i} = 0.
	\end{equation*}
	Since Diff$^{k}\left(M_0\right)$ is complete in the $d_k$ metric, it follows that the sequence $\left(T_n\right)_{n \in \mathbb{N}}$ converges in Diff$^k\left(M\right)$ for every  $k \in \mathbb{N}_0$. Thus, the sequence converges in Diff$^{\infty}\left(M_0\right)$ by definition. 
	\item Furthermore, we estimate:
	\begin{equation} \label{est3}
	d_{\infty}\left(R_{A, B},T \right)= d_{\infty}\left(R_{A, B}, \lim_{n\rightarrow \infty} T_n\right)
 \leq d_{\infty}\left(R_{A, B}, R_{\alpha_1, \a^{\prime}_1}\right) + \sum^{\infty}_{n=1} d_{\infty}\left(T_n, T_{n-1}\right),  
	\end{equation} 
	where we used the notation $T_0=R_{\alpha_1, \a^{\prime}_1}$. Since $d_k\left(R_{A,B}, R_{\alpha_1, \a^{\prime}_1}\right) = d_0\left(R_{A, B}, R_{\alpha_1, \a^{\prime}_1}\right) = \max \left( \left| A - \alpha_1 \right|, \left| B - \alpha^{\prime}_1 \right| \right)$ for every $k \in \mathbb{N}_0$, we have 
	\begin{equation*}
	d_{\infty}\left(R_{A,B}, R_{\alpha_1, \a^{\prime}_1}\right) = \sum^{\infty}_{k=0} \frac{\max \left( \left| A - \alpha_1 \right|, \left| B - \alpha^{\prime}_1 \right| \right)}{2^k \cdot \left(1+d_k\left(R_{A,B}, R_{\alpha_1, \a^{\prime}_1}\right)\right)} \leq 2 \cdot \max \left( \left| A - \alpha_1 \right|, \left| B - \alpha^{\prime}_1 \right| \right).
	\end{equation*}
	Additionally it holds:
	\begin{align*}
	\sum^{\infty}_{n=1} d_{\infty}\left( T_n, T_{n-1}\right) & = \sum^{\infty}_{n=1} \sum^{\infty}_{k=0} \frac{d_k\left(T_n, T_{n-1}\right)}{2^k \cdot \left(1+d_k\left(T_n,T_{n-1}\right)\right)} \\
	&= \sum^{\infty}_{n=1}\left( \sum^{l_n}_{k=0}\frac{d_k\left( T_n, T_{n-1}\right)}{2^k \cdot \left(1+d_k\left( T_n,T_{n-1}\right)\right)} + \sum^{\infty}_{k=l_n +1}\frac{d_k\left(T_n, T_{n-1}\right)}{2^k \cdot \left(1+d_k\left(T_n,T_{n-1}\right)\right)}\right)
	\end{align*}
	As seen in equation (\ref{est4}) $d_k\left(T_n,T_{n-1}\right)\leq \frac{1}{4 \cdot l_n}$ for every $k\leq l_n$. Hereby, it follows further:
	\begin{align*}
	\sum^{\infty}_{n=1} d_{\infty}\left( T_n, T_{n-1}\right) & \leq \sum^{\infty}_{n=1}\left( \frac{1}{4l_n} \cdot \sum^{l_n}_{k=0}\frac{1}{2^k} + \sum^{\infty}_{k=l_n +1}\frac{d_k\left( T_n, T_{n-1}\right)}{2^k \cdot \left(1+d_k\left( T_n, T_{n-1}\right)\right)}\right) \\
	& \leq 2 \cdot \sum^{\infty}_{n=1} \frac{1}{4l_n} + \sum^{\infty}_{n=1}\sum^{\infty}_{k=l_n +1} \frac{1}{2^k}.
	\end{align*}
	Because of $\sum^{\infty}_{k=l_n +1} \frac{1}{2^k} = \left(\frac{1}{2}\right)^{l_n} \leq \frac{1}{l_n}$ we conclude:
	\begin{equation*}
	\sum^{\infty}_{n=1} d_{\infty}\left(T_n, T_{n-1}\right)\leq \sum^{\infty}_{n=1} \frac{1}{l_n} + \sum^{\infty}_{n=1}\frac{1}{l_n} < 2 \cdot \varepsilon.
	\end{equation*}
	Hence, using equation (\ref{est3}) we obtain the desired estimate $d_{\infty}\left(T, R_{A,B}\right) < 4 \cdot \varepsilon$.
	\item Again with the help of Lemma \ref{lem:konj} we compute for every $i \in \mathbb{N}$:
\begin{align*}
d_0\left( T^{m}_{i},T^{m}_{i-1}\right) & = d_0\left(H^{-1}_{i} \circ R_{m \cdot \alpha_{i+1}, m \cdot \alpha^{\prime}_{i+1}} \circ H_{i},  H^{-1}_{i} \circ R_{m \cdot \alpha_{i}, m \cdot \alpha^{\prime}_{i}} \circ H_{i}\right) \\
& \leq ||| H_{i} |||_1 \cdot m  \cdot \max\left( \left|\alpha_{i+1}-\alpha_{i}\right|, \left| \a^{\prime}_{i+1}- \a^{\prime}_i \right| \right).
\end{align*}
Since $m \leq q_{n+1}q^{\prime}_{n+1} \leq q_iq^{\prime}_i$ we conclude for every $i > n$:
\begin{equation*}
d_0 \left( T^{m}_{i}, T^{m}_{i-1}\right) \leq ||| H_i |||_1 \cdot m \cdot \frac{1}{4 \cdot l_i \cdot C_{l_i} \cdot q_iq^{\prime}_i \cdot ||| H_i |||^{l_i +1}_{l_i +1}} < \frac{m}{q_iq^{\prime}_i} \cdot \frac{1}{4l_i} \leq \frac{1}{4l_i}.
\end{equation*}
Thus, for every $m\leq q_{n+1}q^{\prime}_{n+1}$ we get the claimed result:
\begin{equation*}
d_0\left(T^{m}, T^{m}_{n}\right) \leq \lim_{k \rightarrow \infty} \sum^{k}_{i=n+1} d_0\left(T^{m}_{i}, T^{m}_{i-1}\right) < \sum^{\infty}_{i=n+1} \frac{1}{4l_i} < \frac{\varepsilon_n}{4}.
\end{equation*}
\item  By equation (\ref{est4}) we have 
\begin{equation*}
    d_k \left( T_n, T \right) \leq \sum^{\infty}_{m=n+1} d_k \left( T_m, T_{m-1} \right) \leq \sum^{\infty}_{m=n+1} \frac{1}{4 \cdot l_m}.
\end{equation*}
for every $k \leq l_n$. On the other hand, we have noticed in Remark \ref{rem:identity} that $T_n = R_{\a_{n+1}, \a^{\prime}_{n+1}}$ outside of $\mathscr{R}_n$. Obviously this yields $D_{\vec{a}} T_n = D_{\vec{a}} R_{\a, \a^{\prime}}$ outside of $\mathscr{R}_n$ for any multiindex $\vec{a} \in \N^d_0$ of order $1 \leq \abs{\vec{a}}$. Altogether we obtain
\begin{equation*}
     \abs{ \left[D_{\vec{a}} \left( T - R_{\a, \a^{\prime}}\right)\right]_i \left( \vec{x} \right) } \leq \sum^{\infty}_{m=n+1} \frac{1}{4 \cdot l_m} 
\end{equation*}
for every $\vec{x} \in \T^2 \times [0,1]^{d-2} \setminus \mathscr{R}_n$ and any $\vec{a} \in \N^d_0$ of order $1 \leq \abs{\vec{a}} \leq l_n$. By our assumption on the sequence $\left(l_n\right)_{n \in \N}$ we get
\begin{equation*}
     \abs{ \left[D_{\vec{a}} \left( T - R_{\a, \a^{\prime}}\right)\right]_i \left( \vec{x} \right) }< \rho_k \left( \vec{x} \right) 
\end{equation*} 
for every $\vec{x} \in  \mathscr{R}_{n+1} \setminus \mathscr{R}_n$ and any $\vec{a} \in \N^d_0$ of order $1 \leq \abs{\vec{a}} = k \leq l_n$. Since this holds true for every $n \in \mathbb{N}$, we conclude
\begin{equation*}
     \abs{ \left[D_{\vec{a}} \left( T - R_{\a, \a^{\prime}}\right)\right]_i \left( \vec{x} \right) }< \rho_k \left( \vec{x} \right) 
\end{equation*}
for every $\vec{x} \in \T^2 \times [0,1]^{d-2} \setminus \mathscr{R}_n$ and any $\vec{a} \in \N^d_0$ of order $1 \leq \abs{\vec{a}} = k \leq l_n$. Hence, $T\in \text{Diff}^{\infty}_{\rho, \a, \a^{\prime}} \left( \T^2 \times [0,1]^{d-2} \right)$.
\end{enumerate}
\end{pr}

We show that we can satisfy the conditions from this Lemma in our constructions:

\begin{lem} \label{lem:conv}
Let $\left( A,B\right) \in \T^2$ and $\varepsilon>0$ be arbitrary. Moreover, let $\rho = \left( \rho_n \right)_{n \in \mathbb{N}_0}$ be an admissible sequence. There exist sequences 
\begin{equation} \label{eq:alpha}
\alpha_{n+1} = \frac{p_{n+1}}{q_{n+1}} = \a_n + \frac{1}{q_{n+1}}
\end{equation}
and 
\begin{equation} \label{eq:alpha2}
\a^{\prime}_{n+1} = \frac{p^{\prime}_{n+1}}{q^{\prime}_{n+1}} = \a^{\prime}_n + \frac{1}{\bar{q}^{\prime}_{n+1}} + D_n
\end{equation}
of rational numbers with $p_{n+1}$ and $q_{n+1}$ relatively prime as well as $p^{\prime}_{n+1}$ and $q^{\prime}_{n+1}$ relatively prime with properties (\ref{eq:div}), (\ref{eq:prime}), (\ref{eq:relation}) and
\begin{equation} \tag{D} \label{eq:D}
0<D_n < \frac{1}{4n \cdot \left(\bar{q}^{\prime}_{n+1}\right)^{d+1}},
\end{equation}
\begin{equation} \tag{E} \label{eq:large}
q^{\prime}_{n+1} >\bar{q}^{\prime}_{n+1} >  q_{n+1} > 4 \cdot n^{d-1} \cdot \left(q_n q^{\prime}_n \right)^{d+1},
\end{equation}
\begin{equation} \tag{F} \label{eq:largeprime}
q^{\prime}_{n+1} > 4 \cdot \bar{q}^{\prime}_{n+1},
\end{equation}
\begin{equation} \tag{G} \label{eq:diam}
\max \left( \left\|DH_{n} \right\|_0 , \left\|DH^{-1}_{n} \right\|_0 \right)< \frac{q_{n+1} }{2\sqrt{d} \cdot (n+1)^2},
\end{equation}
\begin{equation} \tag{H} \label{eq:smalTranslCond}
    \text{property (\ref{eq:small translation}) in the proof of Lemma \ref{lem:error3}},
\end{equation}
such that our sequence of constructed diffeomorphisms $T_n$ converges in the Diff$^{\infty}(M_0)$-topology to a diffeomorphism $T$ for which $d_{\infty} \left(T, R_{A,B}\right)< 4 \varepsilon$ and $T \in \text{Diff}^{\infty}_{\rho, \a, \a^{\prime}} \left( \T^2 \times [0,1]^{d-2} \right)$ hold true, where $\a = \lim_{n \rightarrow \infty} \a_n$ and $\a^{\prime}= \lim_{n \rightarrow \infty} \a^{\prime}_n$. Additionally, we have $d_0 \left( T^{m_n}, T^{m_n}_n \right) < \frac{\varepsilon_n}{4}$ for every number $m_n \leq q_{n+1}q^{\prime}_{n+1}$ and $n \in \mathbb{N}$.
\end{lem}

\begin{pr}
First of all, we choose numbers $l_n$ satisfying the requirements of Lemma \ref{lem:convgen} and we put 
\begin{equation*}
\alpha_{n+1} = \alpha_n +\frac{1}{k_n \cdot q_n \cdot q^{\prime}_n},
\end{equation*}
Then we have  $q_{n+1}=k_n \cdot q_n \cdot q^{\prime}_n$. In particular, our divisibility condition (\ref{eq:div}) is fulfilled. Moreover, we recall that the construction of the conjugation map $H_n$ does not involve $k_n$. Hence, we can choose the integer $k_n$ sufficiently large such that  conditions (\ref{eq:diam}) and (\ref{eq:smalTranslCond}), and the last inequality in (\ref{eq:large}), as well as
\begin{equation*}
\left| \alpha_{n+1} - \alpha_n \right| \leq \frac{1}{8 \cdot l_n \cdot C_{l_n} \cdot q_n \cdot q^{\prime}_n \cdot ||| H_n |||^{l_n+1}_{l_n+1}},
\end{equation*}
are satisfied. 

In the next step, we put $\bar{q}^{\prime}_{n+1} = q_{n+1}+q_n q^{\prime}_n$ and $\bar{\alpha}^{\prime}_{n+1} = \a^{\prime}_n + \frac{1}{\bar{q}^{\prime}_{n+1}}$. Then we choose a rational number $\alpha^{\prime}_{n+1}= \frac{p^{\prime}_{n+1}}{q^{\prime}_{n+1}}> \bar{\a}^{\prime}_{n+1}$ in a $\frac{1}{4n \cdot \left(\bar{q}^{\prime}_{n+1}\right)^{d+1}}$-neighbourhood of $\bar{\alpha}^{\prime}_{n+1}$ where $p^{\prime}_{n+1}$ and $q^{\prime}_{n+1}$ are relatively prime and $q^{\prime}_{n+1}> 4 \cdot \bar{q}^{\prime}_{n+1} > q_{n+1}$ is relatively prime to $q_{n+1}$. We denote this number by
\begin{equation*}
\a^{\prime}_{n+1} = \a^{\prime}_n + \frac{1}{\bar{q}^{\prime}_{n+1}}+\frac{p^{\prime \prime}_{n+1}}{q^{\prime \prime}_{n+1}} = \a^{\prime}_n + \frac{1}{\bar{q}^{\prime}_{n+1}}+ D_n.
\end{equation*}
In particular, we observe
\begin{equation*}
\left| \alpha^{\prime}_{n+1} - \alpha^{\prime}_n \right| \leq \frac{1}{4 \cdot l_n \cdot C_{l_n} \cdot q_n \cdot q^{\prime}_n  \cdot ||| H_n |||^{l_n+1}_{l_n+1}}
\end{equation*}
and so the conditions of Lemma \ref{lem:convgen} are satisfied.
\end{pr}

\begin{rem} \label{rem:rigid}
By the property $d_0 \left( T^{m_n}, T^{m_n}_n \right) < \frac{\varepsilon_n}{4}$ for every number $m_n \leq q_{n+1}q^{\prime}_{n+1}$ and $T^{q_{n+1}q^{\prime}_{n+1}}_n = H^{-1}_n \circ R^{q_{n+1}q^{\prime}_{n+1}}_{\a_{n+1}, \a^{\prime}_{n+1}} \circ H_n = \text{id}$ the constructed diffeomorphisms are uniformly rigid. Hence, they have topological entropy zero by Theorem \ref{thm:Glasner}.
\end{rem}

\section{Proof of the LB property for $T\times T$} \label{section:LB}

\subsection{Definition of towers} \label{subsection:towers}
In order to define the base of the first tower we consider the sets
\begin{equation*}\begin{split}
& C^{(1)}_{k(i,j)} \\
= & \left( \frac{a_{n,1}(i,j)}{q_n} + \frac{2i+\tilde{\varepsilon}_n}{2q_n q^{\prime}_n}, \frac{a_{n,1}(i,j)}{q_n} +  \frac{2i+\tilde{\varepsilon}_n}{2q_n q^{\prime}_n}+\frac{q_n q^{\prime}_n}{q_{n+1}} \right) \\
& \times \left[\frac{a^{\prime}_{n,1}(i,j)}{q^{\prime}_n}+ \frac{2j+3\varepsilon_n}{2q_n q^{\prime}_n}-k \cdot \Delta_n, \frac{a^{\prime}_{n,1}(i,j)}{q^{\prime}_n}+ \frac{2j+3\varepsilon_n}{2q_n q^{\prime}_n}-k \cdot \Delta_n + \frac{1-4\varepsilon_n}{2q_nq_n'}\right] \times \left[\varepsilon_n, 1-\varepsilon_n\right]^{d-2}
\end{split} \end{equation*}
for $0 \leq i < q^{\prime}_n$ and $0 \leq j < q_n$, where $\Delta_n=\frac{1}{\bar{q}^{\prime}_{n+1}}-\left(m_n - 1 \right) \cdot D_n>0$ as in Section \ref{section:comb}, $\tilde{\varepsilon}_n=\varepsilon_n (nq_n q_n')^{2-d}$ as in Subsection \ref{subsection:h1}, and $\varepsilon_n$ is defined in 
equation (\ref{eq:eps}). By condition (\ref{eq:large}), we have
\begin{equation} \label{eq:towerbase}
\frac{q_nq_n'}{\overline{q}'_{n+1}}<\frac{q_nq_n'}{q_{n+1}}<\frac{\varepsilon_n}{8n^{d-2}(q_nq_n')^{d-1}}=\frac{\tilde{\varepsilon}_n}{8q_nq_n'}\le \frac{\varepsilon_n}{8q_nq_n'}.
\end{equation}
Thus $q_n q^{\prime}_n \cdot \Delta_n < \frac{\varepsilon_n}{8 q_n q^{\prime}_n}$. Therefore 
$C^{(1)}_{k(i,j)}\subset I^{(n,1)}_{j_1,j_2,0}$, where $i,j,j_1,j_2$ are related as in equation (\ref{eq:j1j2}) with $b=a_{n,1}(i,j)$ and $b'=a'_{n,1}(i,j)$. By Definitions \ref{dfn:goodh2} and \ref{dfn:goodh}, $I^{(n,1)}_{j_1,j_2,0}$ is in the ``good domain'' $\mathscr{D}(h_n^{-1})$. 

\begin{rem} \label{rem:mapTower1}
The sets $C^{(1)}_{k}$ are defined in such a way that $R^{m_n - 1}_{\a_{n+1}, \a^{\prime}_{n+1}}\left(C^{(1)}_{k}\right) = C^{(1)}_{k+1}$ for $0 \leq k < q_n q^{\prime}_n-1$.
\end{rem}

We set
\begin{equation*}
\tilde{B}^{(n)}_{0,1} =\bigcup_{0 \leq i < q^{\prime}_n} \bigcup_{0 \leq j < q_n} C^{(1)}_{k(i,j)}.
\end{equation*}
By construction of our conjugation map $h_n$ and applying Remark \ref{rem:hn2behavior} with $\ell=0$, we have
\begin{equation*}
h^{-1}_n \left( \tilde{B}^{(n)}_{0,1}  \right) \subset \left[0, \frac{1}{q_n}\right] \times \left[0, \frac{1}{q^{\prime}_n}\right] \times \prod^d_{i=3} \left[\frac{\varepsilon_n}{nq_nq_n'},\frac{1-\varepsilon_n}{nq_nq_n'}\right]
\end{equation*}
Thus $h^{-1}_n \left( \tilde{B}^{(n)}_{0,1}  \right)$ has diameter less than $\sqrt{d}/q_n.$ This is a special case of the analysis in Lemma \ref{lem:xi}, in which it is shown that the diameter of most of the levels of the towers goes to 0 as $n$ goes to infinity. 
We define the base of the first tower as
\begin{equation*}
B^{(n)}_{0,1} = H^{-1}_n \left( \tilde{B}^{(n)}_{0,1} \right).
\end{equation*}

Similarly, we define the base of the second tower using sets
\begin{equation*}\begin{split}
& C^{(2)}_{k(i,j)} \\
= & \left[ \frac{a_{n,2}(i,j)}{q_n} + \frac{2i+3\varepsilon_n}{2q_n q^{\prime}_n}+\frac{k}{q_{n+1}}, \frac{a_{n,2}(i,j)}{q_n} +  \frac{2i+3\varepsilon_n}{2q_n q^{\prime}_n}+\frac{k}{q_{n+1}} +\frac{1-4\varepsilon_n}{2q_nq_n'} \right] \\
& \times \Bigg[\frac{a^{\prime}_{n,2}(i,j)}{q^{\prime}_n}+ \frac{2j+\tilde{\varepsilon}_n}{2q_n q^{\prime}_n}+k m_n D_n + \frac{2\varepsilon_n}{\bar{q}^{\prime}_{n+1}}, \frac{a^{\prime}_{n,2}(i,j)}{q^{\prime}_n}+ \frac{2j+\tilde{\varepsilon}_n}{2q_n q^{\prime}_n}+ k  m_n D_n+ \frac{q_nq^{\prime}_n - 2\varepsilon_n}{\bar{q}^{\prime}_{n+1}}\Bigg]  \\
& \times \left[\varepsilon_n, 1-\varepsilon_n\right]^{d-2}
\end{split} \end{equation*}
for $0 \leq i < q^{\prime}_n$ and $0 \leq j < q_n$. By condition (\ref{eq:D}) we have $\frac{1}{q_{n+1}}-m_nD_n >0$. From this inquality and (\ref{eq:towerbase}) it follows that $C^{(2)}_{k(i,j)}\subset I^{(n,2)}_{j_1,j_2,0}$, where $i,j,j_1,j_2$ are related as in equation (\ref{eq:j1j2}) with $b=a_{n,2}(i,j)$ and $b'=a'_{n,2}(i,j)$. 
\begin{rem} \label{rem:mapTower2}
This time, the sets $C^{(2)}_{k}$ are defined in such a way that $R^{m_n}_{\a_{n+1}, \a^{\prime}_{n+1}}\left(C^{(2)}_{k}\right) = C^{(2)}_{k+1}$ for $0 \leq k < q_n q^{\prime}_n-1$.
\end{rem}
Once again, we set
\begin{equation*}
\tilde{B}^{(n)}_{0,2} = \bigcup_{0 \leq i < q^{\prime}_n} \bigcup_{0 \leq j < q_n} C^{(2)}_{k(i,j)}.
\end{equation*}
As before, we apply Remark \ref{rem:hn2behavior} to obtain
\begin{equation*}
h^{-1}_n \left( \tilde{B}^{(n)}_{0,2} \right) \subset \left[0, \frac{1}{q_n}\right] \times \left[0, \frac{1}{q^{\prime}_n}\right] \times \prod^d_{i=3} \left[\frac{\varepsilon_n}{nq_nq_n'},\frac{1-\varepsilon_n}{nq_nq_n'}\right].
\end{equation*}
Hereby, we define the base of the second tower as
\begin{equation*}
B^{(n)}_{0,2} = H^{-1}_n \left( \tilde{B}^{(n)}_{0,2} \right).
\end{equation*}

Moreover, we recall the number
\begin{equation*}
m_n = \frac{q_{n+1}}{q_n q^{\prime}_n}+1 = \frac{\bar{q}^{\prime}_{n+1}}{q_n q^{\prime}_n} \in \N,
\end{equation*}
and we define the following sets:
\begin{align*}
&\tilde{B}_{i,1}^{(n)}:=R^i_{\alpha_{n+1},\alpha'_{n+1}}\left(\tilde{B}^{(n)}_{0,1}\right) \text{ for } i=0,1,...,m_n-2, \\
&\tilde{B}_{i,2}^{(n)}:=R^i_{\alpha_{n+1},\alpha'_{n+1}}\left(\tilde{B}^{(n)}_{0,2}\right) \text{ for } i=0,1,...,m_n-1.
\end{align*}
The levels of the towers are defined to be 
\begin{align*}
&B^{(n)}_{i,1} \coloneqq T^{i}_n\left(B^{(n)}_{0,1}\right)=H_n^{-1}\left(\tilde{B}^{(n)}_{i,1}\right) \text{ for } i=0,1,...,m_n-2, \\
&B^{(n)}_{i,2} \coloneqq T^{i}_n\left(B^{(n)}_{0,2}\right)= H_n^{-1}\left(\tilde{B}^{(n)}_{i,2}\right) \text{ for } i=0,1,...,m_n-1.
\end{align*}
As we will see in Lemma \ref{lem:disj}, the levels are pairwise disjoint. \\

For $0\le j_{1}<q_{n},0\le j_{2}<q_{n}',0\le\varepsilon<\frac{1}{4},$
we define
\begin{align*}
   P_{j_{1},j_{2},\varepsilon}^{(1)}= & \left\{ (\theta_{1},\theta_{2})\in\mathbb{T}^{2}:0\le\theta_{1}\le1,\theta_{1}+\frac{j_{2}}{q_{n}'}-\frac{j_{1}}{q_{n}}+\frac{\varepsilon}{2q_{n}q_{n}'}<\theta_{2}<\theta_{1}+\frac{j_{2}}{q_{n}'}-\frac{j_{1}}{q_{n}}+\frac{1-\varepsilon}{2q_{n}q_{n}'}\right\} \\
& \times[\varepsilon_{n},1-\varepsilon_{n}]^{d-2}
\end{align*}
and
\begin{align*}
P_{j_{1},j_{2},\varepsilon}^{(2)}= & \left\{ (\theta_{1},\theta_{2})\in\mathbb{T}^{2}:0\le\theta_{2}\le1,\theta_{2}+\frac{j_{1}}{q_{n}}-\frac{j_{2}}{q_{n}'}+\frac{\varepsilon}{2q_{n}q_{n}'}<\theta_{1}<\theta_{2}+\frac{j_{2}}{q_{n}'}-\frac{j_{1}}{q_{n}}+\frac{1- \varepsilon}{2q_{n}q_{n}'}\right\} \\
& \times[\varepsilon_{n},1-\varepsilon_{n}]^{d-2}.
\end{align*}
In case $\varepsilon=0,$ we let $P_{j_{1},j_{2}}^{(s)}=P_{j_{1},j_{2},0}^{(s)}.$
Since $q_{n}$ and $q_{n}'$ are relatively prime, the sets in the
collection $\{P_{j_{1},j_{2}}^{(s)}:0\le j_{1}<q_{n},0\le j_{2}<q_{n}',s=1,2\}$
are pairwise disjoint, and the union of their closures is $\mathbb{T}^{2}\times[\varepsilon_{n},1-\varepsilon_{n}]^{d-2}.$
Moreover, the boundaries of these sets within the $\mathbb{T}^{2}$
factor form a collection of $2q_{n}q_{n}'$ equally spaced lines of
slope $1.$ The lower boundary of the $\mathbb{T}^{2}$ factor of
$P_{j_{1},j_{2}}^{(1)}$ is the same as the left boundary of the $\mathbb{T}^{2}$
factor of $P_{j_{1},j_{2}}^{(2)},$ and these boundaries pass through
the point $(\theta_{1},\theta_{2})=(\frac{j_{1}}{q_{n}},\frac{j_{2}}{q_{n}'}).$
Each scaled $\widetilde{G}_{1}$ is contained in the closure of the $\mathbb{T}^2$ factor of some
$P_{j_{1},j_{2},\varepsilon_{n}}^{(1)}$ and has upper and lower boundaries
contained in the upper and lower boundaries of the $\mathbb{T}^2$ factor of that $P_{j_{1},j_{2},\varepsilon_{n}}^{(1)}$.
Similarly, each scaled $\widetilde{G}_{2}$ is contained in the closure
of the $\mathbb{T}^2$ factor of some $P_{j_{1},j_{2},\varepsilon_{n}}^{(2)}$ and has right and left
boundaries contained in the right and left boundaries of the $\mathbb{T}^2$ factor of that $P_{j_{1},j_{2},\varepsilon_{n}}^{(2)}.$
By a scaled version of $\widetilde{G}_{s}$ we mean a set of the form 
\[
{\lambda}^{-1}\left(\widetilde{G}_{s}\right)+\left(\frac{\tau_{1}}{q_{n}q'_{n}},\frac{\tau_{2}}{q_{n}q_{n}'}\right),
\]
with $\lambda = q_nq'_n$, for some $0\le \tau_1,\tau_2 <q_nq_n'$, where we use the definition of $\widetilde{G}_s$ given in equation (\ref{eq:Gtilde}).

In the proof of Lemma \ref{lem:disj}, we will make use of inequality (\ref{eq:towerbase}) and the following inequality,
which is an immediate consequence of the definitions and properties
(\ref{eq:D}) and (\ref{eq:large}) in Section \ref{section:conv}:

 \begin{equation}
-\frac{\varepsilon_{n}}{8m_{n}q_{n}q_{n}'}<q_{n}q_{n}'\left(\frac{1}{\overline{q}_{n+1}'}+D_{n}-\frac{1}{q_{n+1}}\right)<0.\label{eq:IncrementInequality}
\end{equation}

\begin{lem} \label{lem:disj}
For $0\le i<q_{n}q_{n}',$ the sets in the family $\{R_{\alpha_{n+1},\alpha_{n+1}'}^{i+jq_{n}q_{n}'}C_{0}^{(1)}:0\le j<m_{n}-1\}$
are pairwise disjoint subsets of $P_{ip_{n},ip_{n}',\varepsilon_{n}}^{(1)}$,
where $(ip_{n},ip_{n}')$ is taken mod $(q_{n},q_{n}')$. Similarly,
for $0\le i<q_{n}q_{n}',$ the sets in the family $\{R_{\alpha_{n+1},\alpha_{n+1}'}^{i+jq_{n}q_{n}'}C_{0}^{(2)}:0\le j<m_{n}\}$
are pairwise disjoint subsets of $P_{ip_{n},ip_{n}',\varepsilon_{n}}^{(2)}$.
Consequently, the sets in the family $\{R_{\alpha_{n+1},\alpha_{n+1}'}^{k}C_{0}^{(1)}:0\le k<(m_{n}-1)q_{n}q_{n}'\}\cup\{R_{\alpha_{n+1},\alpha_{n+1}'}^{k}C_{0}^{(2)}:0\le k<m_nq_{n}q_{n}'\}$
are pairwise disjoint, and it follows that the sets in the family
$\{B_{i,1}^{(n)},B_{k,2}^{(n)}:i=0,1,\dots,m_{n}-2;k=0,1,\dots,m_{n}-1\}$
are pairwise disjoint.
\end{lem}
\begin{pr}
We start by proving the first statement in the case $i=0.$ We only
need to consider the $\mathbb{T}^{2}$ factors of the sets, and references
to upper, lower, right, or left will be with respect to the $(\theta_{1},\theta_{2})$
coordinate system in $\mathbb{T}^{2}.$ The upper left corner of $C_{0}^{(1)}$
is $(\frac{\tilde{\varepsilon}_n}{2q_nq_n'},\frac{1-\varepsilon_n}{2q_{n}q_{n}'}),$ and the lower right corner
of $C_{0}^{(1)}$ is $(\frac{\tilde{\varepsilon}_n}{2q_nq_n'}+\frac{q_{n}q_{n}'}{q_{n+1}},\frac{3\varepsilon_{n}}{2q_{n}q_{n}'}).$
Since $\tilde{\varepsilon}_n\le\varepsilon_n$, we have $\frac{1-2\varepsilon_{n}}{2q_{n}q_{n}'}<\theta_{2}-\theta_{1}<\frac{1-\varepsilon_{n}}{2q_{n}q_{n}'}$
for the upper left corner, and by (\ref{eq:towerbase}), we have $\frac{7\epsilon_{n}}{8q_{n}q_{n}'}<\theta_{2}-\theta_{1}<\frac{1-\varepsilon_n}{2q_nq_n'}$
for the lower right corner. Both of these points are in $P_{0,0,\varepsilon_{n}}^{(1)}$
and consequently $C_{0}^{(1)}\subset P_{0,0, \varepsilon_{n}}^{(1)}.$
When we apply $R_{\alpha_{n+1},\alpha_{n+1}'}^{q_{n}q_{n}'}$ to a
point $(\theta_{1},\theta_{2}),$ the first coordinate increases by
$q_{n}q_{n}'\alpha_{n+1}=q_{n}q_{n}'(\frac{p_{n}}{q_{n}}+\frac{1}{q_{n+1}})=\frac{q_{n}q_{n}'}{q_{n+1}}$
(mod 1), which is the $\theta_{1}$-width of $C_{0}^{(1)},$ and the
second coordinate increases by $q_{n}q_{n}'\alpha_{n+1}'=q_{n}q_{n}'(\frac{p_{n}'}{q_{n}'}+\frac{1}{\overline{q}_{n+1}'}+D_{n})=\frac{q_{n}q_{n}'}{\overline{q}_{n+1}'}+q_{n}q_{n}'D_{n}$
(mod 1). Thus the change in $\theta_{2}-\theta_{1}$ when applying
$R_{\alpha_{n+1},\alpha_{n+1}'}^{q_{n}q_{n}'}$ is $q_{n}q_{n}'(\frac{1}{\overline{q}_{n+1}'}+D_{n}-\frac{1}{q_{n+1}}).$
Therefore, by (\ref{eq:IncrementInequality}) the lower right corner
of $C_{0}^{(1)}$ moves closer (in the $\theta_{2}$ direction) to
the lower boundary of $P_{0,0,\varepsilon_{n}}^{(1)}$ by an amount smaller
than $\frac{\varepsilon_{n}}{8m_{n}q_{n}q_{n}'}$ under one application
of $R_{\alpha_{n+1},\alpha_{n+1}'}^{q_{n}q_{n}'}$, and it moves closer
by an amount smaller than $\frac{\varepsilon_{n}}{8q_{n}q_{n}'}$ after
$m_{n}-2$ applications of $R_{\alpha_{n+1},\alpha_{n+1}'}^{q_{n}q_{n}'}$.
Since the lower right corner of $C_{0}^{(1)}$ is above the lower
boundary of $P_{0,0,\varepsilon_{n}}^{(1)}$ by more than $\frac{7\varepsilon_{n}}{8q_{n}q_{n}'}-\frac{\varepsilon_{n}}{2q_{n}q_{n}'}$=$\frac{3\varepsilon_{n}}{8q_{n}q_{n}'}$,
it is still above the lower boundary of $P_{0,0,\varepsilon_{n}}^{(1)}$ after $m_{n}-2$ applications of $R_{\alpha_{n+1},\alpha_{n+1}'}^{q_{n}q_{n}'}$.
Therefore all of the sets $R_{\alpha_{n+1},\alpha_{n+1}'}^{jq_{n}q_{n}'}C_{0}^{(1)},0\le j<m_{n}-1,$
are contained in $P_{0,0,\varepsilon_{n}}^{(1)}.$ They are disjoint
because $C_{0}^{(1)}$ moves to the right by an amount exactly equal
to its $\theta_{1}$-width under application of $R_{\alpha_{n+1},\alpha_{n+1}'}^{q_{n}q_{n}'}$and
the combined width of these $m_{n}-1$ sets is $(m_{n}-1)\frac{q_{n}q_{n}'}{q_{n+1}}=1.$
(See Figure \ref{fig:tower}, which shows those images of $C_{0}^{(1)}$ under applications
of powers of $R_{\alpha_{n+1},\alpha_{n+1}'}^{q_{n}q_{n}'}$ that
lie in one of the scaled versions of $\widetilde{G}_{1}.$ To indicate
all of the $R_{\alpha_{n+1},\alpha_{n+1}'}^{jq_{n}q_{n}'}C_{0}^{(1)},0\le j<m_{n}-1,$
the figure would have to be continued for the full width of $\mathbb{T}^{2}.$) 

Now consider an arbitrary $i$ such that $0\le i<q_{n}q_{n}'.$ Under
application of $R_{\alpha_{n+1},\alpha_{n+1}'}^{i}$ to a point $(\theta_{1},\theta_{2}),$
the first coordinate increases by $i(\frac{p_{n}}{q_{n}}+\frac{1}{q_{n+1}})$
and the second coordinate increases by $i(\frac{p_{n}'}{q_{n}'}+\frac{1}{\overline{q}_{n+1}'}+D_{n}).$
Adding $i(\frac{p_{n}}{q_{n}},\frac{p_{n}'}{q_{n}'})$ to points in
$P_{0,0,\varepsilon_{n}}^{(1)}$ moves them to $P_{ip_{n},ip_{n}',\varepsilon_{n}}^{(1)}$,
and the additional increment of $i(\frac{1}{q_{n+1}},\frac{1}{\overline{q}_{n+1}'}+D_{n})$
moves the lower right corner point of $C_{0}^{(1)}$ closer (in vertical
distance) to the lower boundary of $P_{ip_{n},ip_{n}',\varepsilon_{n}}^{(1)}$
than it was to the lower boundary of $P_{0,0,\varepsilon_{n}}^{(1)}$
by $i(\frac{1}{\overline{q}_{n+1}'}+D_{n}-\frac{1}{q_{n+1}}).$ Subsequent
applications of $R_{\alpha_{n+1},\alpha_{n+1}'}^{jq_{n}q_{n}'},0\le j<m_{n}-1$
move the lower right corner point even closer to the lower boundary
of $P_{ip_{n},ip_{n},\varepsilon_{n}}^{(1)}$ by an additional amount of at most
$(m_{n}-2)q_{n}q_{n}'(\frac{1}{\overline{q}_{n+1}'}+D_{n}-\frac{1}{q_{n+1}}).$
Thus the total vertical distance that the lower right corner point
of $R_{\alpha_{n+1},\alpha_{n+1}'}^{i+jq_{n}q_{n}'}C_{0}^{(1)}$,
$0\le j<m_{n}-1$ is above the lower boundary of $P_{ip_{n},ip_{n}',\varepsilon_{n}}^{(1)}$
is still positive by the same estimate as in the case $i=0$. Therefore
all of the sets $R_{\alpha_{n+1},\alpha_{n+1}'}^{i+jq_{n}q_{n}'}C_{0}^{(1)}$,
$0\le j<m_{n}-1$ are contained in $P_{ip_{n},ip_{n}',\varepsilon_{n}}^{(1)}$.
They are disjoint, as in the case $i=0.$ This completes the proof
of the first statement in the Lemma. 

Analogously, we prove the second statement of the Lemma (see Figure
\ref{fig:tower} as well). For successive values of $j$, in the range $0\le j<m_{n}-1,$
the sets $R^{jq_{n}q_{n}'}C_{0}^{(2)}$ are almost adjacent, but there
is a gap (too small to be visible in Figure \ref{fig:tower}) between them. The
$\theta_{2}$-height of these sets is $\frac{q_{n}q_{n}'-4\varepsilon_{n}}{\overline{q}_{n+1}'}$
and the gap (in the $\theta_{2}$ direction) between the sets is $\frac{4\varepsilon_{n}}{\overline{q}_{n+1}'}+q_{n}q_{n}'D_{n}.$
Thus the top boundary of the set $R_{\alpha_{n+1},\alpha_{n+1}'}^{(m_{n}-1)q_{n}q_{n}'}C_{0}^{(2)}$
is $m_{n}\left(\frac{q_{n}q_{n}'}{\overline{q}_{n+1}'}\right)+(m_{n}-1)q_{n}q_{n}'D_{n}-\frac{4\varepsilon_{n}}{\overline{q}_{n+1}'}$
units above the lower boundary of $C_{0}^{(2)}.$ Since $m_{n}\left(\frac{q_{n}q_{n}'}{\overline{q}_{n+1}'}\right)=1,$
we exploit the inequality
\[
q_{n}q_{n}'\cdot(m_{n}-1)\cdot D_{n}<\frac{\varepsilon_{n}}{\overline{q}_{n+1}'},
\]
 by condition (\ref{eq:D}), to obtain the disjointness of $R_{\alpha_{n+1},\alpha_{n+1}'}^{(m_{n}-1)q_{n}q_{n}'}C_{0}^{(2)}$
from $C_{0}^{(2)}.$ The rest of the proof of the second statement
proceeds as for the first statement, making use of the fact that inequality (\ref{eq:IncrementInequality}) remains true with $\varepsilon_n$ replaced by $\tilde{\varepsilon}_n$. 

Note that for distinct values of $i$ with $0\le i<q_{n}q_{n}',$
the pairs $(ip_{n},ip_{n}')$ are distinct mod $(q_{n},q_{n}').$
Since the sets in the family $\{P_{j_{1},j_{2},\varepsilon_{n}}^{(s)}:0\le j_{1}<q_{n},0\le j_{2}<q_{n}',s=1,2\}$
are pairwise disjoint, the third statement follows as well. 
\end{pr}

\begin{figure}[hbtp] 
\begin{center}
\includegraphics[scale=0.7]{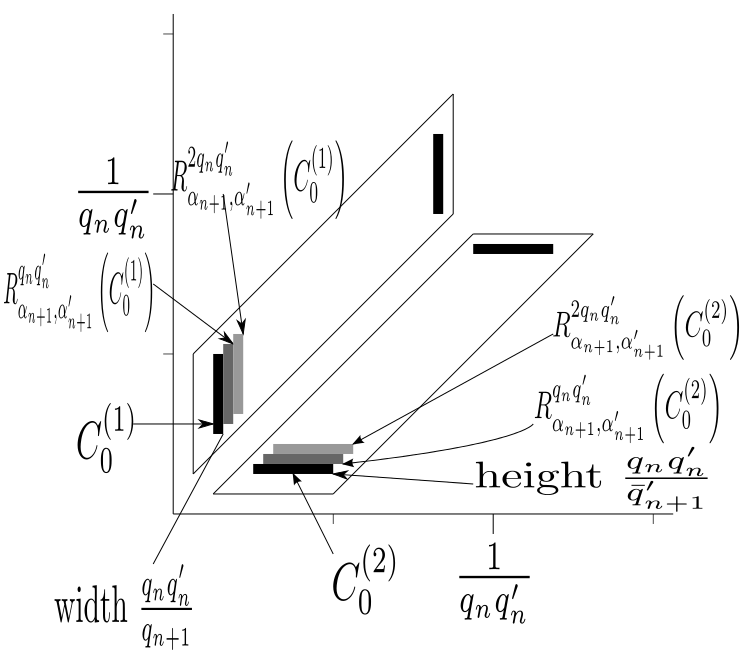}
\caption{Sketch of some parts of the levels of the two towers. The rightmost sets correspond to $R^{\frac{1-2\varepsilon_n}{q_n q^{\prime}_n }q_{n+1}}_{\a_{n+1}, \a^{\prime}_{n+1}} \left( C^{(1)}_0 \right)$ and $R^{\frac{1-2\varepsilon_n}{q_n q^{\prime}_n }\bar{q}^{\prime}_{n+1}}_{\a_{n+1}, \a^{\prime}_{n+1}} \left( C^{(2)}_0 \right)$ respectively. We point out that the levels are much smaller than the parallelograms $\tilde{G}_{s,\ell}$ of the additional subdivision in the higher dimensional case.}\label{fig:tower}
\end{center}
\end{figure}

Hence, we are able to define these two towers:
\begin{dfn}
The first tower $B^{(n)}_1$ with base $B^{(n)}_{0,1}$ and height $m_n-1$ consists of the sets $B^{(n)}_{i,1}$ for $i=0,1,...,m_n-2$.
The second tower $B^{(n)}_2$ with base $B^{(n)}_{0,2}$ and height $m_n$ consists of the sets $B^{(n)}_{i,2}$ for $i=0,1,...,m_n-1$.
\end{dfn}

In the rest of this subsection we check that these towers satisfy the requirements of the definition of an $\left(h,h+1\right)$-approximation. First of all, we notice that both towers are substantial because we have:
\begin{align*}
\mu\left(B^{(n)}_1\right)& =(m_n-1) \cdot \mu\left(B^{(n)}_{0,1}\right) = \frac{q_{n+1}}{q_n \cdot q^{\prime}_n}  \cdot q_n q^{\prime}_n \cdot \frac{q_n q^{\prime}_n}{q_{n+1}} \cdot \frac{1-4\varepsilon_n}{2q_n q^{\prime}_n}\cdot \left(1-2\varepsilon_n\right)^{d-2} \ge \frac{\left(1-4\varepsilon_n\right)^{d-1}}{2}, \\
\mu\left(B^{(n)}_2\right)& =m_n \cdot \mu\left(B^{(n)}_{0,2}\right) = \frac{\bar{q}_{n+1}^{\prime}}{q_{n}\cdot q_{n}^{\prime}}\cdot \frac{1-4\varepsilon_{n}}{2}\cdot q_{n}q_{n}^{\prime}\cdot\frac{q_{n} q_{n}^{\prime}\cdot(1-4\varepsilon_{n}/(q_{n}q_{n}'))}{\bar{q}_{n+1}^{\prime}}\cdot\frac{1}{2q_{n}q_{n}^{\prime}}\cdot \left(1-2\varepsilon_n\right)^{d-2}\\
& >\frac{\left(1-4\varepsilon_{n}\right)^{d}}{2}
\end{align*}
Using the notation from Section \ref{subsection:theory} we consider the partial partitions

\[
\begin{array}{c}
\xi_{n,1}\coloneqq\left\{ B_{i,1}^{(n)}\;:\;i=0,1,...,m_{n}-2\right\} ,\\
\xi_{n,2}\coloneqq\left\{ B_{i,2}^{(n)}\;:\;i=0,1,...,m_{n}-1\right\} ,
\end{array}
\]
and we let $\xi_{n}:=\xi_{n,1}\cup\xi_{n,2}.$ Let $\sigma_{n}$ be
the associated permutation of $\xi_{n}$ satisfying $\sigma_{n}\left(B_{i,1}^{(n)}\right)=B_{i+1,1}^{(n)}$
for $i=0,1,...,m_{n}-3$, $\sigma_{n}\left(B_{m_{n}-2,1}^{(n)}\right)=B_{0,1}^{(n)}$
as well as $\sigma_{n}\left(B_{i,2}^{(n)}\right)=B_{i+1,2}^{(n)}$
for $i=0,1,...,m_{n}-2$, $\sigma_{n}\left(B_{m_{n}-1,2}^{(n)}\right)=B_{0,2}^{(n)}$.
In Lemma \ref{lem:xi} below, we show that $\xi_{n}\rightarrow\varepsilon$
as $n\rightarrow\infty$. This is a preliminary step for obtaining
the stronger Lemma \ref{lem:eta}, which is needed for a linked approximation. 

We now introduce some notation that will be useful in the proofs of
Lemmas \ref{lem:xi} and \ref{lem:eta}. For $s=1,2,$ let $\Gamma_{n}^{(s)}$ be defined
by
\[
\begin{array}{ccc}
\Gamma_{n}^{(1)}: & = & \left\{ \left(\frac{a_{n,1}(i,j)}{q_{n}},\frac{a'_{n,1}(i,j)}{q'_{n}}\right)+\left(\frac{i}{q_{n}q_{n}'},\frac{j}{q_{n}q_{n}'}\right):0\le i<q_{n},0\le j<q_{n}'\right\} \\
 & = & \left\{ k\left(\frac{r_{n}}{q_{n}}+\frac{1}{q_{n}q_{n}'},\frac{r_{n}'}{q_{n}}+\frac{1}{q_{n}q_{n}'}\right):0\le k<q_{n}q_{n}'\right\} ,\\
\Gamma_{n}^{(2)}: & = & \left\{ \left(\frac{a_{n,2}(i,j)}{q_{n}},\frac{a'_{n,2}(i,j)}{q'_{n}}\right)+\left(\frac{i}{q_{n}q_{n}'},\frac{j}{q_{n}q_{n}'}\right):0\le i<q_{n},0\le j<q_{n}'\right\} \\
 & = & \left\{ k\left(\frac{r_{n}+p_{n}}{q_{n}}+\frac{1}{q_{n}q_{n}'},\frac{r_{n}'+p_{n}'}{q_{n}}+\frac{1}{q_{n}q_{n}'}\right):0\le k<q_{n}q_{n}'\right\} ,
\end{array}
\]
where $a_{n},a_{n}',r_{n},r_{n}'$ are as defined in Section \ref{section:comb}, and
all computations are mod 1. The equivalence of the two versions of
the definitions of $\Gamma_{n}^{(s)}$ follows from Lemma \ref{lem:combidisj}. More
concretely, $\Gamma_{n}^{(1)}$ is the set of vectors by which we
have to translate $S_{0,0}^{(1)}$ to obtain the collection of rectangles
$\left\{ S_{k}^{(1)}:0\le k<q_{n}q_{n}'\right\} ,$ that is, the grey
rectangles in Figure \ref{figure:combi}, and similarly, $\Gamma_{n}^{(2)}$ is the
set of vectors by which we have to translate $S_{0,0}^{(2)}$ to obtain
the collection of rectangles $\left\{ S_{k}^{(2)}:0\le k<q_{n}q_{n}'\right\} ,$
which are the black rectangles in Figure \ref{figure:combi}. We define the group (under
addition mod 1)

\[
\Lambda_{n}=\left\{ \left(\frac{\tau_{1}}{q_{n}q_{n}'},\frac{\tau_{2}}{q_{n}q_{n}'}\right):0\le\tau_{1},\tau_{2}<q_{n}q_{n}'\right\} .
\]
The second versions of the above definitions of $\Gamma_{n}^{(1)}$
and $\Gamma_{n}^{(2)}$ show that $\Gamma_{n}^{(1)}$ and $\Gamma_{n}^{(2)}$
are subgroups of $\Lambda_{n}.$ Moreover, it also follows from the
results in Section \ref{section:comb} that for $s=1,2$ the collection of cosets $\left\{ \Gamma_{n}^{(s)}+(\frac{i_{1}}{q_{n}},\frac{i_{2}}{q_{n}'}):0 \le i_1 < q_n, 0 \le i_2 < q'_n\right\} $
 form a partition of $\Lambda_{n}.$

For each $0\le \ell<\gamma$ (where $\gamma=(nq_{n}q_{n}')^{d-2}$
as in Subsection \ref{subsubsec:Theta}) and $s=1,2$,
let $U_{\ell}^{(s)}$ = $\left((\frac{1}{q_{n}q_{n}'})\widetilde{G}_{s,\ell}+\Gamma_{n}^{(s)}\right)\times[\varepsilon_{n},1-\varepsilon_{n}]^{d-2},$
that is, the product of all $\Gamma_{n}^{(s)}$ translates of a scaled
version of $\widetilde{G}_{s,\ell}$ with $[\varepsilon_{n},1-\varepsilon_{n}]^{d-2}.$
By Lemmas \ref{lem:prop h1} and \ref{lem:prop h2} and the observations following Lemma \ref{lem:prop h2}, for
each $s=1,2,$ and each $\ell$ with $0\le \ell<\gamma$, $U_{\ell}^{(s)}$
is contained in the good domain of $h_{n}^{-1},$ and $h_{n}^{-1}(U_{\ell}^{(1)}\cup U_{\ell}^{(2)})$
is contained in a single cuboid of dimensions $\frac{1}{q_{n}}\times\frac{1}{q_{n}'}\times\left(\frac{1}{nq_{n}q_{n}'}\right)^{d-2},$
which has diameter less than $\sqrt{d}/q_{n}.$ By the $(\frac{1}{q_{n}},\frac{1}{q_{n}'})$-equivariance
of the map $h_{n}^{-1},$ each $(\frac{i_{1}}{q_{n}},\frac{i_{2}}{q_{n}'})$
translate (in the first two coordinates) of $U_{\ell}^{(1)}\cup U_{\ell}^{(2)}$
also has $h_{n}^{-1}$ image contained in a single cuboid of the same
dimensions. By condition (\ref{eq:diam}) in Section \ref{section:conv}, it then follows that the
image under $H_{n}^{-1}=H_{n-1}^{-1}\circ h_{n}^{-1}$ of any $(\frac{i_{1}}{q_{n}},\frac{i_{2}}{q_{n}'})$
translate of $U_{\ell}^{(1)}\cup U_{\ell}^{(2)}$ has diameter less than
$1/(2n^{2}).$ In the proof of Lemma \ref{lem:eta}, we will show that for most
$i$'s the $i$th level of the combined tower is contained in one
of the translates under some $(\frac{i_{1}}{q_{n}},\frac{i_{2}}{q_{n}'})$
of some $U_{\ell}^{(1)}\cup U_{\ell}^{(2)},$ where $(i_{1},i_{2})$ and
$\ell$ depend on $i,$ while in Lemma \ref{lem:xi}, we will show that for $s=1,2$,
for most $i$'s the $i$th level of tower $s$ is contained in one
of the translates under some $(\frac{i_{1}}{q_{n}},\frac{i_{2}}{q_{n}'})$
of some $U_{\ell}^{(s)},$ where $(i_{1},i_{2})$ and $\ell$ depend on
$i$ and $s.$
\begin{lem} \label{lem:xi}
We have
\begin{equation*}
\xi_n \rightarrow \varepsilon \text{ as } n\rightarrow \infty. 
\end{equation*}
\end{lem}
\begin{pr}
The lemma is proven if we show that the partial partitions $\tilde{\xi}_{n}\coloneqq\left\{ c\in\xi_{n}\;:\;\text{diam}\left(c\right)<\frac{1}{n}\right\} $
satisfy $\mu\left(\bigcup_{c\in\tilde{\xi}_{n}}c\right)\rightarrow1$
as $n\rightarrow\infty$. In fact, we will show that
\begin{equation}
\mu\left(\bigcup_{c\in\tilde{\xi}_{n}\cap\xi_{n,1}}c\right)\rightarrow\frac{1}{2}.\label{eq:xi_1}
\end{equation}
 A similar argument will show that
\begin{equation}
\mu\left(\bigcup_{c\in\tilde{\xi}_{n}\cap\xi_{n,2}}c\right)\rightarrow\frac{1}{2}.\label{eq:xi_2}
\end{equation}
 We note that
\[
B_{i,s}^{(n)}=T_{n}^{i}\left(B_{0,s}^{(n)}\right)=H_{n}^{-1}\circ R_{\alpha_{n+1},\alpha_{n+1}^{\prime}}^{i}\circ H_{n}\left(H_{n}^{-1}\left(\tilde{B}_{0,s}^{(n)}\right)\right)=H_{n-1}^{-1}\circ h_{n}^{-1}\circ R_{\alpha_{n+1},\alpha_{n+1}^{\prime}}^{i}\left(\tilde{B}_{0,s}^{(n)}\right)
\]
for $s=1,2.$ As we observed above, for each $\ell$ with $0\le \ell<\gamma$,
and each $0\le i_{1}<q_{n},0\le i_{2}<q_{n}',$ the set $H_{n-1}^{-1}\circ h_{n}^{-1}\left(U_{\ell}^{(1)}+\left(\frac{i_{1}}{q_{n}},\frac{i_{2}}{q_{n}'}\right)\right)$
has diameter less than $1/(2n^{2}).$ To prove (\ref{eq:xi_1}) it
thus suffices to show that for most $i$'s (that is, a proportion
going to $1),$ we have 
\[
\tilde{B}^{(n)}_{i,1} = R_{\alpha_{n+1},\alpha_{n+1}^{\prime}}^{i}\left(\tilde{B}_{0,1}^{(n)}\right)\subset U_{\ell}^{(1)}+\left(\frac{i_{1}}{q_{n}},\frac{i_{2}}{q_{n}'}\right),
\]
for some $\ell$ with $0\le \ell<\gamma$, and some $0\le i_{1}<q_{n},0\le i_{2}<q_{n}',$
where $\ell,i_{1},i_{2}$ depend on $i.$ For the rest of this proof,
we do not use the last $d-2$ coordinates of the sets under discussion,
because all of the sets $U_{\ell}^{(1)}$ and all of the sets $\tilde{B}_{i,1}^{(n)}$ have factor $[\varepsilon_{n},1-\varepsilon_{n}]^{d-2}$
in the last $d-2$ coordinates. 
(However,
we cannot assume $d=2,$ because each set $\widetilde{G}_{1,\ell}$ has
$\theta_{1}$ coordinate varying in an interval of length $\frac{1-\varepsilon_{n}}{\gamma},$
where $\gamma=(nq_{n}q_{n}')^{d-2}.)$

We will prove the following
\begin{claim*}
If 
\begin{equation}
R_{\alpha_{n+1},\alpha_{n+1}^{\prime}}^{i}C_{0}^{(1)}\subset\frac{1}{q_{n}q_{n}'}\widetilde{G}_{1,\ell}+\Gamma_{n}^{(1)}+\left(\frac{i_{1}}{q_{n}},\frac{i_{2}}{q_{n}'}\right),\label{eq:ClaimHypothesis}
\end{equation}
 then
\begin{equation}
\tilde{B}_{i,1}^{(n)}\subset\frac{1}{q_{n}q_{n}'}\widetilde{G}_{1,\ell}+\Gamma_{n}^{(1)}+\left(\frac{i_{1}}{q_{n}},\frac{i_{2}}{q_{n}'}\right).\label{eq:ClaimConclusion}
\end{equation}
Moreover, (\ref{eq:ClaimHypothesis}) holds if and only if the $\theta_{1}$-factor of $R_{\alpha_{n+1},\alpha_{n+1}^{\prime}}^{i}C_{0}^{(1)}$
is contained in the projection of $\frac{1}{q_{n}q_{n}'}\widetilde{G}_{1,\ell}+\Lambda_{n}$
in the $\theta_{1}$-direction. \end{claim*}
\begin{proof}
Suppose that (\ref{eq:ClaimHypothesis}) holds. Then $R_{\alpha_{n+1},\alpha_{n+1}^{\prime}}^{i}C_{0}^{(1)}\subset\frac{1}{q_{n}q_{n}'}\widetilde{G}_{1,\ell}+(\gamma_{1},\gamma_{1}')+\left(\frac{i_{1}}{q_{n}},\frac{i_{2}}{q_{n}'}\right)$
for some $(\gamma_{1},\gamma_{1}')\in\Gamma_{n}^{(1)}.$ Since $\tilde{B}_{i,1}^{(n)}=\cup_{0\le k <q_{n}q_{n}'}R_{\alpha_{n+1},\alpha_{n+1}^{\prime}}^{i+k(m_{n}-1)}C_{0}^{(1)},$
we fix a choice of $k$ with $0\le k<q_{n}q_{n}'$ and we show that
$R_{\alpha_{n+1},\alpha_{n+1}^{\prime}}^{i+k(m_{n}-1)}C_{0}^{(1)}\subset\frac{1}{q_{n}q_{n}'}\widetilde{G}_{1,\ell}+\Gamma_{n}^{(1)}+\left(\frac{i_{1}}{q_{n}},\frac{i_{2}}{q_{n}'}\right).$ 

By equations (\ref{a1}) and (\ref{a2}), we have
\[
\begin{array}{ccc}
k(m_{n}-1)(\alpha_{n+1},\alpha_{n+1}') & = & k\left(\frac{r_{n}}{q_{n}}+\frac{1}{q_{n}q_{n}'},\frac{r_{n}'}{q_{n}'}+\frac{1}{q_{n}q_{n}'}\right)-\left(0,k\Delta_n\right)\\
 & = & (\gamma_{2},\gamma_{2}')-\left(0,k\Delta_n\right)
\end{array}
\]
for some $(\gamma_{2},\gamma_{2}')\in\Gamma_{n}^{(1)}.$ Let $A_{1}=R_{\alpha_{n+1},\alpha_{n+1}^{\prime}}^{i}C_{0}^{(1)}$
and $A_{2}=R_{\alpha_{n+1},\alpha_{n+1}^{\prime}}^{i+k(m_{n}-1)}C_{0}^{(1)}.$
Then $A_{2}=A_{1}+(\gamma_{2},\gamma_{2}')-\left(0,k\Delta_n\right).$
By Lemma \ref{lem:disj}, the rectangles $A_{1}$ and $A_{2}$ are both in $\cup\{P_{j_{1},j_{2},\varepsilon_{n}}^{(1)}:0\le j_{1}<q_{n},0\le j_2<q_{n}'\}.$
Moreover, since $\cup\{P_{j_{1},j_{2},\varepsilon_{n}}^{(1)}:0\le j_{1}<q_{n},0\le j_2<q_{n}'\}$
is invariant under addition by elements of $\Lambda_{n},$ the rectangle
$A_{1}+(\gamma_{2},\gamma_{2}')$ is also contained in $\cup\{P_{j_{1},j_{2},\varepsilon_{n}}^{(1)}:0\le j_{1}<q_{n},0\le j_2<q_{n}'\}$.
But the rectangles $A_{2}$ and $A_{1}+(\gamma_{2},\gamma_{2}')$
have the same $\theta_{1}$-factor, and in the $\theta_{2}$-factor,
$A_{2}$ is shifted by $k\Delta_n$, where $0\le k\Delta_n <q_nq_n'\Delta_n<\frac{\varepsilon_n}{8q_nq_n'}$,
which is smaller than the minimum vertical distance between different
$P_{j_{1},j_{2},\varepsilon_{n}}^{(1)}$'s. Therefore $A_{2}$ and $A_{1}+(\gamma_{2},\gamma_{2}')$
are in the same $P_{j_{1},j_{2},\varepsilon_{n}}^{(1)}.$ Every translate
of $\frac{1}{q_{n}q_{n}'}\widetilde{G}_{1,j}$ by an element of $\Lambda_{n}$
that intersects $P_{j_{1},j_{2},\varepsilon_{n}}^{(1)}$ is contained
in $P_{j_{1},j_{2},\varepsilon_{n}}^{(1)}$, and the upper and lower
boundaries of that translate of $\frac{1}{q_{n}q_{n}'}\widetilde{G}_{1,\ell}$
are contained in the upper and lower boundaries of $P_{j_{1},j_{2},\varepsilon_{n}}^{(1)}$.
Since the $\theta_{1}$-factors of $A_{2}$ and $A_{1}+(\gamma_{2},\gamma_{2}')$
are the same, it follows that $A_{2}\subset\frac{1}{q_{n}q_{n}'}\widetilde{G}_{1,\ell}+(\gamma_{1},\gamma_{1}')+(\gamma_{2},\gamma_{2}')+\left(\frac{i_{1}}{q_{n}},\frac{i_{2}}{q_{n}'}\right)$
if and only if $A_{1}+(\gamma_{2},\gamma_{2}')\subset\frac{1}{q_{n}q_{n}'}\widetilde{G}_{1,\ell}+(\gamma_{1},\gamma_{1}')+(\gamma_{2},\gamma_{2}')+\left(\frac{i_{1}}{q_{n}},\frac{i_{2}}{q_{n}'}\right).$
The latter inclusion holds by assumption. Therefore $A_{2}\subset\frac{1}{q_{n}q_{n}'}\widetilde{G}_{1,\ell}+\Gamma_{n}^{(1)}+\left(\frac{i_{1}}{q_{n}},\frac{i_{2}}{q_{n}'}\right),$
and (\ref{eq:ClaimConclusion}) follows. The second part of the claim
also holds by the above observation regarding the upper and lower
boundaries of the translates of $\frac{1}{q_{n}q_{n}'}\widetilde{G}_{1,\ell}$
by elements of $\Lambda_{n}.$ \end{proof}

By the Claim, (\ref{eq:xi_1}) holds if for most $i$'s with $0\le i<m_{n}-2,$
the $\theta_{1}$-factor of $R_{\alpha_{n+1},\alpha_{n+1}^{\prime}}^{i}C_{0}^{(1)}$
is contained in the projection of $\frac{1}{q_{n}q_{n}'}\widetilde{G}_{1,\ell}+\Lambda_{n}$
in the $\theta_{1}$-direction for some $\ell$ with $0\le \ell<\gamma.$
According to the
definitions in Section \ref{section:conj}, the projections of the translates of $\frac{1}{q_{n}q_{n}'}\widetilde{G}_{1,\ell}$
by elements of $\Lambda_{n}$ consist of $\gamma$ disjoint
subintervals of $[0,\frac{1}{q_{n}q_{n}'}],$ each of length $\frac{1-\varepsilon_{n}}{\gamma q_{n}q_{n}'}$,
and their translates under multiples of $\frac{1}{q_{n}q_{n}'}.$
Thus the projections obtained consist of approximately $\gamma q_{n}q_{n}'$
intervals of length slightly less than $\frac{1}{\gamma q_{n}q_{n}'.}$
The intervals in the $\theta_{1}$-factors of $R_{\alpha_{n+1},\alpha_{n+1}^{\prime}}^{i}C_{0}^{(1)}$
are of length $\frac{q_{n}q_{n}'}{q_{n+1}}$, and mod $\frac{1}{q_{n}q_{n}'}$
their left endpoints form an arithmetic progression with increment
$\frac{1}{q_{n+1}}.$ Thus mod $\frac{1}{q_{n}q_{n}'}$ the collection
of $m_{n}-1$ such left endpoints are evenly spaced throughout the
interval $[0,\frac{1}{q_{n}q_{n}'}].$ Since the ratio of $\frac{1-\varepsilon_{n}}{\gamma q_{n}q_{n}'}$
to $\frac{1}{q_{n+1}}$ is very large, for most $i$'s, the $\theta_{1}$-factor of $R_{\alpha_{n+1},\alpha_{n+1}^{\prime}}^{i}C_{0}^{(1)}$
is contained in the projection of $\frac{1}{q_{n}q_{n}'}\widetilde{G}_{1,\ell}+\Lambda_{n}$
in the $\theta_{1}$-direction for some $\ell$ with $0\le \ell<\gamma.$
This completes the proof of (\ref{eq:xi_1}).

The proof of (\ref{eq:xi_2}) is similar, but slightly more delicate.
In the hypothesis of the analogous Claim, we have to replace $\frac{1}{q_{n}q_{n}'}\widetilde{G}_{1,\ell}$
by $\frac{1}{q_{n}q_{n}'}\widetilde{G}_{2,\ell,\varepsilon_{n}}$ where
the set $\frac{1}{q_{n}q_{n}'}\widetilde{G}_{2,\ell,\varepsilon_{n}}$ is
a subset of $\frac{1}{q_{n}q_{n}'}\widetilde{G}_{2,\ell}$ obtained by
decreasing the $\theta_{1}$-interval by a factor of $\varepsilon_{n}$
at the beginning and at the end. The analogous rectangles $A_{2}$
and $A_{1}+(\gamma_{1},\gamma_{1}')$ in the proof of the Claim are
shifted slightly relative to each other in both the vertical and the
horizontal directions, but this shift is so small that we can again
conclude that $A_{2}$ and $A_{1}+(\gamma_{1},\gamma_{1}')$ are in
the same $P_{j_{1},j_{2},\varepsilon_{n}}^{(2)}.$ 
\end{pr}

\begin{lem} \label{lem:eta}
We have
\begin{equation*}
\eta_n \rightarrow \varepsilon \text{ as } n\rightarrow \infty. 
\end{equation*}
\end{lem}

\begin{pr}
By the observations preceding Lemma \ref{lem:xi}, it suffices to show that
for most $i$'s with $0\le i<m_{n}-2,$ there exist $j$ and $i_{1},i_{2},$
depending only on $i$ such that we have simultaneously 
\[
\tilde{B}_{i,1}^{(n)}\subset\frac{1}{q_{n}q_{n}'}\widetilde{G}_{1,\ell}+\Gamma_{n}^{(1)}+\left(\frac{i_{1}}{q_{n}},\frac{i_{2}}{q_{n}'}\right)
\]
 and 
\[
\tilde{B}_{i,2}^{(n)}\subset\frac{1}{q_{n}q_{n}'}\widetilde{G}_{2,\ell}+\Gamma_{n}^{(2)}+\left(\frac{i_{1}}{q_{n}},\frac{i_{2}}{q_{n}'}\right).
\]
If we let $\widetilde{G}_{1,\ell,\varepsilon_{n}}$ be the same as $\widetilde{G}_{1,\ell}$
except the $\theta_{1}$-interval is reduced by a factor of $\varepsilon_{n}$
at both ends, and $\widetilde{G}_{2,\ell,\varepsilon_{n}}$ is the same
as $\widetilde{G}_{2,\ell}$ except that the $\theta_{2}$-interval is
reduced by a factor of $\varepsilon_{n}$ at both ends, then by the
proof of Lemma \ref{lem:xi}, it suffices to prove that for most $i$'s there
exist $\ell$ and $i_{1},i_{2}$ such that 
\begin{equation}
R_{\alpha_{n+1},\alpha_{n+1}^{\prime}}^{i}C_{0}^{(1)}\subset\frac{1}{q_{n}q_{n}'}\widetilde{G}_{1,\ell,\varepsilon_{n}}+\left(\frac{i_{1}}{q_{n}},\frac{i_{2}}{q_{n}'}\right)\label{eq:EpsilonInclusion1}
\end{equation}
 and 
\begin{equation}
R_{\alpha_{n+1},\alpha_{n+1}^{\prime}}^{i}C_{0}^{(2)}\subset\frac{1}{q_{n}q_{n}'}\widetilde{G}_{2,\ell,\varepsilon_{n}}+\left(\frac{i_{1}}{q_{n}},\frac{i_{2}}{q_{n}'}\right).\label{eq:EpsilonInclusion2}
\end{equation}
Note that the lower left corner points $c^{(1)}$ and $c^{(2)}$ of
$C_{0}^{(1)}$ and $C_{0}^{(2)}$ are $(\frac{\tilde{\varepsilon}_n}{2q_nq_n'},\frac{3\varepsilon_{n}}{2q_{n}q_{n}'})$
and $(\frac{3\varepsilon_{n}}{2q_{n}q_{n}'},\frac{\tilde{\varepsilon}_n}{2q_nq_n'})$, respectively.
Thus $\theta_{1}=\frac{\tilde{\varepsilon}_n}{2q_nq_n'}$ on the left boundary of $C_{0}^{(1)},$
and $\theta_{2}=\frac{\tilde{\varepsilon}_n}{2q_nq_n'}$ on the lower boundary of $C_{0}^{(2)}.$
Modulo $\frac{1}{q_{n}}$ the $\theta_{1}$-coordinate of $R_{\alpha_{n+1},\alpha_{n+1}^{\prime}}^{i}c^{(1)}$
is equal to $\frac{\tilde{\varepsilon}_n}{2q_nq_n'}+\frac{i}{q_{n+1}},$ and modulo $\frac{1}{q_{n}'}$
the $\theta_{2}$-coordinate of $R_{\alpha_{n+1},\alpha_{n+1}^{\prime}}^{i}c^{(2)}$
is equal to $\frac{\tilde{\varepsilon}_n}{2q_nq_n'}+i\left(\frac{1}{\overline{q}_{n+1}'}+D_{n}\right).$
For $0\le i<m_{n}-1,$ the difference between these two coordinates
is less than $m_{n}(\frac{1}{q_{n+1}}-\frac{1}{\overline{q}_{n+1}'}-D_{n})<\frac{1}{4n^{d-1}(q_{n}q_{n}')^{d+1}}$,
which is much smaller than the $\theta_{1}$-width of $\frac{1}{q_{n}q_{n}'}\widetilde{G}_{1,\ell,\varepsilon_{n}}$
or the $\theta_{2}$-height of $\frac{1}{q_{n}q_{n}'}\widetilde{G}_{2,\ell,\varepsilon_{n}},$
which is of the order $\frac{1}{n^{d-2}(q_{n}q_{n}')^{d-1}}.$ This
implies that for most $i$'s, we have $R_{\alpha_{n+1},\alpha_{n+1}^{\prime}}^{i}C_{0}^{(1)}\subset\frac{1}{q_{n}q_{n}'}\widetilde{G}_{1,\ell,\varepsilon_{n}}+\Lambda_{n}$
and $R_{\alpha_{n+1},\alpha_{n+1}^{\prime}}^{i}C_{0}^{(2)}\subset\frac{1}{q_{n}q_{n}'}\widetilde{G}_{2,\ell,\varepsilon_{n}}+\Lambda_{n}$
for the same $\ell.$ The points $R_{\alpha_{n+1},\alpha_{n+1}^{\prime}}^{i}c^{(1)}$
and $R_{\alpha_{n+1},\alpha_{n+1}^{\prime}}^{i}c^{(2)}$ are close
together compared to $\frac{1}{q_{n}q_{n}'}.$ Thus for most $i$'s
they are in the same $\left[0,\frac{1}{q_{n}q_{n}'}\right]\times\left[0,\frac{1}{q_{n}q_{n}'}\right]$+$\left(\frac{\tau_{1}}{q_{n}q_{n}'},\frac{\tau_{2}}{q_{n}q_{n}'}\right)$
square. Moreover, since $\frac{m_{n}-1}{q_{n+1}}=\frac{1}{q_{n}q_{n}'},$
except for possibly a few $i's$ near $m_{n}-1,$ both points $R_{\alpha_{n+1},\alpha_{n+1}^{\prime}}^{i}c^{(1)}$
and $R_{\alpha_{n+1},\alpha_{n+1}^{\prime}}^{i}c^{(2)}$ are in the same 
$\left[0,\frac{1}{q_{n}q_{n}'}\right]\times\left[0,\frac{1}{q_{n}q_{n}'}\right]$
square modulo $\left(\frac{1}{q_{n}},\frac{1}{q_{n}'}\right)$. Thus,
for most $i$'s there exist $\ell,i_{1},i_{2}$ such that (\ref{eq:EpsilonInclusion1})
and (\ref{eq:EpsilonInclusion2}) hold simultaneously.
\end{pr}

\subsection{Speed of approximation}
In this section we want to prove that $T$ admits a good linked approximation of type $(h,h+1)$:
\begin{prop} \label{prop:h+1}
The constructed diffeomorphism $T$ admits a good linked approximation of type $(h,h+1)$ with speed of order $O(1/h^2)$.
\end{prop}

Since $T$ admits a linked approximation by Lemma \ref{lem:eta}, we have to compute the speed of the approximation in order to prove this statement. First of all, we observe 
\begin{equation} \label{eq:firstofall}
\sum_{c \in \xi_n} \mu\left(T\left(c\right) \triangle \sigma_n\left(c\right)\right) \leq \sum_{c \in \xi_n}\left(\mu\left(T\left(c\right) \triangle T_n\left(c\right)\right)+\mu\left(T_n\left(c\right) \triangle \sigma_n\left(c\right)\right)\right)
\end{equation}
recalling that $\sigma_n$ is the associated permutation satisfying $\sigma_n \left( B^{(n)}_{i,1} \right) = B^{(n)}_{i+1,1}$ for $i=0,...,m_n-3$, $\sigma_n \left( B^{(n)}_{m_n-2,1} \right) = B^{(n)}_{0,1}$ as well as $\sigma_n \left( B^{(n)}_{k,2} \right) = B^{(n)}_{k+1,2}$ for $k=0,1,...,m_n-2$, $\sigma_n \left( B^{(n)}_{m_n-1,2} \right) = B^{(n)}_{0,2}$. In the subsequent lemmas we examine each summand.
\begin{lem} \label{lem:error1}
We have
\begin{equation*} 
\sum_{c \in \xi_n}\mu\left(T_n\left(c\right) \triangle \sigma_n\left(c\right)\right) \leq \frac{6 \cdot q^2_n \cdot \left( q^{\prime}_n \right)^2}{q_{n+1} \cdot \bar{q}^{\prime}_{n+1}}.
\end{equation*}
\end{lem}

\begin{pr}
We note $\sigma_n|_{B^{(n)}_{i,1}} = T_n|_{{B^{(n)}_{i,1}}}$ for $i=0,...,m_n-3$ and $\sigma_n\left(B^{(n)}_{m_n-2,1}\right) = B^{(n)}_{0,1}$ as well as $\sigma_n|_{B^{(n)}_{i,2}} = T_n|_{{B^{(n)}_{i,2}}}$ for $i=0,...,m_n-2$, $\sigma_n\left(B^{(n)}_{m_n-1,2}\right) = B^{(n)}_{0,2}$. \\
To estimate the expression $\sum_{c \in \xi_{n,1}}\mu\left(T_n\left(c\right) \triangle \sigma_n\left(c\right)\right)$ we recall equations (\ref{a1}) as well as (\ref{a2}). The sets $C^{(1)}_{k_1}$ are defined in such a way that $R^{m_n - 1}_{\a_{n+1}, \a^{\prime}_{n+1}}\left(C^{(1)}_{k_1}\right) = C^{(1)}_{k_1+1}$ for $0 \leq k_1 < q_n q^{\prime}_n-1$ (see Remark \ref{rem:mapTower1}). Thus,
$$ \mu\left(T^{m_n-1}_n\left(B^{(n)}_{0,1}\right) \triangle B^{(n)}_{0,1}\right) = \mu \left( R^{m_n - 1}_{\a_{n+1}, \a^{\prime}_{n+1}} \left( C^{(1)}_{q_n q^{\prime}_n - 1} \right) \triangle C^{(1)}_0 \right). $$
Additionally, we observe that 
\begin{align*}
& q_n q^{\prime}_n \cdot \left(m_n -1\right) \cdot \alpha_{n+1} \equiv 0 \mod 1 \\
& q_n q^{\prime}_n \cdot \left(m_n -1\right) \cdot \alpha^{\prime}_{n+1} \equiv -\frac{q_n q^{\prime}_n}{\bar{q}^{\prime}_{n+1}} + q_n q^{\prime}_n \cdot \left( m_n-1\right) \cdot D_n \mod 1. 
\end{align*}
Hence, $R^{m_n - 1}_{\a_{n+1}, \a^{\prime}_{n+1}} \left( C^{(1)}_{q_n q^{\prime}_n - 1} \right)$ and $ C^{(1)}_0 $ differ in the deviation of at most $\frac{q_n q^{\prime}_n }{\bar{q}^{\prime}_{n+1}}$ from $0$ in the $\theta_2$- direction. This yields the following measure difference:
\begin{equation*}
\mu\left(T^{m_n-1}_n\left(B^{(n)}_{0,1}\right) \triangle B^{(n)}_{0,1}\right) \leq  2 \cdot \frac{q_n q^{\prime}_n}{\bar{q}^{\prime}_{n+1}} \cdot \frac{q_n q^{\prime}_n}{q_{n+1}} \cdot \left(1-2\varepsilon_n\right)^{d-2} \leq \frac{2 \cdot q^2_n \cdot \left( q^{\prime}_n \right)^2}{q_{n+1} \cdot \bar{q}^{\prime}_{n+1}}.
\end{equation*}
Analogously the sets $C^{(2)}_{k_2}$ are defined in such a way that $R^{m_n}_{\a_{n+1}, \a^{\prime}_{n+1}}\left(C^{(2)}_{k_2}\right) = C^{(2)}_{k_2+1}$ for $0 \leq k_2 < q_n q^{\prime}_n-1$ (see Remark \ref{rem:mapTower2}). So $T_n\left(B^{(n)}_{m_n-1,2}\right)$ and $B^{(n)}_{0,2}$ differ in the deviation of $\frac{q_nq^{\prime}_n}{q_{n+1}}$ from $0$ in the $\theta_1$-direction and the deviation of $q_nq^{\prime}_n \cdot m_n \cdot D_n$ from $0$ in the $\theta_2$-direction by equations (\ref{a3}) and (\ref{a4}):
\begin{equation*}
\mu\left(T^{m_n}_n\left(B^{(n)}_{0,2}\right) \triangle B^{(n)}_{0,2}\right) \leq 2 \cdot \left( \frac{q_n q^{\prime}_n}{q_{n+1}} \cdot \frac{q_n q^{\prime}_n}{\bar{q}^{\prime}_{n+1}} + \frac{1-4\varepsilon_n}{2q_n q^{\prime}_n} \cdot q_n q^{\prime}_n \cdot m_n \cdot D_n \right) \cdot \left(1-2 \varepsilon_n \right)^{d-2} \leq \frac{4 \cdot q^2_n \cdot \left( q^{\prime}_n \right)^2}{q_{n+1} \cdot \bar{q}^{\prime}_{n+1}}
\end{equation*}
using inequality (\ref{eq:D}). This yields
\begin{equation*} 
\sum_{c \in \xi_n}\mu\left(T_n\left(c\right) \triangle \sigma_n\left(c\right)\right) \leq \frac{6 \cdot q^2_n \cdot \left( q^{\prime}_n \right)^2}{q_{n+1} \cdot \bar{q}^{\prime}_{n+1}}.
\end{equation*}
\end{pr}

\begin{rem}\label{rem:bigO}
We note that this error term is of order $O\left(\frac{1}{\left(m_n - 1 \right) m_n} \right)$ but not of order $o\left(\frac{1}{\left(m_n-1\right)m_n}\right)$. So we cannot hope for an excellent approximation of type $\left(h,h+1\right)$.
\end{rem}

In the next step we consider $\sum_{c \in \xi_n}\mu\left(T\left(c\right) \triangle T_n\left(c\right)\right)$: 

\begin{lem} \label{lem:error3}
We have 
\begin{equation*} 
\sum_{c \in \xi_n}\mu\left(T\left(c\right) \triangle T_{n}\left(c\right)\right) \leq 2 \cdot \varepsilon_{n+1}.
\end{equation*}
\end{lem}

\begin{pr}
In requirement (\ref{eq:smalTranslCond}) of Lemma \ref{lem:conv} we choose $q_{\ell+1}$, $\ell\ge1$, sufficiently large so that 
\begin{equation}
\mu\left(R_{\alpha_{\ell+1}-\alpha_{\ell},\alpha_{\ell+1}'-\alpha_{\ell}'}\circ\left(R_{\alpha_{\ell},\alpha_{\ell}'}\circ H_{\ell}(c)\right)\triangle \left(R_{\alpha_{\ell},\alpha_{\ell}'}\circ H_{\ell}(c)\right)\right)\le\varepsilon_{\ell} \cdot \mu(c),\label{eq:small translation}
\end{equation}
for $c\in\xi_{n}$, for $1\le n\le\ell-1$. This is possible because
$R_{\alpha_{\ell},\alpha_{\ell}'}\circ H_{\ell}$ and $\varepsilon_{\ell}$
are determined by parameters with index at most $\ell$, $\xi_{n}$
is determined by parameters of index at most $n+1$, and for any $L^{1}(\mu)$-function $f$, the map from $\mathbb{T}^{2}$ to $L^{1}(\mu)$ defined
by $(\alpha,\alpha')\mapsto f\circ R_{\alpha,\alpha'}$ is continuous.\\

Then for $n\in\mathbb{N}$ and $k\ge0$ we compute for every $c\in \xi_n$ 
\begin{align*}
& \mu(T_{n+k+1}(c)\triangle T_{n+k}(c)) \\
= & \mu\left(H^{-1}_{n+k+1} \circ R_{\alpha_{n+k+2},\alpha_{n+k+2}'}\circ H_{n+k+1}(c)\triangle H^{-1}_{n+k} \circ R_{\alpha_{n+k+1},\alpha_{n+k+1}'}\circ H_{n+k}(c)\right) \\
= & \mu\left(H^{-1}_{n+k+1} \circ R_{\alpha_{n+k+2},\alpha_{n+k+2}'}\circ H_{n+k+1}(c)\triangle H^{-1}_{n+k} \circ h^{-1}_{n+k+1} \circ R_{\alpha_{n+k+1},\alpha_{n+k+1}'}\circ h_{n+k+1} \circ H_{n+k}(c)\right) \\
= & \mu\left(R_{\alpha_{n+k+2},\alpha_{n+k+2}'}\circ H_{n+k+1}(c)\triangle R_{\alpha_{n+k+1},\alpha_{n+k+1}'}\circ H_{n+k+1}(c)\right) \\
= & \mu\left(R_{\alpha_{n+k+2}-\alpha_{n+k+1},\alpha_{n+k+2}'-\alpha_{n+k+1}'}\circ\left(R_{\alpha_{n+k+1},\alpha_{n+k+1}'}\circ H_{n+k+1}(c)\right)\triangle R_{\alpha_{n+k+1},\alpha_{n+k+1}'}\circ H_{n+k+1}(c)\right) \\
\le & \varepsilon_{n+k+1} \cdot \mu(c).
\end{align*}
Thus
\begin{align*}
\sum_{c\in\xi_{n}}\mu(T(c)\triangle T_{n}(c)) & \le \sum_{c\in\xi_{n}}\sum_{k\ge0}\mu(T_{n+k+1}(c)\triangle T_{n+k}(c)) \le \sum_{c\in\xi_{n}}\sum_{k\ge0}\varepsilon_{n+k+1}\mu(c) \le \sum_{k\ge0}\varepsilon_{n+k+1} \\
& \le 2\varepsilon_{n+1}.
\end{align*}
\end{pr}

\begin{proof}[Proof of Proposition \ref{prop:h+1}]
Using equation (\ref{eq:firstofall}) and the precedent two lemmas we conclude
\begin{equation*}
\sum_{c \in \xi_n} \mu\left(T\left(c\right) \triangle \sigma_n\left(c\right)\right) \leq \frac{6 \cdot q^2_n \cdot \left( q^{\prime}_n \right)^2}{q_{n+1} \cdot \bar{q}^{\prime}_{n+1}}+ 2 \cdot \varepsilon_{n+1}.
\end{equation*}

We already observed in the Remark \ref{rem:bigO} that the first term on the right hand side of this inequality is of order $O\left(\frac{1}{\left(m_n - 1 \right) m_n} \right)$.
We recall our choice $\varepsilon_{n+1}=\frac{2}{(n+1) \cdot q_{n+1}q^{\prime}_{n+1}}$ in the construction of the conjugation map (see equation (\ref{eq:eps})).  We have
$\varepsilon_{n+1}$ of order $o\left(\frac{1}{\left(m_n - 1 \right) m_n} \right)$ because 

\begin{align*}
\frac{ \frac{2}{(n+1) \cdot q_{n+1}q^{\prime}_{n+1}}}{\frac{1}{\left(m_n-1\right)m_n}} & = \frac{q_{n+1} \cdot \bar{q}^{\prime}_{n+1}}{q^2_n \cdot \left(q^{\prime}_n\right)^2} \cdot \left( \frac{ 2}{(n+1) \cdot q_{n+1}\cdot q^{\prime}_{n+1}}\right) \\
& = \frac{2}{(n+1) \cdot q^2_n \cdot \left(q^{\prime}_n\right)^2}  \cdot \frac{\bar{q}^{\prime}_{n+1}}{q^{\prime}_{n+1}} \\
& < \frac{1}{2(n+1)\cdot q^2_n \cdot \left(q^{\prime}_n\right)^2},  \end{align*}
which goes to 0 as $n\to\infty$. In the last step of the above calculation we are exploiting condition (\ref{eq:largeprime}). Therefore $\sum_{c \in \xi_n} \mu\left(T\left(c\right) \triangle \sigma_n\left(c\right)\right)$ is of order $O\left(\frac{1}{\left(m_n - 1 \right) m_n} \right)$, or, equivalently, of order $O(1/h^2).$
\end{proof}

\subsection{Proof of Proposition \ref{prop:red}} \label{subsection:appl}
By Section \ref{section:conv} our sequence of volume-preserving diffeomorphisms $T_n$ converges in the $C^{\infty}$-topology and the limit diffeomorphism $T$ has topological entropy zero. In Proposition \ref{prop:h+1} we proved that $T$ admits a good linked approximation of type $(h,h+1)$ with speed $O\left(1/h^2\right)$. Hence, we can apply our criterion in Proposition \ref{prop:crit} and deduce that $T$ has zero entropy and $T \times T$ is loosely Bernoulli.

\subsection{Additional properties}
In fact, we can describe the approximation of the constructed diffeomorphism $T$ in an even more detailed way:
\begin{prop}
There exists $c>0$ such that the limit diffeomorphism $T$ admits an approximation with towers $\tau_n$ consisting of $N_n> c \cdot m_n$ consecutive passing of columns with height $m_n - 1$ and $N_n$ consecutive passing of columns with height $m_n $. 
\end{prop}

This tower structure is depicted in Figure \ref{fig:towerstructure}. 

In each set $C^{(1)}_{k(i,j)}$ we introduce the sets
\begin{equation*} \begin{split}
& C^{(1)}_{k(i,j), t} \\
= & \left( \frac{a_{n,1}(i,j)}{q_n} + \frac{2i+\tilde{\varepsilon}_n}{2q_n q^{\prime}_n}, \frac{a_{n,1}(i,j)}{q_n} +  \frac{2i+\tilde{\varepsilon}_n}{2q_n q^{\prime}_n}+ \frac{q_n q^{\prime}_n}{q_{n+1}} \right) \\
& \times \left[\frac{a^{\prime}_{n,1}(i,j)}{q^{\prime}_n}+ \frac{2j+3\varepsilon_n}{2q_n q^{\prime}_n}-k \cdot \Delta_n+ \frac{t \cdot q_nq^{\prime}_n+\varepsilon_n}{\bar{q}^{\prime}_{n+1}} , \frac{a^{\prime}_{n,1}(i,j)}{q^{\prime}_n}+ \frac{2j+3\varepsilon_n}{2q_n q^{\prime}_n}-k \cdot \Delta_n+ \frac{(t+1) \cdot q_nq^{\prime}_n-\varepsilon_n}{\bar{q}^{\prime}_{n+1}} \right] \\
& \times \left[\varepsilon_n, 1-\varepsilon_n\right]^{d-2}
\end{split} \end{equation*}
and
\begin{equation*} \begin{split}
& \tilde{C}^{(1)}_{k(i,j), t} \\
= & \left( \frac{a_{n,1}(i,j)}{q_n} + \frac{2i+\tilde{\varepsilon}_n}{2q_n q^{\prime}_n},\frac{a_{n,1}(i,j)}{q_n} +  \frac{2i+\tilde{\varepsilon}_n}{2q_n q^{\prime}_n} +\frac{q_n q^{\prime}_n}{q_{n+1}} \right) \\
& \times \Bigg[\frac{a^{\prime}_{n,1}(i,j)}{q^{\prime}_n}+ \frac{2j+3\varepsilon_n}{2q_n q^{\prime}_n}-k \cdot \Delta_n+ \frac{t \cdot q_nq^{\prime}_n}{\bar{q}^{\prime}_{n+1}}, \frac{a^{\prime}_{n,1}(i,j)}{q^{\prime}_n}+ \frac{2j+3\varepsilon_n}{2q_n q^{\prime}_n}-k \cdot \Delta_n+ \frac{(t+1) \cdot q_nq^{\prime}_n}{\bar{q}^{\prime}_{n+1}} \Bigg) \\
& \times \left[\varepsilon_n, 1-\varepsilon_n\right]^{d-2}
\end{split} \end{equation*}
for $t \in \N$, $0 \leq t < \lfloor\left(1-4\varepsilon_n \right) \cdot \frac{\bar{q}^{\prime}_{n+1}}{2q^2_n \left(q^{\prime}_n\right)^2}\rfloor \eqqcolon t^{\ast}_n$. By the same observation as in Remark \ref{rem:mapTower1} we have $R^{m_n-1}_{\a_{n+1}, \a^{\prime}_{n+1}}\left(C^{(1)}_{k,t}\right) =C^{(1)}_{k+1,t}$ for every $0 \leq k < q_n q^{\prime}_n-1$. From the relations on the rotation numbers in equations (\ref{a1}) and (\ref{a2}), as well as $m^2_n \cdot D_n< \frac{\varepsilon_n}{2\bar{q}^{\prime}_{n+1}}$ by condition (\ref{eq:D}), we have 
\begin{equation*}
R^{t \cdot q_n q^{\prime}_n \cdot \left(m_n - 1\right)}_{\a_{n+1}, \a^{\prime}_{n+1}}\left( C^{(1)}_{0,t^{\ast}_n -1} \right) \subset \tilde{C}^{(1)}_{0,t^{\ast}_n - t -1} 
\end{equation*}
for every $0 \leq t < \lfloor \left(1-4\varepsilon_n \right) \cdot \frac{\bar{q}^{\prime}_{n+1}}{2q^2_n \left(q^{\prime}_n\right)^2} \rfloor$. So for each value of $k$ and $t$ the sets $R^{i}_{\alpha_{n+1}, \alpha^{\prime}_{n+1}} \left( C^{(1)}_{k,t} \right)$, $0 \leq i < m_n - 1$, correspond to a column of the first tower. Hence, there are about $\left(1-4\varepsilon_n \right) \cdot \frac{\bar{q}^{\prime}_{n+1}}{2q_n q^{\prime}_n} = \left(1-4\varepsilon_n \right) \cdot \frac{m_n}{2}$ many columns.

\begin{figure}[hbtp] 
\begin{center}
\includegraphics[scale=0.7]{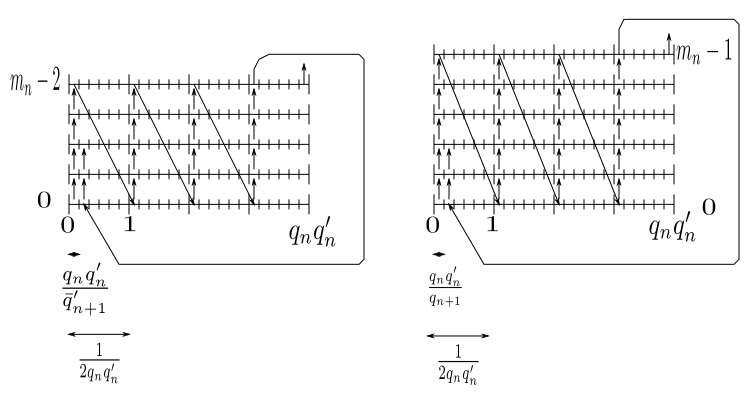}
\caption{General idea of the two towers and their levels. In the left tower each of the $q_nq^{\prime}_n$ sections of length $\frac{1}{2q_nq^{\prime}_n}$ in the base corresponds to a set $C^{(1)}_{k(i,j)}$. The sections of length $\frac{q_nq^{\prime}_n}{\bar{q}^{\prime}_{n+1}}$ contained in it correspond to the sets $C^{(1)}_{k(i,j),t}$ (except for the $\frac{\varepsilon_n}{\bar{q}^{\prime}_{n+1}}$-terms). Analogously, in the right tower each of the $q_nq^{\prime}_n$ sections of length $\frac{1}{2q_nq^{\prime}_n}$ in the base corresponds to a set $C^{(2)}_{k(i,j)}$ and the sections of length $\frac{q_nq^{\prime}_n}{\bar{q}^{\prime}_{n+1}}$ correspond to the sets $C^{(2)}_{k(i,j),t}$.}\label{fig:towerstructure}
\end{center}
\end{figure}

Analogously, in each set $C^{(2)}_{k(i,j)}$ we define the sets

\begin{equation*}\begin{split}
& C^{(2)}_{k(i,j),t} \\
= & \left[ \frac{a_{n,2}(i,j)}{q_n} + \frac{2i+3\varepsilon_n}{2q_n q^{\prime}_n}+\frac{k}{q_{n+1}}+\frac{t \cdot q_n q^{\prime}_n}{q_{n+1}}, \frac{a_{n,2}(i,j)}{q_n} +  \frac{2i+3\varepsilon_n}{q_n q^{\prime}_n}+\frac{k}{q_{n+1}}+\frac{(t+1) \cdot q_n q^{\prime}_n}{q_{n+1}}\right] \\
& \times \Bigg[\frac{a^{\prime}_{n,2}(i,j)}{q^{\prime}_n}+ \frac{2j+\tilde{\varepsilon}_n}{2q_n q^{\prime}_n}+k \cdot m_n \cdot D_n+ \frac{4\varepsilon_n}{\bar{q}^{\prime}_{n+1}}, \frac{a^{\prime}_{n,2}(i,j)}{q^{\prime}_n}+ \frac{2j+\tilde{\varepsilon}_n}{2q_n q^{\prime}_n}+k \cdot m_n\cdot D_n+ \frac{q_nq^{\prime}_n - 4 \varepsilon_n}{\bar{q}^{\prime}_{n+1}}\Bigg]  \\
& \times \left[\varepsilon_n, 1-\varepsilon_n\right]^{d-2}
\end{split}\end{equation*}

and 
\begin{equation*} \begin{split}
& \tilde{C}^{(2)}_{k(i,j),t} \\
= & \Bigg[ \frac{a_{n,2}(i,j)}{q_n} + \frac{2i+3\varepsilon_n}{2q_n q^{\prime}_n}+\frac{k}{q_{n+1}}+\frac{t \cdot q_n q^{\prime}_n}{q_{n+1}}, \frac{a_n(i,2j)}{q_n} +  \frac{2i+3\varepsilon_n}{2q_n q^{\prime}_n}+\frac{k}{q_{n+1}}+\frac{(t+1) \cdot q_n q^{\prime}_n}{q_{n+1}}\Bigg) \\
& \times \Bigg[\frac{a^{\prime}_{n,2}(i,j)}{q^{\prime}_n}+ \frac{2j+\tilde{\varepsilon}_n}{2q_n q^{\prime}_n}+ k \cdot m_n \cdot D_n + \frac{2\varepsilon_n}{\bar{q}^{\prime}_{n+1}}, \frac{a^{\prime}_{n,2}(i,j)}{q^{\prime}_n}+ \frac{2j+\tilde{\varepsilon}_n}{2q_n q^{\prime}_n} +k \cdot m_n \cdot D_n+ \frac{q_nq^{\prime}_n - 2\varepsilon_n}{\bar{q}^{\prime}_{n+1}}\Bigg]  \\
& \times \left[\varepsilon_n, 1-\varepsilon_n\right]^{d-2}
\end{split} \end{equation*}
for $t \in \N$, $0 \leq t < \lfloor \left(1-4\varepsilon_n \right) \cdot \frac{q_{n+1}}{2q^2_n \left(q^{\prime}_n\right)^2}\rfloor$. As in Remark \ref{rem:mapTower2} we have  $R^{m_n}_{\a_{n+1}, \a^{\prime}_{n+1}}\left(C^{(2)}_{k,t}\right) =C^{(2)}_{k+1,t}$ for every $0 \leq k < q_n q^{\prime}_n-1$. From the relations on the rotation numbers in equations (\ref{a3}) and (\ref{a4}), as well as $m^2_n D_n < \frac{\varepsilon_n}{2 \bar{q}^{\prime}_{n+1}}$ by condition (\ref{eq:D}), we also have
\begin{equation*}
R^{t \cdot q_n q^{\prime}_n \cdot m_n}_{\a_{n+1}, \a^{\prime}_{n+1}}\left( C^{(2)}_{0,0} \right) \subset \tilde{C}^{(2)}_{0,t} 
\end{equation*}
for every $0 \leq t < \lfloor \lfloor \left(1-4\varepsilon_n \right) \cdot \frac{q_{n+1}}{2q^2_n \left(q^{\prime}_n\right)^2}\rfloor$. So for each value of $k$ and $t$ the sets $R^{i}_{\alpha_{n+1}, \alpha^{\prime}_{n+1}} \left( C^{(2)}_{k,t} \right)$, $0 \leq i < m_n$, correspond to a column of the second tower. Hence, there are about $\left(1-4\varepsilon_n \right) \cdot \frac{q_{n+1}}{2q_n q^{\prime}_n}  = \left(1-4\varepsilon_n \right) \cdot \frac{m_n-1}{2}$ many columns.

\section{Proof of genericity} \label{sec:gen}
In this final section we prove that the set of smooth diffeomorphisms with topological entropy zero and loosely Bernoulli Cartesian square is a residual subset of $$\mathcal{B}=\overline{\left\{ g \circ S_{\alpha, \beta} \circ g^{-1} : g \in \text{Diff}^{\infty}\left(M, \nu\right), \ (\a, \b ) \in \T^2\right\}}^{C^{\infty}}.$$  We present a detailed proof of the statement in the special case of $M_0=\mathbb{T}^2 \times [0,1]^{d-2}$ in Subsection \ref{subsec:special case} and then describe its generalization to the general case in Subsection \ref{subsec:general}. 

\subsection{The case $\mathbb{T}^2 \times [0,1]^{d-2}$} \label{subsec:special case}
First of all, we will use Proposition \ref{prop:red} to prove the denseness of the set of diffeomorphisms with topological entropy zero and loosely Bernoulli Cartesian square in $$\mathcal{B}_0=\overline{\left\{ g \circ R_{\alpha, \beta} \circ g^{-1} : g \in \text{Diff}^{\infty}\left(M_0, \mu\right), \ (\a, \b ) \in \T^2\right\}}^{C^{\infty}}.$$
According to Proposition \ref{prop:red}, for $(A,B)\in \mathbb{T}^2$ and $k\in \mathbb{N}$, there exist a sequence $\left(T_{A,B,k,n}\right)_{n=1}^{\infty}$ in  $\mathcal{B}_0$, constructed as in Sections \ref{section:comb}-\ref{section:LB}, and $T_{A,B,k}\in$ $\mathcal{B}_0$ such that $\lim_{n\to\infty} d_{\infty}(T_{A,B,k,n},T_{A,B,k})=0$ and $d_{\infty}(T_{A,B,k},R_{A,B})<1/k.$ Each $T_{A,B,k}$ has topological entropy zero and loosely Bernoulli Cartesian square. We will fix a choice of such a sequence $(T_{A,B,k,n})_{n=1}^{\infty}$ for every $(A,B)\in\mathbb{T}^2$ and $k\in\mathbb{N}$ throughout Section \ref{sec:gen}.  
\begin{lem} \label{lem:dense}
The set $\mathcal{D}_0:=\{g\circ T_{A,B,k}\circ g^{-1}:(A,B)\in\mathbb{T}^2, k\in\mathbb{N}, g\in \text{Diff}^{\infty}(M_0,\mu)\}$ is a dense subset of $\mathcal{B}_0$ with respect to the $C^{\infty}$ topology, and it consists of diffeomorphisms with topological entropy zero and loosely Bernoulli Cartesian square. 
\end{lem}
\begin{pr}
\noindent
From the compactness of $M_0$ and repeated applications of the chain rule, we see that for each $g\in\text{Diff}^{\infty}(M_0,\mu)$, the map from $\text{Diff}^{\infty}(M_0,\mu)$ to itself given by $\varphi \mapsto g \circ \varphi \circ g^{-1}$ is continuous with respect to the metric $d_{\infty}$. Thus
$\lim_{n\to\infty} d_{\infty}(gT_{A,B,k,n}g^{-1},gT_{A,B,k}g^{-1})=0$ for all $(A,B)\in\mathbb{T}^2, k\in\mathbb{N}, g\in\text{Diff}^{\infty}(M_0,\mu).$ By construction, $gT_{A,B,k,n}g^{-1}\in\mathcal{B}_0.$ Therefore $gT_{A,B,k}g^{-1}\in\mathcal{B}_0$ and $\mathcal{D}_0\subset\mathcal{B}_0.$ Since
$\lim_{k\to\infty}d_{\infty}(T_{A,B,k},R_{A,B})=0,$ we have $\lim_{k\to\infty}d_{\infty}(gT_{A,B,k}g^{-1},gR_{A,B}g^{-1})=0.$ Thus $gR_{A,B}g^{-1}\in\overline{\mathcal{D}}_0$ and $\mathcal{B}_0=\overline{\mathcal{D}}_0.$

Since $T_{A,B,k}$ has topological entropy zero and loosely Bernoulli Cartesian square, this is also true for each $g\circ T_{A,B,k}\circ g^{-1}$, because these properties are invariant under conjugation by a measure-preserving homeomorphism.
\end{pr}
Before we prove the genericity result in the case of $M_0$, we introduce the following definition.
\begin{dfn}
A measurable partial partition $\mathcal{P}$ of a probability space $(X,\mu)$ is said to $\delta$-\emph{generate} if for every measurable set $Y\subset X$ there is a set $Z$ consisting of a union of elements of $\mathcal{P}$ such that $\mu(Y\Delta Z)<\delta.$
\end{dfn}
\begin{rem} \label{rem:uniform eta}
As we showed in Lemma \ref{lem:eta}, there is a sequence of positive real numbers $(\delta_n)_{n=1}^{\infty}$ with $\delta_n \to 0$ such that the partial partition $\eta_n=\eta_{A,B,k,n}$ consisting of the levels in the union of the two towers in the construction of $T_{A,B,k,n}$ is $\delta_n$-generating. In addition, we see from the proof of Lemma \ref{lem:eta}, that we may choose $\delta_n$ independently of $A,B,k$. That is, even though the partial partitions $\eta_n$ depend on $A,B,k$, the convergence of $\eta_n$ to the partition into points as $n\to\infty$ is uniform in $A,B,k$.
\end{rem}

\begin{prop} \label{prop:residual}
The set of diffeomorphisms with topological entropy zero and loosely Bernoulli Cartesian square is a residual subset of $\mathcal{B}_0$ with respect to the $C^{\infty}$ topology.
\end{prop}

\begin{pr}
We use the same approach as in \cite[Section 7]{AK}. For this purpose, we again consider the sequences of diffeomorphisms 
$\left(T_{A,B,k,n}\right)_{n=1}^{\infty}$ chosen at the beginning of this subsection.

Let $U_{A,B,k,n}$ be the following neighbourhood of the diffeomorphism $T_{A,B,k,n}$:
\begin{equation*}
\begin{split}
    U_{A,B,k,n} \coloneqq  \Bigg{\{} & \varphi \in \text{Diff}^{\infty}\left(M_0, \mu\right): \ d_{l_{n+1}}\left(T_{A,B,k,n}, \varphi \right) < \frac{2}{l_{n+1}}, \\
    & d_0\left(T_{A,B,k,n}^{q_{n+1}q^{\prime}_{n+1}}, \varphi^{q_{n+1}q^{\prime}_{n+1}}\right) < \varepsilon_n, \sum_{c \in \eta_n} \mu\left(\varphi\left(c\right) \triangle \sigma_n\left(c\right)\right) < \frac{\mathcal{A}+1}{m_n \cdot \left( m_n -1\right)} \Bigg{\}},
    \end{split}
\end{equation*}
where the sequence $(l_n)_{n\in\mathbb{N}}$ is as in Lemma \ref{lem:convgen},  and $\mathcal{A}$ is a constant, obtained from Proposition \ref{prop:h+1},  such that $\sum_{c \in \eta_n} \mu\left(T_{A,B,k,n}\left(c\right) \triangle \sigma_n\left(c\right)\right) \le \mathcal{A}/(m_n\cdot (m_n-1))$. 
We define 
$$\Theta_{n,k}:=\bigcup_{(A,B)\in \mathbb{T}^2}\bigcup_{g\in\text{Diff}^{\infty}(M_0,\mu)}gU_{A,B,k,n}g^{-1}.$$ Since the neighbourhoods $U_{A,B,k,n}$ are open, the sets $\Theta_{n,k}$ are open as well. Then
$$\Theta:=\bigcap_{n=1}^{\infty} \bigcup_{s= n}^{\infty} \bigcup_{k=1}^{\infty}
\Theta_{s,k}$$ is a $G_{\delta}$-set as the countable intersection of open sets.

For all the sequences $\left(T_{A,B,k,n}\right)_{n \in \mathbb{N}}$, the limit diffeomorphism $T_{A,B,k}$ belongs to $U_{A,B,k,n}$  for every $n \in \mathbb{N}$ by construction (see Lemma \ref{lem:convgen}, Lemma \ref{lem:conv}, and the proof of Proposition \ref{prop:h+1}). Thus $\Theta\cap\mathcal{B}_0$ contains all the diffeomorphisms of the form $g\circ T_{A,B,k}\circ g^{-1}$. By Lemma \ref{lem:dense} the set of such diffeomorphisms forms a dense subset of  $\mathcal{B}_0$. Therefore $\Theta\cap\mathcal{B}_0$ is a dense $G_{\delta}$-subset of $\mathcal{B}_0.$

In the next step we show that $T \in \Theta$ admits a good linked approximation of type $\left(h,h+1\right)$ with speed $O \left(1/h^2\right)$. For any $T \in  \Theta$ there exist a sequence $\left(n_j\right)_{j \in \mathbb{N}}$ with $n_j \rightarrow \infty$ as $j \rightarrow \infty$, a sequence $(k_j)_{j \in \mathbb{N}}$ in $\mathbb{N}$, a sequence $(A_j,B_j)_{j \in \mathbb{N}}$ in $\mathbb{T}^2$, and a sequence of diffeomorphisms $(g_j)_{j \in \mathbb{N}}$ in $\text{Diff}^{\infty}(M_0,\mu)$ such that $T \in g_jU_{A_j,B_j,k_j,n_j}g_j^{-1}$.  By Remark \ref{rem:uniform eta}, $\eta_{n_j} \rightarrow \varepsilon$ as $j \rightarrow \infty$, where $\eta_{n_j}$ is the partial partition consisting of levels of the unions of the two towers for the diffeomorphism $T_{A_j,B_j,k_j,n_j}$. Then $T$ admits a good linked approximation of type $\left(h,h+1\right)$ of speed $O\left( \frac{1}{m_{n_j} \cdot \left(m_{n_j}-1\right)} \right)$ by the definition of the neighbourhoods $U_{A_j,B_j,k_j,n_j}$. 

Moreover, $T$ is uniformly rigid along the sequence $\left(q_{n_j+1}q^{\prime}_{n_j+1} \right)_{j \in \mathbb{N}}$ because of $T_{A_j,B_j,k_j,n_j}^{q_{n_j+1}q^{\prime}_{n_j+1} } = \text{id}$ and the closeness between $g_jT_{A_j,B_j,k_j,n_j}g^{-1}_j$ and $T$. Hence, every $T \in \Theta$ has topological entropy zero by Theorem \ref{thm:Glasner}. 

Thus, the set of diffeomorphisms in $\mathcal{B}_0$ admitting a good linked approximation of type $\left(h,h+1\right)$ with speed $O\left( 1/h^2\right)$ contains a dense $G_{\delta}$-set. Since this type of approximation implies the loosely Bernoulli property for the Cartesian square by Proposition \ref{prop:crit}, we conclude that the set of diffeomorphisms $T \in \mathcal{B}_0$ with topological entropy zero and loosely Bernoulli Cartsesian square is a residual subset in the $C^{\infty}$-topology.
\end{pr}

\subsection{The general case} \label{subsec:general}
In the general case of a smooth compact connected manifold $M$ admitting an effective $\mathbb{T}^2$ action we let $\tilde{T}_{A,B,k}$ and $\tilde{T}_{A,B,k,n}$ be the diffeomorphisms obtained from $T_{A,B,k}$ and $T_{A,B,k,n}$, respectively, as in Section \ref{subsection:First steps} with the aid of the map $G:\mathbb{T}^2 \times [0,1]^{d-2} \to M$ from Proposition \ref{prop:G}. From the discussion preceding Proposition \ref{prop:red}, we see that $T_{A,B,k,n}$ and $T_{A,B,k}$ may be chosen so that $\lim_{n\to\infty}d_{\infty}(\tilde{T}_{A,B,k,n},\tilde{T}_{A,B,k})=0$ and $d_{\infty}(\tilde{T}_{A,B,k},S_{A,B})<1/k$. Now let $\mathcal{D}=\{g\tilde{T}_{A,B,k}g^{-1}:(A,B)\in\mathbb{T}^2, k\in\mathbb{N},g\in\text{Diff}^{\infty}(M,\nu)\}.$ The same proof as for Proposition \ref{lem:dense} shows that $\mathcal{D}$ is a dense subset of $\mathcal{B}.$ As observed in Subsection \ref{subsection:First steps}, each $\tilde{T}_{A,B,k}$ is loosely Bernoulli and has topological entropy zero.

To prove the analog of Proposition \ref{prop:residual} for $(M,\nu)$
we transfer the partial partition $\eta_n=\eta_{A,B,k,n}$ from the construction of $T_{A,B,k,n}$ to a partial partition $\tilde{\eta}_n = \Meng{G(c)}{c \in \eta_n}$ on $M$. Note that if ${\eta_n}$ is $\delta$-generating, then so is $\tilde{\eta}_n$. Therefore Remark \ref{rem:uniform eta} applies to the partial partitions $\tilde{\eta}_n$ as well, and the proof of genericity then follows along the lines of the proof in the special case. \\

\noindent\emph{Acknowledgement:} The second author would like to thank the Indiana University Bloomington for hospitality at a visit in August 2016 when large parts of this paper were completed. He also acknowledges financial support by the ``Forschungsfonds'' at the Department of Mathematics, University of Hamburg.

\end{document}